\documentclass[11pt, twoside, reqno]{article}

\usepackage{amssymb}
\usepackage{amsmath}
\usepackage{mathrsfs}
\usepackage{amsthm}
\usepackage{amsfonts}
\usepackage{color}
\usepackage{latexsym}
\usepackage{txfonts}
\usepackage{indentfirst}

\allowdisplaybreaks

\def\red{\color{red}}

\def\rr{{\mathbb R}}
\def\rn{{\mathbb{R}^n}}

\def\nn{{\mathbb N}}
\def\zz{{\mathbb Z}}
\def\cc{{\mathbb C}}
\def\CG{{\mathcal G}}

\def\CS{{\mathcal S}}

\def\CM{{\mathcal M}}
\def\CA{{\mathcal M}}
\def\CX{{\mathcal X}}
\def\CY{{\mathcal Y}}

\def\CJ{{\mathcal J}}

\newcommand{\BMO}{\mathrm{BMO}}

\newcommand{\CD}{\mathcal D}
\newcommand{\CH}{\mathcal H}
\newcommand{\wD}{\wt{\mathcal D}}

\def\fz{\infty}
\def\az{\alpha}
\def\bz{\beta}
\def\dz{\delta}
\def\ez{\epsilon}
\def\kz{\kappa}
\def\thz{\theta}
\def\vz{\varphi}

\def\lf{\left}
\def\r{\right}
\def\ls{\lesssim}
\def\noz{\nonumber}
\def\wz{\widetilde}

\def\loc{{\mathrm{loc}}}

\def\XXint#1#2#3{{\setbox0=\hbox{$#1{#2#3}{\int}$ }
\vcenter{\hbox{$#2#3$ }}\kern-.6\wd0}}

\def\lz{{\lambda}}

\def\CG{{\mathcal G}}
\def\CA{{\mathcal A}}
\def\CS{{\mathcal S}}
\def\CY{\mathcal Y}

\def\gz{\gamma}

\def\RI{\mathrm I}
\newcommand{\RJ}{\mathrm J}

\DeclareMathOperator{\diam}{diam}

\DeclareMathOperator*{\esssup}{ess\ sup}

\newcommand{\wt}{\widetilde}

\newtheorem{theorem}{Theorem}[section]
\newtheorem{lemma}[theorem]{Lemma}

\newtheorem{proposition}[theorem]{Proposition}

\theoremstyle{definition}
\newtheorem{remark}[theorem]{Remark}
\newtheorem{definition}[theorem]{Definition}

\newtheorem{question}[theorem]{Question}
\renewcommand{\appendix}{\par
   \setcounter{section}{0}%
   \setcounter{subsection}{0}%
   \setcounter{subsubsection}{0}%
   \gdef\thesection{\@Alph\c@section}%
   \gdef\thesubsection{\@Alph\c@section.\@arabic\c@subsection}%
   \gdef\theHsection{\@Alph\c@section.}%
   \gdef\theHsubsection{\@Alph\c@section.\@arabic\c@subsection}%
   \csname appendixmore\endcsname
 }

\newcommand{\go}[1]{\CG_0^\eta(#1)}

\newcommand{\GO}[1]{\mathring{\CG}(#1)}
\newcommand{\GOO}[1]{\mathring{\CG}^\eta_0(#1)}

\numberwithin{equation}{section}

\begin{document}
\title{\bf\Large Wavelet Characterization of Besov and Triebel--Lizorkin Spaces on Spaces of Homogeneous
Type and Its Applications
\footnotetext{\hspace{-0.35cm} 2020 {\it Mathematics Subject Classification}. Primary 46E35;
Secondary 46E36, 46E39, 42B25, 30L99.\endgraf
{\it Key words and phrases.} space of homogeneous type, Besov space, Triebel--Lizorkin space,
almost diagonal operator, wavelet, molecule, Littlewood--Paley function.\endgraf
This project is partially supported by the National Key Research and Development Program of China
(Grant No.\ 2020YFA0712900) and the National
Natural Science Foundation of China (Grant Nos.
11971058, 12071197 and 11871100).}}
\author{Ziyi He, Fan Wang, Dachun Yang\footnote{Corresponding author,
E-mail: \texttt{dcyang@bnu.edu.cn}/{\red March 15, 2021}/Final version.} \ \ and Wen Yuan}
\date{}
\maketitle
	
\vspace{-0.7cm}

\begin{center}
\begin{minipage}{13cm}
{\small {\bf Abstract}\quad
In this article, the authors establish the wavelet characterization of Besov and Triebel--Lizorkin
spaces on a given space $(X,d,\mu)$ of homogeneous type in the sense of Coifman and Weiss. Moreover,
the authors introduce almost diagonal operators on Besov and Triebel--Lizorkin sequence spaces
on $X$, and obtain their boundedness. Using this wavelet characterization
and this boundedness of almost diagonal operators, the authors
obtain the molecular characterization of Besov and Triebel--Lizorkin spaces.
Applying this molecular characterization, the authors further establish
the Littlewood--Paley characterizations of Triebel--Lizorkin spaces on $X$. The main novelty
of this article is that all these results get rid of their dependence on the reverse doubling property
of $\mu$ and also the triangle inequality of $d$, by fully using the geometrical
property of $X$ expressed via its equipped quasi-metric $d$, dyadic reference points, dyadic cubes, and wavelets.}
\end{minipage}
\end{center}

\vspace{0.1cm}

\tableofcontents

\section{Introduction}\label{s1}

Besov and Triebel--Lizorkin spaces have a long history. In 1951, Nikol'sk\u{\i} \cite{nik51} introduced
some function spaces on the Euclidean space $\rn$, which are nowadays denoted by $B^s_{p,\fz}(\rn)$. Later,
Besov \cite{bes59,bes61} introduced the Besov space $B^s_{p,q}(\rn)$ for any given $q\in(0,\fz]$. On the other
hand, Lizorkin \cite{liz71,liz74} and Triebel \cite{tri73} introduced some new function spaces, which
are nowadays denoted by $F^s_{p,q}(\rn)$, for any given $s\in(0,\fz)$, $p\in(1,\fz)$, and $q\in(1,\fz]$. Later, Peetre
\cite{pee73,pee75,pee76} extended the ranges of $s$, $p$, and $q$ to all possible choices.
Frazier and Jawerth \cite{fj85,fj90} obtained the atomic and the molecular characterizations
of Besov and Triebel--Lizorkin spaces. Bownik \cite{bow05, bow07} introduced and developed Besov
and Triebel--Lizorkin spaces on anisotropic Euclidean spaces.
The related wavelet characterization of these spaces was given in \cite{mey90}.
For more information on Besov and Triebel--Lizorkin spaces on $\rn$, we refer the reader to the monographs
\cite{tri83,tri92,tri06,sa18}.

Now, we recall the notion of spaces of homogeneous type originally introduced by Coifman and Weiss \cite{cw71,cw77}. Let
$X$ be a non-empty set and $d$ a \emph{quasi-metric}, namely, a non-negative function on $X\times X$
satisfying the following conditions: for any $x,\ y,\ z\in X$,
\begin{enumerate}
\item $d(x,y)=0$ if and only if $x=y$;
\item $d(x,y)=d(y,x)$;
\item there exists a positive constant $A_0\in[1,\fz)$, independent of $x$, $y$, and $z$, such that
$$
d(x,z)\le A_0[d(x,y)+d(y,z)].
$$
\end{enumerate}
Then $(X,d)$ is called a \emph{quasi-metric space}.
A measure $\mu$ on $X$ is called a \emph{doubling measure} if any ball $B$ of $X$ is of finite measure,
namely, $\mu(B)\in(0,\fz)$, and $\mu$ satisfies the following \emph{doubling condition}:
there exists a positive constant $C_{(\mu)}$ such that, for any ball $B$, $\mu(2B)\le C_{(\mu)}\mu(B)$.
Here and thereafter,
for any $\tau\in(0,\fz)$ and any ball $B$, $\tau B$ denotes the ball of $X$ with the same
center as $B$ and $\tau$ times radius of $B$. Observe that the doubling condition implies
that, for any $\lz\in[1,\fz)$ and any ball $B$,
\begin{equation}\label{eq-doub}
\mu(\lz B)\le C_{(\mu)}\lz^\omega\mu(B),
\end{equation}
where $\omega:=\log_2 C_{(\mu)}$. The \emph{upper dimension $\omega_0$} of $X$ is defined by setting
\begin{align}\label{eq-updim}
\omega_0&:=\inf\{\omega\in(0,\fz):\ \textup{there exists a $C_{(\mu)}\in(0,\fz)$ such that \eqref{eq-doub} holds true}\noz\\
&\qquad\qquad\textup{for any ball $B$ and any $\lz\in(0,\fz)$}\}.
\end{align}
A triple $(X,d,\mu)$ is called a \emph{space of
homogeneous type} if $(X,d)$ is a quasi-metric space and $\mu$ a doubling measure on $X$.
If $A_0:=1$, then $(X,d,\mu)$ is called a \emph{metric measure space of
homogeneous type}, or simply, a \emph{doubling metric measure space}.
The spaces of homogeneous type have
proved a natural and general setting for the study of both function spaces and the boundedness of operators on them.

Throughout this article, we always make the following assumptions.
For any $x\in X$ and $r\in(0,\fz)$, $B(x,r)$ denotes the ball $B(x,r):=\{y\in X:\ d(y,x)<r\}$.
For any point $x\in X$, we assume that the balls $\{B(x,r)\}_{r\in(0,\fz)}$ form a basis of open neighborhoods
of $x$; assume that $\mu$ is \emph{Borel regular}
which means that open sets are measurable and every set
$A\subset X$ is contained in a Borel set $E$ satisfying that $\mu(A)=\mu(E)$. We also assume that
$\mu(B(x,r))\in (0,\fz)$ for any given $x\in X$ and $r\in(0,\fz)$. For the presentation concision, we always
assume that $(X, d,\mu)$ is nonatomic, namely, $\mu(\{x\})=0$ for any $x\in X$.

In 1977, Coifman and Weiss \cite{cw77} introduced the atomic Hardy spaces on spaces of homogeneous type,
and raised an \emph{open question} whether or not additional geometrical conditions are necessary to guarantee
the radial maximal function characterization of these atomic Hardy spaces. From then on, the real-variable theory of
function spaces on spaces of homogeneous type attracted a lot of attentions.

Indeed, the first progress was made on Ahlfors regular spaces which are special cases of spaces of homogeneous
type. Recall that a triplet $(X,d,\mu)$ is called an \emph{Ahlfors-$n$ regular space} if
there exists a constant $C\in[1,\fz)$ such that, for any $x\in X$ and
$r\in(0,\diam X)$, $C^{-1}r^n\le\mu(B(x,r))\le C r^n$. Here and thereafter, for any non-empty subset
$\Omega\subset X$, $\diam\Omega:=\sup\{d(x,y):\ x,\ y\in\Omega\}$. On an Ahlfors regular
space $(X,d,\mu)$  satisfying the additional assumption that there exist a $\thz\in(0,1]$ and a positive constant
$\wz C$ such that, for any $x,\ x',\ y\in X$,
$$
|d(x,y)-d(x',y)|\le\wz C[d(x,x')]^\thz[d(x,y)+d(x',y)]^{1-\thz},
$$
Marc\'{\i}as and Segovia \cite{ms79b} established the radial maximal function characterization of atomic
Hardy spaces introduced by Coifman and Weiss. In 1994, Han and Sawyer \cite{hs94} introduced homogeneous Besov and
Triebel--Lizorkin spaces over Ahlfors regular spaces. Later, Han and Yang \cite{hy02,hy03} introduced their inhomogeneous
counterparts. For more real-variable characterizations of function spaces on Ahlfors regular spaces, we refer the reader to
\cite{yang02,yang03a,yang04,yang05,yang05b,dh09,kyz11} and their references. In particular, Koskela et al. \cite{kyz11} considered the action of quasi-conformal mappings on Haj\l asz--Triebel--Lizorkin spaces
over Ahlfors regular spaces. Moreover, Alvarado and Mitrea \cite{am15} established a sharp
real-variable theory of Besov and Triebel--Lizorkin
spaces over Ahlfors regular spaces.
Recently, Jaiming and Negreira \cite{jn19} obtained a new Plancherel--P\^{o}lya
inequality for Besov spaces over Ahlfors regular spaces.

Ahlfors regular spaces were further generalized to RD-spaces which are also special cases of spaces of homogeneous
type and were originally introduced by Han et al. \cite{hmy08} (see also \cite{hmy06}).
Recall that an \emph{{\rm RD}-space} $(X,d,\mu)$ is a doubling metric measure space with the following additional \emph{reverse
doubling condition}: there exist constants $\wz C\in (0,1]$ and $\kz\in(0,\omega]$ such that, for any ball $B(x,r)$
with $x\in X$, $r\in(0, \diam X/2)$, and $\lz\in[1,\diam X/[2r])$, $\wz C\lz^\kz\mu(B(x,r))\le\mu(B(x,\lz r))$.
Obviously, an Ahlfors regular space is an RD-space, and a connected space of homogeneous type
is also an RD-space (see \cite{hmy08,yz11}, and also
\cite{yz11} for more equivalent characterizations of RD-spaces).
In 2008, Han et al. \cite{hmy08} established Calder\'{o}n reproducing formulae on RD-spaces and also introduced
Besov and Triebel--Lizorkin spaces on RD-spaces. With the help of Calder\'on reproducing formulae, a real-variable
theory of function spaces on RD-spaces has been rapidly developed;
see, for instance, \cite{hmy06,gly08,gly09x,yz08,yz10,zsy16}.
In particular, M\"uller and Yang \cite{my09} obtained the difference characterization of Triebel--Lizorkin spaces
on RD-spaces. Koskela et al. \cite{kyz10} established the grand Littlewood--Paley function characterization of
Triebel--Lizorkin spaces on RD-spaces. Moreover, Yang and Zhou \cite{yz11} characterized Besov and Triebel--Lizorkin
spaces via (local) Hardy spaces on RD-spaces.

As was mentioned above, Calder\'{o}n reproducing formulae take an important role in both real-variable
theory of function spaces and boundedness of operators. Thus, to develop a real-variable theory of function spaces
on spaces of homogeneous type, a key point is to establish corresponding Calder\'{o}n
reproducing formulae. Along this line, a breakthrough work was made by Auscher and Hyt\"{o}nen
\cite{ah13,ah15} in which a wavelet system
on a space $X$ of homogeneous type was constructed. The constructed wavelets in \cite{ah13,ah15} have the exponential
decay and the $\eta$-H\"{o}lder regularity with
some $\eta\in(0,1)$ (see \cite[Theorem 7.1]{ah13}). As a direct application, Fu and Yang \cite{fy18} obtained
the wavelet characterizations of atomic Hardy spaces on $X$ introduced by Coifman and Weiss \cite{cw77}.

As a first attempt of reproducing formulae on a given space $X$
of homogeneous type, Han et al. \cite{hlw18} established wavelet reproducing formulae which hold true in both test
functions and distributions. Soon after, Han et al. \cite{hhl16} characterized atomic Hardy spaces via wavelet coefficients. Another kind of Hardy spaces by
using different spaces of distributions was also introduced by Han et al. \cite{hhl17}.
Later, He et al. \cite{hhllyy19} introduced a new kind of (inhomogeneous) approximations of the identity with exponential decay
(for short, exp-(I)ATI) and
established new Calder\'on reproducing formulae via these new approximations of the identity on $X$. Using these
new Calder\'on reproducing formulae, He et al. \cite{hhllyy19} and \cite{hyy19} respectively established
the real-variable theory of (local) Hardy spaces  on $X$, in which
the question asked by Coifman and Weiss \cite{cw77} was completely answered, that is, no additional
geometrical property is necessary to guarantee the radial maximal function characterization of Hardy
spaces on $X$. Later, Fu et al. \cite{fmy20} introduced the Musielak--Orlicz Hardy spaces on $X$ and
established their various real-variable characterizations. Recently, Zhou et al. \cite{zhy20} obtained a
real-variable theory of Hardy--Lorentz spaces on $X$.

As applications of Hardy spaces on a metric measure space $X$ of homogeneous
type, Liu et al. \cite{lyy18} considered the bilinear
decomposition for pointwise products of Hardy spaces and their dual spaces.
Liu et al. \cite{lcfy17,lcfy18} also obtained the
endpoint boundedness of commutators on Hardy spaces over $X$. We should mention
that, in \cite{lcfy18}, Liu et al.  introduced
almost diagonal operators, which were applied to study the boundedness
of some operators on $\BMO(X)$, the space of functions with bounded
mean oscillation (see \cite[Proposition 4.2]{lcfy18}).
Georgiadis et al. \cite{gkkp17} (see also \cite{gk20}), and Kerkyacharian and Petrushev
\cite{kp15} studied homogeneous Besov and Triebel--Lizorkin spaces, associated with
operators, on $X$. For some recent progress on the
real-variable theory of function spaces on $X$ associated with operators, see also,
for instance, \cite{bbd18,bd20,bdk,bdl,bdl20,gkkp19,gk20}. Moreover, the real-variable theory of
Besov and Triebel--Lizorkin spaces associated
with operators on spaces of homogeneous type can be applied
to statistics and probability; see, for instance,
\cite{cgkpp20}.

Motivated by these previous works, on one hand, Wang et al. \cite{whhy20} introduced Besov and
Triebel--Lizorkin spaces on a given space $X$ of homogeneous type via the exp-(I)ATI,
and also obtained the boundedness of Calder\'{o}n--Zygmund operators on these function
spaces, and Wang et al. \cite{whyy21} further established the difference characterization
of these spaces. On the other hand, Han et al. \cite{hhhlp20} introduced
another kind of Besov and Triebel--Lizorkin spaces on $X$ by using the wavelet system introduced in
\cite{ah13}, and gave a necessary and sufficient condition of the
embedding theorem of Besov and Triebel--Lizorkin spaces. Comparing the results in \cite{hhhlp20} with those in
\cite{whhy20}, the following question naturally arises.

\begin{question}\label{q-co}
Do these two kinds of Besov and Triebel--Lizorkin spaces on a given space $X$ of homogeneous type introduced,
respectively, in \cite{whhy20} and \cite{hhhlp20}, coincide?
\end{question}

The first aim of this article is to give an affirmative answer to Question \ref{q-co}.

On the other hand, it is well known that the most
important and useful cores of the real-variable
characterizations of Besov and Triebel--Lizorkin spaces on Euclidean spaces
are their atomic and their molecular characterizations (see, for instance, \cite{fj85,fj90}).
Recall that the atomic and the molecular characterizations of Besov and Triebel--Lizorkin
spaces on Ahlfors regular spaces were obtained by Han and Sawyer \cite{hs94}.
Moreover, as a special case of Triebel--Lizorkin spaces on a given space $X$ of homogeneous type, (local) Hardy
spaces on $X$ also hold their atomic and their molecular characterizations
(see \cite{hhllyy19,hyy19}). Motivated by these, it is natural to ask the following question.

\begin{question}\label{q-decay}
Do Besov and Triebel--Lizorkin spaces on a given space $X$ of homogeneous type admit atomic or molecular
characterizations?
\end{question}

In this article, we partly answer Question \ref{q-decay} by establishing the molecular characterization of Besov
and Triebel--Lizorkin spaces on $X$, which is new \emph{even} for the corresponding function spaces on RD-spaces,
while it is still unknown whether or not these function spaces
have an atomic characterization which is a \emph{challenging problem}
due to the lack of Calder\'on reproducing formulae with bounded support. To establish the molecular characterization
of Besov and Triebel--Lizorkin spaces on $X$, we
introduce almost diagonal operators on Besov and Triebel--Lizorkin sequence spaces
on $X$, and obtain their boundedness. Using this boundedness and the established
wavelet characterization of Besov and Triebel--Lizorkin
spaces, we then establish the molecular characterization of Besov and Triebel--Lizorkin
spaces and hence partly answer Question \ref{q-decay}.
Moreover, applying this molecular characterization, we further establish the
Littlewood--Paley characterizations of Triebel--Lizorkin spaces.
Similar results for inhomogeneous
Besov and Triebel--Lizorkin spaces are also obtained.

We point out that all these results get rid of their dependence on the reverse doubling
property of $\mu$ and the triangle inequality of $d$, by fully using the
geometrical property of $X$ expressed via
its equipped quasi-metric $d$, dyadic reference points, dyadic cubes, and wavelets
(see Remark \ref{geo} below for more details).
These results further \emph{complete} the real-variable theory of
Besov and Triebel--Lizorkin spaces
on spaces of homogeneous type.

The organization of the remainder of this article is as follows.

In Section \ref{s-pre}, we recall some known notions and conclusions on spaces of homogeneous type used in this article.

Section \ref{s-wave} mainly concerns about the wavelet characterization of homogeneous Besov
and Triebel--Lizorkin spaces.

In Section \ref{s-ado}, we introduce the notions of homogeneous Besov and Triebel--Lizorkin sequence spaces
and almost diagonal operators on them, and then prove that
these operators are bounded on homogeneous Besov and Triebel--Lizorkin sequence spaces.

Sections \ref{s-mol} and \ref{s-lp} mainly concern the applications of results
obtained in the previous sections.

In Section \ref{s-mol}, we first introduce the notion of molecules on $X$ in Definition
\ref{def-mol} below, and then establish the molecular characterization of Besov and
Triebel--Lizorkin spaces, which is new even for Besov and
Triebel--Lizorkin spaces on RD-spaces. Observe that, in the establishment of
the atomic and the molecular characterizations of (local) Hardy spaces
in \cite{hhllyy19,hyy19} and Hardy--Lorentz spaces in \cite{zhy20}, to avoid the
dependence on the reverse doubling property of the
equipped measure $\mu$, the main techniques are to use Calder\'on reproducing formulae
which include the terms of exponential decay, consisting of the side length of related
dyadic cubes and the distance between point and dyadic reference points,
and therefore to fully use the geometrical properties of $X$.
Since the molecules in Definition \ref{def-mol} have the form of test
functions on $X$, which only have the polynomial decay, but without the
exponential decay, it is curious how to get rid of dependence on
the reverse doubling property of $\mu$ in order to establish the molecular
characterization of Besov and Triebel--Lizorkin spaces. Indeed, to escape
the dependence on the reverse doubling property of $\mu$,
differently from those molecules in \cite[Definition (6.21)]{hs94} centered at
all dyadic cubes, we choose these molecules in Definition \ref{def-mol}
centered at \emph{some subtly selected dyadic cubes}, namely, on those dyadic cubes
which are ``supports" of wavelets constructed in \cite{ah13},
so that we can fully use the wavelet characterization of Besov and Triebel--Lizorkin spaces,
and therefore the geometrical properties of $X$.

In Section \ref{s-lp}, we establish the Littlewood--Paley function characterizations of Triebel--Lizorkin
spaces on a given space $X$ of homogeneous type. Indeed, we use the molecular characterization of Triebel--Lizorkin spaces
established in Section \ref{s-mol} to obtain their Lusin area function characterization.
Moreover, using this Lusin area function characterization and establishing an important change-of-angle formula
for the variant of the Lusin area function (see Proposition \ref{prop-angle} below),
we then obtain the Littlewood--Paley $g_\lz^*$-function characterization
of Triebel--Lizorkin spaces, in which the range of $\lz$ is the \emph{known best possible}
(see Remark \ref{rem-glz} below).

In Section \ref{s-in}, we present all the corresponding results in the inhomogeneous case.

At the end of this section, we make some conventions on notation. We \emph{always assume} that $\omega_0$ is as
in \eqref{eq-updim} and $\eta$ is the H\"{o}lder regularity index of approximations of the
identity with exponential decay (see Definition \ref{def-eti} below). We assume that $\dz$ is a very small
positive number, for instance, $\dz\le(2A_0)^{-10}$ in order to construct the dyadic cube system and the
wavelet system on $X$ (see \cite[Theorem 2.2]{hk12} or Lemma \ref{lem-cube} below).
For any $x,\ y\in X$ and $r\in(0,\fz)$, let
$$
V_r(x):=\mu(B(x,r))\quad \mathrm{and}\quad V(x,y):=
\begin{cases}
\mu(B(x,d(x,y))) & \textup{if }x\neq y,\\
0 & \textup{if } x=y,
\end{cases}
$$
where $B(x,r):=\{y\in X:\ d(x,y)<r\}$.
We always let $\nn:=\{1,2,\ldots\}$ and $\zz_+:=\nn\cup\{0\}$.
For any $p\in[1,\fz]$, we use $p'$ to denote its \emph{conjugate index}, namely, $1/p+1/p'=1$.
The symbol $C$ denotes a positive constant which is independent of the main parameters, but it may vary from
line to line. We also use $C_{(\az,\bz,\ldots)}$ to denote a positive constant depending on the indicated
parameters $\az$, $\bz$, \ldots. The symbol $A \ls B$ means that there exists a positive constant $C$ such
that $A \le CB$. The symbol $A \sim B$ is used as an abbreviation of $A \ls B \ls A$. If $f\le Cg$ and $g=h$ or
$g\le h$, we then write $f\ls g\sim h$ or $f\ls g\ls h$, \emph{rather than} $f\ls g=h$
or $f\ls g\le h$. For any $s,\ t\in\rr$, denote the \emph{minimum} of $s$
and $t$ by $s\wedge t$ and the \emph{maximum} by $s\vee t$.
For any finite set $\CJ$, we use $\#\CJ$ to denote its \emph{cardinality}. Also, for any set
$E$ of $X$, we use $\mathbf{1}_E$ to denote its \emph{characteristic function} and $E^\complement$ the set
$X\setminus E$. For any $x,\ y\in X$ and $\ez,\ r\in(0,\fz)$, we always write
$$
P_\ez(x,y;r):=\frac 1{V_r(x)+V(x,y)}\lf[\frac{r}{r+d(x,y)}\r]^\ez.
$$

\section{Preliminaries}\label{s-pre}

In this section, we mainly recall some known notions and conclusions on a given space
$(X,d,\mu)$ of homogeneous type. We begin with the notion of spaces of test functions, which was originally
introduced in \cite{hmy06} (see also \cite{hmy08}).

\begin{definition}\label{def-test}
Let $x_1\in X$, $r\in(0,\fz)$, $\bz\in(0,1]$, and $\gz\in(0,\fz)$. A function $f$ defined on $X$ is called a
\emph{test function of type $(x_1,r,\bz,\gz)$}, denoted by $f\in\CG(x_1,r,\bz,\gz)$, if there exists a
positive constant $C$ such that
\begin{enumerate}
\item (the \emph{size condition}) for any $x\in X$, $|f(x)|\le C P_\gz(x_1,x;r)$;
\item (the \emph{regularity condition}) for any $x,\ y\in X$ satisfying $d(x,y)\le (2A_0)^{-1}[r+d(x_1,x)]$,
\begin{equation*}
|f(x)-f(y)|\le C\lf[\frac{d(x,y)}{r+d(x_1,x)}\r]^\bz P_\gz(x_1,x;r).
\end{equation*}
\end{enumerate}
For any $f\in\CG(x_1,r,\bz,\gz)$, its norm
$\|f\|_{\CG(x_1,r,\bz,\gz)}$
is defined by setting
$$
\|f\|_{\CG(x_1,r,\bz,\gz)}:=\inf\{C\in(0,\fz):\ \textup{(i) and (ii) hold true}\}.
$$
Also define $\GO{x_1,r,\bz,\gz}$ by setting
$$
\GO{x_1,r,\bz,\gz}:=\lf\{f\in\CG(x_1,r,\bz,\gz):\ \int_X f(x)\,d\mu(x)=0\r\}
$$
equipped with the norm $\|\cdot\|_{\GO{x_1,r,\bz,\gz}}:=\|\cdot\|_{\CG(x_1,r,\bz,\gz)}$.
\end{definition}

It is known that, for any fixed $\bz\in(0,1]$ and $\gz\in(0,\fz)$, $\CG(x_1,r,\bz,\gz)$ and
$\mathring\CG(x_1,r,\bz,\gz)$ are Banach spaces.
In what follows, we fix $x_0\in X$, and let
$\CG(\bz,\gz):=\CG(x_0,1,\bz,\gz)$ and $\GO{\bz,\gz}:=\GO{x_0,1,\bz,\gz}$.
It is easy to show that,
for any $x_1\in X$ and $r\in(0,\fz)$,
$\CG(\bz,\gz)=\CG(x_1,r,\bz,\gz)$ and
$\GO{\bz,\gz}=\mathring\CG(x_1,r,\bz,\gz)$  with the positive
equivalence constants depending on $x_0$, $x_1$, and $r$.

Now, we suppose that $\ez\in(0,1]$ and $\bz,\ \gz\in(0,\ez]$, and define $\CG_0^\ez(\bz,\gz)$ [resp.,
$\mathring{\CG}_0^\ez(\bz,\gz)$] to be the closure of $\CG(\ez,\ez)$ [resp., $\mathring\CG(\ez,\ez)$] in the space
$\CG(\bz,\gz)$ [resp., $\mathring\CG(\bz,\gz)$], equipped with the norm
$\|\cdot\|_{\CG_0^\ez(\bz,\gz)}:=\|\cdot\|_{\CG(\bz,\gz)}$ [resp., $\|\cdot\|_{\mathring{\CG}_0^\eta(\bz,\gz)}
:=\|\cdot\|_{\mathring\CG(\bz,\gz)}$]. Denote by $(\CG_0^\ez(\bz,\gz))'$
[resp., $(\mathring{\CG}_0^\ez(\bz,\gz))'$]
the dual space of $\CG_0^\ez(\bz,\gz)$ [resp., $\mathring\CG_0^\ez(\bz,\gz)$], equipped with the weak-$*$
topology. The spaces $\CG_0^\ez(\bz,\gz)$ and $\mathring{\CG}_0^\ez(\bz,\gz)$ are called the \emph{spaces of test
functions}, and $(\CG_0^\ez(\bz,\gz))'$ and $(\mathring{\CG}_0^\ez(\bz,\gz))'$
the \emph{spaces of distributions}.

Let $p\in(0,\fz]$. The \emph{Lebesgue space $L^p(X)$} is defined to be the set of all measurable functions $f$
on $X$ such
that $\|f\|_{L^p(X)}<\fz$, where
$$
\|f\|_{L^p(X)}:=\begin{cases}
\lf[\displaystyle \int_X|f(x)|^p\,d\mu(x)\r]^{1/p} & \textup{if } p\in(0,\fz),\\
\displaystyle\esssup_{x\in X} |f(x)| & \textup{if } p=\fz.
\end{cases}
$$
For any given $p\in(0,\fz)$, the \emph{weak Lebesgue space $L^{p,\fz}(X)$} is defined to be the set of all
measurable functions $f$ on $X$ such that
$$
\|f\|_{L^{p,\fz}(X)}:=\sup_{\lz\in(0,\fz)}\lz[\mu(\{x\in X:\ |f(x)|>\lz\})]^{1/p}<\fz.
$$
Denote by $L^1_\loc(X)$ the set of all locally integrable functions on $X$. For any $f\in L^1_\loc(X)$, the
\emph{Hardy--Littlewood maximal function $\CM(f)$} is defined by setting, for any $x\in X$,
$$
\CM(f)(x):=\sup_{\textup{ball }B\ni x}\frac 1{\mu(B)}\int_B|f(y)|\,d\mu(y).
$$
In \cite{cw71}, Coifman and Weiss showed that the Hardy--Littlewood maximal operator $\CM$ is bounded on
$L^p(X)$ with any given $p\in(1,\fz)$ (see also \cite[(3.6)]{cw77}), and   from
$L^1(X)$ to $L^{1,\fz}(X)$ (see, for instance, \cite[pp.\ 71--72, Theorem 2.1]{cw71}).

Next, we recall the dyadic system on $X$ introduced by Hyt\"{o}nen and Kairema \cite{hk12}.

\begin{lemma}[{\cite[Theorem 2.2]{hk12}}]\label{lem-cube}
Fix constants $0<c_0\le C_0<\fz$ and $\dz\in(0,1)$ such that $12A_0^3C_0\dz\le c_0$. Assume that
a set of points, $\CX^{k}:=\{z_\az^k:\ k\in\zz,\ \az\in\CA_k\}\subset X$ with $\CA_k$, for any $k\in\zz$, being a
countable set of indices, has the following properties: for any $k\in\zz$,
\begin{enumerate}
\item[\rm (i)] $d(z_\az^k,z_\bz^k)\ge c_0\dz^k$ if $\az\neq\bz$;
\item[\rm (ii)] $\min_{\az\in\CA_k} d(x,z_\az^k)\le C_0\dz^k$ for any $x\in X$.
\end{enumerate}
Then there exists a family of sets, $\{Q_\az^k:\  k\in\zz,\ \az\in\CA_k\}$, satisfying
\begin{enumerate}
\item[\rm (iii)] for any $k\in\zz$, $\bigcup_{\az\in\CA_k} Q_\az^k=X$ and $\{ Q_\az^k:\ \az\in\CA_k\}$
is disjoint;
\item[\rm (iv)] if $k,\ l\in\zz$ and $k\le l$, then either $Q_\az^k\supset Q_\bz^l$ or
$Q_\az^k\cap Q_\bz^l=\emptyset$;
\item[\rm (v)] for any $k\in\zz$ and $\az\in\CA_k$, $B(z_\az^k,c_\natural\dz^k)\subset Q_\az^k\subset
B(z_\az^k,C^\natural\dz^k)$,
where $c_\natural:=(3A_0^2)^{-1}c_0$, $C^\natural:=2A_0C_0$, and $z_\az^k$ is called ``the \emph{center}'' of
$Q_\az^k$.
\end{enumerate}
\end{lemma}

According to the construction of $\{\CA_k\}_{k\in\zz}$, we may further assume that, for any $k\in\zz$,
$\CX^{k+1}\supset\CX^{k}$. Thus, for any $k\in\zz$, we can let
$$
\CG_k:=\CA_{k+1}\setminus\CA_k\qquad \textup{and}\qquad
\CY^k:=\{z_\az^{k+1}\}_{\az\in\CG_k}=:\{y_\az^{k}\}_{\az\in\CG_k}.
$$

The following lemma is on the existence of the wavelet system on $X$, which is a combination of
\cite[Theorem 7.1 and Corollary 10.4]{ah13}.

\begin{lemma}\label{lem-wave}
There exist constants $a\in(0,1]$, $\eta\in(0,1)$, $C,\ \nu\in(0,\fz)$, and wavelet functions
$\{\psi_\az^k:\ k\in\zz,\ \az\in\CG_k\}$ [resp., $\{\psi_\az^k:\ k\in\zz\cap[k_0,\fz),\ \az\in\CG_k\}$ for
some $k_0\in\zz$ when $\mu(X)<\fz$] satisfying that, for any $k\in\zz$ [resp., $k\in\zz\cap[k_0\in\fz)$ when
$\mu(X)<\fz$] and $\az\in\CG_k$,
\begin{enumerate}
\item (the \emph{decay condition}) for any $x\in X$,
$$
\lf|\psi_\az^k(x)\r|\le
\frac C{\sqrt{V_{\dz^k}(y_\az^k)}}\exp\lf\{-\nu\lf[\frac{d(x,y_\az^k)}{\dz^k}\r]^a\r\};
$$
\item (the \emph{H\"{o}lder-regularity condition}) for any $x,\ x'\in X$ with $d(x,x')\le\dz^k$,
$$
\lf|\psi_\az^k(x)-\psi_\az^k(x')\r|\le\frac C{\sqrt{V_{\dz^k}(y_\az^k)}}
\lf[\frac{d(x,x')}{\dz^k}\r]^\eta\exp\lf\{-\nu\lf[\frac{d(x,y_\az^k)}{\dz^k}\r]^a\r\};
$$
\item (the \emph{cancellation condition})
$$
\int_X \psi_\az^k(x)\,d\mu(x)=0.
$$
\end{enumerate}
Moreover, the functions $\{\psi_\az^k\}_{k,\, \az}$ form an orthonormal basis of $L^2(X)$, and an unconditional
basis of $L^p(X)$ for any given $p\in(1,\fz)$.
\end{lemma}

Now, we recall the notion of approximations of the identity with exponential decay introduced in
\cite{hlyy19}. In what follows, for any $k\in\zz$ and $y\in X$, let
$d(y,\mathcal Y^k):=\inf_{z\in\CY^k} d(y,z)$.

\begin{definition}[{\cite[Definition 2.7]{hlyy19}}]\label{def-eti}
Assume that $\mu(X)=\fz$. A sequence $\{Q_k\}_{k\in\zz}$ of bounded linear integral operators on
$L^2(X)$ is called an \emph{approximation of the identity with exponential decay}
(for short, $\exp$-ATI) if there exist constants $C,\ \nu\in(0,\fz)$, $a\in(0,1]$, and $\eta\in(0,1)$
such that, for any $k\in\zz$, the kernel of the operator $Q_k$, a function on $X\times X$, which is still denoted
by $Q_k$, satisfies
\begin{enumerate}
\item $\sum_{k=-\fz}^\fz Q_k=I$ in $L^2(X)$, where $I$ is the identity operator on $L^2(X)$;
\item for any $x,\ y\in X$,
\begin{align*}
|Q_k(x,y)|&\le C\frac1{\sqrt{V_{\dz^k}(x)\,V_{\dz^k}(y)}}
\exp\lf\{-\nu\lf[\frac{d(x,y)}{\dz^k}\r]^a\r\}\exp\lf\{-\nu\lf[\frac{\max\{d(x, \CY^k),d(y,\CY^k)\}}{\dz^k}\r]^a\r\}\\
&=:CE_k(x,y);
\end{align*}
\item for any $x,\ x',\ y\in X$ with $d(x,x')\le\dz^k$,
\begin{equation*}
|Q_k(x,y)-Q_k(x',y)|+|Q_k(y,x)-Q_k(y, x')|\le C\lf[\frac{d(x,x')}{\dz^k}\r]^\eta E_k(x,y);
\end{equation*}
\item for any $x,\ x',\ y,\ y'\in X$ with $d(x,x')\le\dz^k$ and $d(y,y')\le\dz^k$,
\begin{equation*}
|[Q_k(x,y)-Q_k(x',y)]-[Q_k(x,y')-Q_k(x',y')]|
\le C\lf[\frac{d(x,x')}{\dz^k}\r]^\eta\lf[\frac{d(y,y')}{\dz^k}\r]^\eta E_k(x,y);
\end{equation*}
\item for any $x,\ y\in X$,
\begin{equation*}
\int_X Q_k(x,y')\,d\mu(y')=0=\int_X Q_k(x',y)\,d\mu(x').
\end{equation*}
\end{enumerate}
\end{definition}

The existence of such an exp-ATI on $X$ as in Definition \ref{def-eti} is guaranteed by \cite[Theorem 7.1]{ah13},
where $\eta$ is the same as in \cite[Theorem 3.1]{ah13} (see also
Lemma \ref{lem-wave}) which might be very small (see also \cite[Remark 2.8(i)]{hlyy19}).
However, if $d$ is a metric, then $\eta\in(0,1)$ can be chosen
arbitrarily close to $1$ (see \cite[Corollary 6.13]{ht14}).
Moreover, in Definition \ref{def-eti}, we need $\diam X=\fz$ to guarantee (v). Observe that it was shown in
\cite[Lemma 5.1]{ny97} or \cite[Lemma 8.1]{ah13} that $\diam X=\fz$ implies $\mu(X)=\fz$ and hence, under
the assumptions of this article, $\diam X=\fz$ if and only if $\mu(X)=\fz$.

Next, we recall the Calder\'on reproducing formulae. In what follows, by (iii) and (iv) of Lemma \ref{lem-cube},
we always choose a $j_0\in\nn$ to be  sufficiently large such that,
for any $k\in\zz$ and $\az\in\CA_k$,
$$
Q_\az^k=\bigcup_{\{\tau\in\CA_{k+j_0}:\ Q_\tau^{k+j_0}\subset Q_\az^k\}}Q_\tau^{k+j_0}.
$$
Then we let ${\mathfrak N}(k,\az):=\{\tau\in\CA_{k+j_0}:\ Q_\tau^{k+j_0}\subset Q_\az^k\}$
and $N(k,\az)$ to be the \emph{cardinality} of the set ${\mathfrak N}(k,\az)$. Moreover, by Lemma \ref{lem-cube}(v) and
the doubling property \eqref{eq-doub}, we find that $N(k,\az)$ is controlled by a harmless positive
constant depending only on $j_0$, $A_0$, and $\omega$ in \eqref{eq-doub}.
For any $k\in\zz$ and $\az\in\CA_k$, we rearrange
the set $\{Q_\tau^{k+j_0}:\ \tau\in{\mathfrak N}(k,\az)\}$ as $\{Q_\az^{k,m}\}_{m=1}^{N(k,\az)}$, whose centers are
denoted, respectively, by $\{z_\az^{k,m}\}_{m=1}^{N(k,\az)}$.

The following discrete homogeneous Calder\'on reproducing formula was  established in
\cite[Theorems 5.11]{hlyy19}.

\begin{lemma}\label{lem-hdrf}
Let $\{Q_k\}_{k\in\zz}$ be an {\rm exp-ATI} and $\bz,\ \gz\in(0,\eta)$ with $\eta$ as in Definition \ref{def-eti}.
For any $k\in\zz$, $\az\in\CA_k$, and  $m\in\{1,\ldots,N(k,\az)\}$, suppose that $y_\az^{k,m}$ is an arbitrary
point in $Q_\az^{k,m}$. Then there exists a sequence $\{\wz{Q}_k\}_{k=-\fz}^\fz$ of bounded linear operators
on $L^2(X)$ such that, for any $f\in(\GOO{\bz,\gz})'$,
\begin{align*}
f(\cdot)=\sum_{k=-\fz}^\fz\sum_{\az\in\CA_k}\sum_{m=1}^{N(k,\az)}\mu\lf(Q_\az^{k,m}\r)
\wz{Q}_k\lf(\cdot,y_\az^{k,m}\r)Q_kf\lf(y_\az^{k,m}\r)
\end{align*}
in $(\GOO{\bz,\gz})'$. Moreover, there exists a positive constant $C$, independent of $f$ and
$\{y_\az^{k,m}:\ k\in\zz,\ \az\in\CA_k,\ m\in\{1,\ldots,N(k,\az)\}\}$ such that
\begin{enumerate}
\item for any $x,\ y\in X$,
$$
\lf|\wz{Q}_k(x,y)\r|\le C P_\gz(x,y;\dz^k);
$$
\item for any $x,\ x',\ y\in X$ with $d(x,x')\le(2A_0)^{-1}[\dz^k+d(x,y)]$,
$$
\lf|\wz{Q}_k(x,y)-\wz{Q}_k(x',y)\r|\le C \lf[\frac{d(x,x')}{\dz^k+d(x,y)}\r]^\bz P_\gz(x,y;\dz^k);
$$
\item for any $x\in X$,
\begin{equation*}
\int_X \wz{Q}_k(x,y)\,d\mu(y)=0=\int_X\wz{Q}_k(y,x)\,d\mu(y).
\end{equation*}
\end{enumerate}
\end{lemma}

Now, we recall the inhomogeneous discrete Calder\'{o}n reproducing formula. To this end, we need the notion
of $\exp$-IATIs. Recall that, by \cite[Remark 6.2]{hlyy19}, the existence of an exp-IATI does not need the
assumption $\mu(X)=\fz$.

\begin{definition}[{\cite[Definition 6.1]{hlyy19}}]\label{def-ieti}
A sequence $\{Q_k\}_{k=0}^\fz$ of bounded operators on $L^2(X)$ is called an \emph{inhomogeneous
approximation of the identity with exponential decay} (for short, $\exp$-IATI) if there exist constants
$C,\ \nu\in(0,\fz)$, $a\in(0,1]$, and $\eta\in(0,1)$ such that, for any $k\in\zz_+$, the kernel of the
operator $Q_k$, which is still denoted by $Q_k$, has the following properties:
\begin{enumerate}
\item (the \emph{identity condition}) $\sum_{k=0}^\fz Q_k=I$ in $L^2(X)$;
\item when $k\in\nn$, $Q_k$ satisfies (ii) through (v) of Definition \ref{def-eti};
\item $Q_0$ satisfies (ii), (iii), and (iv) of Definition \ref{def-eti} with $k=0$ but without the term
$$
\exp\lf\{-\nu\lf[\max\lf\{d\lf(x,\CY^0\r),d\lf(y,\CY^0\r)\r\}\r]^a\r\};
$$
moreover, for any $x\in X$, $\int_X Q_0(x,y)\,d\mu(y)=1=\int_X Q_0(y,x)\,d\mu(y)$.
\end{enumerate}
\end{definition}

Next, we recall the following inhomogeneous Calder\'{o}n reproducing formula established in \cite{hlyy19}.

\begin{lemma}\label{lem-idrf}
Let $\{Q_k\}_{k=0}^\fz$ be an $\exp$-{\rm IATI} and $\bz,\ \gz\in(0,\eta)$ with $\eta$ as in Definition
\ref{def-eti}. Then there exist $N,\ j_0\in\nn$
such that, for any $y_\az^{k,m}\in Q_\az^{k,m}$ with $k\in\nn$, $\az\in\CA_k$, and
$m\in\{1,\ldots,N(k,\az)\}$, there exists a sequence $\{\wz Q_k\}_{k=0}^\fz$ of bounded linear operators on
$L^2(X)$ such that, for any $f\in(\go{\bz,\gz})'$,
\begin{align}\label{eq-idrf}
f(\cdot)=&\sum_{\az\in\CA_0}\sum_{m=1}^{N(0,\az)}\int_{Q_\az^{0,m}}\wz Q_0(\cdot,y)\,d\mu(y)Q_{\az,1}^{0,m}(f)\noz\\
&\quad+\sum_{k=1}^N\sum_{\az\in\CA_k}\sum_{m=1}^{N(k,\az)}\mu\lf(Q_\az^{k,m}\r)\wz Q_k\lf(\cdot,y_\az^{k,m}\r)
Q_{\az,1}^{k,m}(f)\noz\\
&\quad+\sum_{k=N+1}^\fz\sum_{\az\in\CA_k}\sum_{m=1}^{N(k,\az)}\mu\lf(Q_\az^{k,m}\r)
\wz Q_k\lf(\cdot,y_\az^{k,m}\r)Q_kf\lf(y_\az^{k,m}\r)
\end{align}
in $(\go{\bz,\gz})'$, where, for any $k\in\{0,\ldots,N\}$, $\az\in\CA_k$, and
$m\in\{1,\ldots,N(k,\az)\}$,
$$
Q_{\az,1}^{k,m}(f):=\frac 1{\mu(Q_\az^{k,m})}\int_{Q_\az^{k,m}}Q_kf(u)\,d\mu(u).
$$
Moreover, for any $k\in\zz_+$, the kernel of $\wz{Q}_k$ satisfies (i) and (ii) of Lemma \ref{lem-hdrf},
with the implicit positive constant independent of the choice of $f$ and
$\{y_\az^{k,m}:\ k\in\nn,\ \az\in\CA_k,\ k\in\{1,\ldots,N(k,\az)\}\}$,
and the following integral condition that, for any $x\in X$,
$$
\int_X\wz Q_k(x,y)\,d\mu(y)=\int_X\wz Q_k(y,x)\,d\mu(y)=\begin{cases}
1 & \hbox{if $k\in\{0,\ldots,N\}$},\\
0 & \hbox{if $k\in\{N+1,N+2,\ldots\}$}.
\end{cases}
$$
\end{lemma}

\begin{remark}\label{geo}
We should mention that, compared with the approximations of the identity and Calder\'on reproducing
formulae on   RD-spaces (see \cite{hmy08}), these exp-ATIs and
Calder\'on reproducing formulae on $X$ have
some essential differences presented via some terms such as
$$
\exp\left\{-\nu\left[\frac{\max\{d(x,\CY^k),d(y,\CY^k)\}}{\delta^k}\right]^a\right\}.
$$
Observe that here $x$, $y\in X$, $\CY^k$ is the set of dyadic reference points appearing
in Lemma \ref{lem-cube}, and $d(y,\CY^k)$ is the distance between $y$ and $\CY^k$. Moreover, by
Lemma \ref{lem-wave}, for any given $k\in\zz$ and $\az\in\CA_k$,  the wavelet $\psi_\az^k$ also has
an exponential decaying term
$$
\exp\lf\{-\nu\lf[\frac{d(x,y_\az^k)}{\dz^k}\r]^a\r\},
$$
where $y_\az^k$ is one of dyadic reference points which can be seen as the ``center'' of $\psi_{\az}^k$.
Thus, such terms
closely connect with the geometry of the given space $X$ of homogeneous type.
\end{remark}

At the end of this section, we list some useful inequalities which are widely used later in this article.
We begin with a known very basic inequality.

\begin{lemma}\label{lem-pin}
Let $p\in(0,1]$. Then, for any $\{a_k\}_{k=1}^\fz\subset\cc$,
$$
\lf(\sum_{k=1}^\fz |a_k|\r)^p\le\sum_{k=1}^\fz |a_k|^p.
$$
\end{lemma}

The proofs of the following two lemmas are similar, respectively, to those
of \cite[Lemma 3.5]{whhy20} and \cite[Lemma 5.3]{hmy08}; we omit the details here.

\begin{lemma}\label{lem-bbes}
Let $\gamma \in (0,\fz)$ and $p \in (\omega_0/(\omega_0+\gamma),1]$ with $\omega_0$ as in \eqref{eq-updim}.
Then there exists a positive constant $C$ such that, for any $k,\ k'\in\zz$ and $x\in X$,
\begin{equation*}
\sum_{\alpha\in\CG_k}\mu\lf(Q_\az^{k+1}\r)\lf[P_\gz\lf(x,y_\az^k;\dz^{k\wedge k'}\r)\r]^p
\le C\lf[V_{\delta^{k\wedge k'}}(x)\r]^{1-p}.
\end{equation*}
\end{lemma}

\begin{lemma}\label{lem-btl}
Let $\gz\in (0,\infty)$ and $r\in (\omega/(\omega+\gamma),1]$ with $\omega$ as in \eqref{eq-doub}. Then
there exists a positive constant $C$ such that, for any $k,\ k'\in\zz$, $x\in X$, and
$\{a_\alpha^{k}:\ k\in\zz,\ \az\in\CG_k\}\subset\cc$,
\begin{equation}\label{eq-btl}
\sum_{\az\in\CG_k}\mu\left(Q_\az^{k+1}\right)P_\gz\lf(x,y_\az^k;\dz^{k\wedge k'}\r)\lf|a_\az^k\r|
\le C \delta^{[(k\wedge k')-k]\omega(1/r-1)}\lf[\CM\left(\sum_{\az\in\CG_k}\lf|a_\az^{k}\r|^r
\mathbf 1_{Q_\az^{k+1}}\r)(x)\r]^{1/r}.
\end{equation}
\end{lemma}

Finally, we recall the following Fefferman--Stein vector-valued maximal inequality on $X$ obtained by Grafakos
et al. \cite{gly09}.

\begin{lemma}\label{lem-fs}
Let $p\in(1,\fz)$ and $u\in(1,\fz]$. Then there exists a positive constant $C$ such that, for any
sequence $\{f_j\}_{j=1}^\fz$ of measurable functions,
$$
\lf\|\lf\{\sum_{j=1}^\fz[\CM(f_j)]^u\r\}^{1/u}\r\|_{L^p(X)}
\le C\lf\|\lf(\sum_{j=1}^\fz|f_j|^u\r)^{1/u}\r\|_{L^p(X)}
$$
with the usual modification made when $u=\fz$.
\end{lemma}

\section{Wavelet characterization of Besov and Triebel--Lizorkin spaces}\label{s-wave}

In this section, we establish the wavelet characterization of homogeneous Besov and Triebel--Lizorkin
spaces. To this end, we first recall the notion of Besov and Triebel--Lizorkin spaces. In what follows,
we assume that $\mu(X)=\fz$ and, for any $s\in(-1,1)$ and $\ez\in(0,1]$, let
\begin{equation*}
p(s,\ez):=\max\lf\{\frac{\omega_0}{\omega_0+\ez},\frac{\omega_0}{\omega_0+s+\ez}\r\},
\end{equation*}
where $\omega_0$ is the same as in \eqref{eq-updim}.

\begin{definition}\label{def-btl}
Let $\{Q_k\}_{k=-\fz}^\fz$ be an exp-ATI, and $s\in(-\eta,\eta)$ with $\eta$ as in Definition \ref{def-eti}.
Suppose $p\in(p(s,\eta),\fz]$, $q\in(0,\fz]$, and $\bz$ and $\gz$ satisfy
\begin{equation}\label{eq-bzgz}
\bz\in\lf(\max\left\{0,-s+\omega_0\lf(\frac{1}{p}-1\r)_+\right\},\eta\r)
\quad\hbox{and}\quad
\gz\in\lf(\max\left\{s,\omega_0\lf(\frac{1}{p}-1\r)_+\right\},\eta\r)
\end{equation}
with $\omega_0$ as in \eqref{eq-updim}.
\begin{enumerate}
\item If $s\in(-(\bz\wedge\gz),\bz\wedge\gz)$, $p\in(p(s,\bz\wedge\gz),\fz]$, and $q\in(0,\fz]$,
then the \emph{homogeneous Besov space $\dot B^s_{p,q}(X)$} is defined by setting
$$
\dot B^s_{p,q}(X):=\lf\{f\in\lf(\GOO{\bz,\gz}\r)':\ \|f\|_{\dot B^s_{p,q}(X)}:=
\lf[\sum_{k=-\fz}^\fz\dz^{-ksq}\|Q_kf\|_{L^p(X)}^q\r]^{1/q}<\fz\r\}
$$
with the usual modifications made when $p=\fz$ or $q=\fz$.

\item If $s\in(-(\bz\wedge\gz),\bz\wedge\gz)$, $p\in(p(s,\bz\wedge\gz),\fz)$, and
$q\in(p(s,\bz\wedge\gz),\fz]$, then the \emph{homogeneous Triebel--Lizorkin space $\dot F^s_{p,q}(X)$}
is defined by setting
$$
\dot F^s_{p,q}(X):=\lf\{f\in\lf(\GOO{\bz,\gz}\r)':\ \|f\|_{\dot F^s_{p,q}(X)}:=
\lf\|\lf(\sum_{k=-\fz}^\fz\dz^{-ksq}|Q_kf|^q\r)^{1/q}\r\|_{L^p(X)}<\fz\r\}
$$
with the usual modification made when $q=\fz$.
\end{enumerate}
\end{definition}

It was proved in \cite{whhy20}  that these spaces are independent of the choices of
$\bz$ and $\gz$ as in \eqref{eq-bzgz}, and exp-ATIs (see \cite[Remark 3.13]{whhy20}), which makes Definition \ref{def-btl}
well defined.

Now, we state our main results in this section, namely, the
wavelet characterization of $\dot B^s_{p,q}(X)$ and $\dot F^s_{p,q}(X)$.
In what follows, for any dyadic cube $Q$,
let $\wz{\mathbf 1}_{Q}:=\mathbf 1_{Q}/[\mu(Q)]^{1/2}$.

\begin{theorem}\label{thm-wavebtl}
Let $s\in(-\eta,\eta)$ with $\eta$ as in Definition \ref{def-eti}. The following two statements hold true.
\begin{enumerate}
\item If $s$, $p$, $q$, $\bz$, and $\gz$ are as in Definition \ref{def-btl}(i),
then $f\in\dot B^s_{p,q}(X)$ if and only if $f\in(\GOO{\bz,\gz})'$ and
$$
\|f\|_{\dot B^s_{p,q}(\mathrm{w},X)}:=\lf\{\sum_{k\in\zz}\dz^{-ksq}
\lf[\sum_{\az\in\CG_k}\lf[\mu\lf(Q_\az^{k+1}\r)\r]^{1-p/2}\lf|\lf<f,\psi_\az^k\r>\r|^p
\r]^{q/p}\r\}^{1/q}<\fz
$$
with the usual modifications made when $p=\fz$ or $q=\fz$.
Moreover, there exists a constant $C\in[1,\fz)$ such that, for any $f\in(\GOO{\bz,\gz})'$,
$C^{-1}\|f\|_{\dot B^s_{p,q}(X)}\le\|f\|_{\dot B^s_{p,q}(\mathrm{w},X)}\le C\|f\|_{\dot B^s_{p,q}(X)}$.

\item If $s$, $p$, $q$, $\bz$, and $\gz$ are as in Definition \ref{def-btl}(ii),
then $f\in\dot F^s_{p,q}(X)$ if and only if $f\in(\GOO{\bz,\gz})'$ and
$$
\|f\|_{\dot F^s_{p,q}(\mathrm{w},X)}:=\lf\|\sum_{k\in\zz}\dz^{-ksq}\lf(\sum_{\az\in\CG_k}\lf|
\lf<f,\psi_\az^k\r>\wz{\mathbf 1}_{Q_\az^{k+1}}\r|^q\r)^{1/q}\r\|_{L^p(X)}<\fz
$$
with the usual modification made when $q=\fz$.
Moreover, there exists a constant $C\in[1,\fz)$ such that, for any $f\in(\GOO{\bz,\gz})'$,
$C^{-1}\|f\|_{\dot F^s_{p,q}(X)}\le\|f\|_{\dot F^s_{p,q}(\mathrm{w},X)}\le C\|f\|_{\dot F^s_{p,q}(X)}$.
\end{enumerate}
\end{theorem}

\begin{proof}
Due to similarity, we only prove (ii). To this end, let $s$, $p$, and $q$ be as in (ii), and
$\omega\in[\omega_0,\fz)$
satisfy \eqref{eq-doub} and all the assumptions of (ii) with $\omega_0$
replaced by $\omega$. We first show the sufficiency of (ii).
For this purpose, let $f\in(\GOO{\bz,\gz})'$ with $\bz$
and $\gz$ as in (ii), and $\|f\|_{\dot F^s_{p,q}(\mathrm{w},X)}<\fz$ with
$s$, $p$, and $q$ as in (ii). Note that, by \cite[Proposition 3.12]{whhy20}, we find that $\dot B^s_{p,q}(X)$ is
independent of the choice of exp-ATIs. On the other hand, if we define, for any $k\in\zz$ and $x,\ y\in X$,
\begin{equation}\label{3.1x}
D_k(x,y):=\sum_{\az\in\CG_k}\psi_\az^k(x)\psi_\az^k(y)
\end{equation}
pointwisely, then $\{D_k\}_{k=-\fz}^\fz$ is an exp-ATI (see, for instance, \cite[Remark 2.9(i)]{hlyy19}). Thus, we may choose
$\{Q_k\}_{k=-\fz}^\fz$ in Definition \ref{def-btl} just to be $\{D_k\}_{k=-\fz}^\fz$
in \eqref{3.1x}. By the orthonormality of
$\{D_k\}_{k=-\fz}^\fz$ (see, for instance, \cite[Theorem 7.1]{ah13}), we find that, for any $k\in\zz$ and $x\in X$,
$$
D_{k}f(x)=\sum_{j=-\fz}^\fz\sum_{\az\in\CG_j}\lf<f,\psi_\az^{j}\r>D_k\psi_\az^{j}(x)=\sum_{\az\in\CG_k}
\lf<f,\psi_\az^k\r>\psi_\az^k(x).
$$
From this and Lemma \ref{lem-btl}, we deduce that,
for any $k\in\zz$ and $x\in X$,
\begin{align}\label{eq-wavetl1}
\dz^{-ksq}|D_{k}f(x)|^q&\ls\dz^{-ksq}\lf[\sum_{\az\in\CG_{k}}\lf[\mu\lf(Q_\az^{k+1}\r)\r]^{-1/2}
\lf|\lf<f,\psi_\az^k\r>\r|\exp\lf\{-\nu\lf[\frac{d(x,y_\az^k)}{\dz^k}\r]^a\r\}\r]^q\noz\\
&\ls\dz^{-ksq}\lf\{\sum_{\az\in\CG_k}\mu\lf(Q_\az^{k+1}\r)P_\Gamma\lf(x,y_\az^k;\dz^k\r)\lf[\mu\lf(Q_\az^{k+1}\r)\r]^{-1/2}
\lf|\lf<f,\psi_\az^k\r>\r|\r\}^q\noz\\
&\ls\lf[\CM\lf(\dz^{-ksr}\sum_{\az\in\CG_k}\lf|\lf<f,\psi_\az^k\r>\wz{\mathbf 1}_{Q_\az^{k+1}}\r|^r\r)(x)\r]^{q/r},
\end{align}
where $\Gamma\in(0,\fz)$ is determined later and $r\in(\omega/(\omega+\Gamma),1]$.
Choose $\Gamma$ sufficiently large such that
$r\in(0,\min\{p,q\})$. Then, from \cite[(3.21)]{whhy20}, \eqref{eq-wavetl1}, and Lemma \ref{lem-fs}, we deduce that
\begin{align*}
\|f\|_{\dot F^s_{p,q}(X)}&\sim\lf\|\lf(\sum_{k\in\zz}\dz^{ksq}|D_kf|^q\r)^{1/q}\r\|_{L^p(X)}
\ls\lf\|\lf\{\sum_{k\in\zz}\lf[\CM\lf(\dz^{-ksr}\sum_{\az\in\CG_{k}}\lf|\lf<f,\psi_\az^{k}\r>
\wz{\mathbf 1}_{Q_\az^{k+1}}\r|^r\r)\r]^{q/r}\r\}^{1/q}\r\|_{L^p(X)}\\
&\ls\lf\|\lf[\sum_{k\in\zz}\lf(\dz^{-ksr}\sum_{\az\in\CG_{k}}\lf|\lf<f,\psi_\az^{k}\r>
\wz{\mathbf 1}_{Q_\az^{k+1}}\r|^r\r)^{q/r}\r]^{r/q}\r\|_{L^{p/r}(X)}^{1/r}\sim \|f\|_{\dot F^s_{p,q}(\mathrm{w},X)}.
\end{align*}
This finishes the proof of the sufficiency of (ii).

Next, we prove the necessity of (ii). To achieve this, let $f\in\dot F^s_{p,q}(X)$. By Lemma \ref{lem-hdrf}, we find that
$$
f(\cdot)=\sum_{k\in\zz}\sum_{\az\in\CA_k}\sum_{m=1}^{N(k,\az)}\mu\lf(Q_\az^{k,m}\r)
\wz Q_k\lf(\cdot,y_\az^{k,m}\r)Q_kf\lf(y_\az^{k,m}\r)
$$
in $(\GOO{\bz,\gz})'$ with $\bz$ and $\gz$ as in (ii),
where $\{\wz Q_k\}_{k\in\zz}$ satisfies (i), (ii), and (iii) of Lemma \ref{lem-hdrf}.
Therefore, for any $k_0\in\zz$ and $\az_0\in\CG_{k_0}$, we have
$$
\lf<f,\psi_{\az_0}^{k_0}\r>=\sum_{k\in\zz}\sum_{\az\in\CA_k}\sum_{m=1}^{N(k,\az)}\mu\lf(Q_\az^{k,m}\r)
\lf<\wz Q_k\lf(\cdot,y_{\az}^{k,m}\r),\psi_{\az_0}^{k_0}\r>Q_kf\lf(y_\az^{k,m}\r).
$$
Using an argument similar to that used in the estimations of \cite[(3.29) and (3.30)]{whhy20},
we find that, for any $k,\ k_0\in\zz$, $\az\in\CA_{k}$, $m\in\{1,\ldots,N(k,\az)\}$, and $\az_0\in\CG_{k_0}$,
\begin{equation*}
\lf|\lf<\wz Q_k\lf(\cdot,y_{\az}^{k,m}\r),\psi_{\az_0}^{k_0}\r>\r|
\ls\dz^{(k-k_0)\bz}\lf[\mu\lf(Q_{\az_0}^{k_0+1}\r)\r]^{1/2}P_\gz\lf(y_\az^{k,m},y_{\az_0}^{k_0};\dz^{k_0}\r)
\end{equation*}
when $k\ge k_0$, and
$$
\lf|\lf<\wz Q_k\lf(\cdot,y_{\az}^{k,m}\r),\psi_{\az_0}^{k_0}\r>\r|
\ls\dz^{(k_0-k)\gz}\lf[\mu\lf(Q_{\az_0}^{k_0+1}\r)\r]^{1/2}P_\gz\lf(y_\az^{k,m},y_{\az_0}^{k_0};\dz^{k}\r)
$$
when $k<k_0$.
By these estimates, we find that, for any $k_0\in\zz$ and $x\in X$,
\begin{align*}
&\sum_{\az_0\in\CG_{k_0}}\lf|\lf<f,\psi_{\az_0}^{k_0}\r>\wz{\mathbf 1}_{Q_{\az_0}^{k_0+1}}(x)\r|\\
&\quad\ls\sum_{k=k_0}^\fz\dz^{(k-k_0)\bz}\sum_{\az\in\CA_k}\sum_{m=1}^{N(k,\az)}P_\gz\lf(y_\az^{k,m},x;\dz^{k_0}\r)
\mu\lf(Q_\az^{k,m}\r)\lf|Q_kf\lf(y_\az^{k,m}\r)\r|\\
&\qquad+\sum_{k=-\fz}^{k_0-1}\dz^{(k_0-k)\gz}\sum_{\az\in\CA_k}\sum_{m=1}^{N(k,\az)}
\mu\lf(Q_\az^{k,m}\r)P_\gz\lf(y_\az^{k,m},x;\dz^{k}\r)\lf|Q_kf\lf(y_\az^{k,m}\r)\r|.
\end{align*}
Thus, using \cite[Lemma 3.7]{whhy20}, we conclude that, for any $k_0\in\zz$,
\begin{align*}
&\dz^{-k_0sq}\sum_{\az_0\in\CG_{k_0}}\lf|\lf<f,\psi_{\az_0}^{k_0}\r>\wz{\mathbf 1}_{Q_{\az_0}^{k_0+1}}\r|^q\\
&\quad\ls\lf\{\sum_{k=k_0}^\fz\dz^{(k-k_0)[\bz+s-\omega(1/r-1)]}\lf[\dz^{-ksr}
\CM\lf(\sum_{\az\in\CA_k}\sum_{m=1}^{N(k,\az)}
\lf|Q_kf\lf(y_\az^{k,m}\r)\r|^r\mathbf 1_{Q_\az^{k,m}}\r)\r]^{1/r}\r\}^q\\
&\quad\qquad+\lf\{\sum_{k=-\fz}^{k_0-1}\dz^{(k_0-k)(\gz-s)}\lf[\dz^{-ksr}
\CM\lf(\sum_{\az\in\CA_k}\sum_{m=1}^{N(k,\az)}
\lf|Q_kf\lf(y_\az^{k,m}\r)\r|^r\mathbf 1_{Q_\az^{k,m}}\r)\r]^{1/r}\r\}^q,
\end{align*}
where $r\in(\omega/[\omega+\gz],1]$ is determined later.
Since $\min\{p,q\}>p(s,\bz\wedge\gz)$,
we may choose $r\in(\omega/[\omega+\gz],\min\{p,q,1\})$
such that $\omega(1/r-1)<\bz+s$. Thus, from the H\"{o}lder inequality when
$q\in(1,\fz]$, or Lemma \ref{lem-pin} when $q\in(p(s,\bz\wedge\gz),1]$,
and $\gz>s$, we further deduce that
\begin{equation*}
\sum_{k_0\in\zz}\dz^{-k_0sq}\sum_{\az_0\in\CG_{k_0}}\lf|\lf<f,\psi_{\az_0}^{k_0}\r>
\widetilde{\mathbf 1}_{Q_{\az_0}^{k_0+1}}\r|^q
\ls\sum_{k=-\fz}^{\fz}\lf[\dz^{-ksr}
\CM\lf(\sum_{\az\in\CA_k}\sum_{m=1}^{N(k,\az)}
\lf|Q_kf\lf(y_\az^{k,m}\r)\r|^r\mathbf 1_{Q_\az^{k,m}}\r)\r]^{q/r}.
\end{equation*}
By this and Lemma \ref{lem-fs}, we obtain
\begin{align*}
\|f\|_{\dot F^s_{p,q}(\mathrm{w},X)}&\ls\lf\|\lf\{\sum_{k=-\fz}^{\fz}\lf[\dz^{-ksr}\CM\lf(\sum_{\az\in\CA_k}\sum_{m=1}^{N(k,\az)}
\lf|Q_kf\lf(y_\az^{k,m}\r)\r|^r\mathbf 1_{Q_\az^{k,m}}\r)\r]^{q/r}\r\}^{r/q}\r\|_{L^{p/r}(X)}^{1/r}\\
&\ls\lf\|\lf[\sum_{k=-\fz}^{\fz}\dz^{-ksq}\sum_{\az\in\CA_k}\sum_{m=1}^{N(k,\az)}
\lf|Q_kf\lf(y_\az^{k,m}\r)\r|^q\mathbf 1_{Q_\az^{k,m}}\r]^{1/q}\r\|_{L^{p}(X)}^{1/r}.
\end{align*}
Finally, by the arbitrariness of $y_\az^{k,m}$, we further conclude that
$\|f\|_{\dot F^s_{p,q}(\mathrm{w},X)}\ls\|f\|_{\dot F^s_{p,q}(X)}$.
This finishes the proof of the necessity of (ii) and hence of Theorem \ref{thm-wavebtl}.
\end{proof}

\section{Boundedness of almost diagonal operators on sequence spaces}
\label{s-ado}

In this section, we mainly consider the boundedness of almost diagonal operators on sequence spaces.
Through this section, we always assume $\mu(X)=\fz$, which, by \cite[Lemma 5.1]{ny97} or
\cite[Lemma 8.1]{ah13}, is equivalent to $\diam X=\fz$.

In what follows, let
\begin{equation}\label{4.1x}
\wt{\mathcal D}:=\lf\{Q_\az^{k+1}:\ k\in\zz,\ \az\in\CG_k\r\}.
\end{equation}
Using this, we now introduce the following notion of homogeneous Besov and Triebel--Lizorkin sequence spaces.
\begin{definition}
Let $s\in\rr$ and $q\in (0,\fz]$.
\begin{enumerate}
\item[(i)] If $p\in(0,\fz]$, then the \emph{homogeneous Besov sequence space $\dot b^s_{p,q}$} is defined to be
the set of all sequences $\lz:=\{\lz_Q\}_{Q\in\wD}\subset \cc$ such that
$$
\|\lz\|_{\dot b^s_{p,q}}:=\lf[\sum_{k\in\zz}\dz^{-ksq}\lf\{\sum_{\az\in\CG_k}
\lf[\mu\lf(Q_\az^{k+1}\r)\r]^{1-p/2}\lf|\lz_{Q_\az^{k+1}}\r|^p\r\}^{q/p}\r]^{1/q}<\fz
$$
with usual modifications made when $p=\fz$ or $q=\fz$.

\item[(ii)] If $p\in(0,\fz)$, then the \emph{homogeneous Triebel--Lizorkin sequence space
$\dot f^s_{p,q}$} is defined to be the set of all sequences
$\lz:=\{\lz_Q\}_{Q\in\wD}\subset \cc$ such that
$$
\|\lz\|_{\dot f^s_{p,q}}:=\lf\|\lf(\sum_{k\in\zz}\dz^{-ksq}\sum_{\az\in\CG_k}\lf|\lz_{Q_\az^{k+1}}
\wz{\mathbf 1}_{Q_\az^{k+1}}\r|^q\r)^{1/q}\r\|_{L^p(X)}<\fz
$$
with the usual modification made when $q=\fz$. Recall that, for any $k\in\zz$ and $\az\in\CG_k$,
$$
\wz{\mathbf 1}_{Q_\az^{k+1}}:=\lf[\mu\lf(Q_\az^{k+1}\r)\r]^{-1/2}\mathbf 1_{Q_\az^{k+1}}.
$$
\end{enumerate}
\end{definition}

Next, we introduce the notion of almost diagonal operators. To this end, we first introduce
some notation. Let $A:=\{A_{Q,P}\}_{Q,\ P\in\wD}\subset\cc$. For
any sequence $\lz:=\{\lz_P\}_{P\in\wD}\subset\cc$, define $A\lz:=\{(A\lz)_Q\}_{Q\in\wD}$ by setting, for any $Q\in\wD$,
$$
(A\lz)_Q:=\sum_{P\in\wD} A_{Q,P}\lz_P
$$
if, for any $Q\in\wD$, the above summation converges. Let $Q\in\wD$ be such that $Q:=Q_\az^{k+1}$ for some
$k,\ l\in\zz$ and $\az\in\CG_k$. Define the ``\emph{center}'' $x_Q$ of $Q$ by setting $x_Q:=y_\az^k=z_\az^{k+1}$ and
the ``\emph{side-length}'' $\ell(Q)$ of $Q$ by setting $\ell(Q):=\dz^{k+1}$.

Now, we introduce the notion of homogeneous almost diagonal operators on homogeneous Besov and Triebel--Lizorkin
sequence spaces.

\begin{definition}\label{def-ado}
Let $A:=\{A_{Q,P}\}_{Q,\ P\in\wD}\subset\cc$ and $\omega_0$ be as in \eqref{eq-updim}.
\begin{enumerate}
\item The operator $A$ is called an \emph{almost diagonal operator on $\dot b^s_{p,q}$}, with $s\in\rr$ and
$p,\ q\in(0,\fz]$, if there exist an $\ez\in(0,\fz)$ and an $\omega\in[\omega_0,\fz)$ satisfying \eqref{eq-doub} such that
\begin{equation}\label{eq-defado}
K:=\sup_{Q,\ P\in\wD} \frac{|A_{Q,P}|}{\mathfrak M_{Q,P}(\ez)}<\fz,
\end{equation}
where, for any $Q,\ P\in\wD$,
\begin{align}\label{eq-defopq}
\mathfrak M_{Q,P}(\ez)&:=\lf[\frac{\ell(Q)}{\ell(P)}\r]^s[\mu(Q)\mu(P)]^{1/2}P_{\ez+J-\omega}(x_Q,x_P;\max\{\ell(Q),\ell(P)\})
\noz\\
&\qquad\times\min\lf\{\lf[\frac{\ell(Q)}{\ell(P)}\r]^{\ez/2},\lf[\frac{\ell(P)}{\ell(Q)}\r]^{\ez/2+J-\omega}\r\}
\end{align}
with $J:=\omega/\min\{1,p\}$.
\item The operator $A$ is called an \emph{almost diagonal operator on $\dot f^s_{p,q}$}, with $s\in\rr$,
$p\in(0,\fz)$, and $q\in(0,\fz]$, if there exist an
$\ez\in(0,\fz)$ and an $\omega\in[\omega_0,\fz)$ satisfying \eqref{eq-doub}
such that \eqref{eq-defado} holds true,
where, for any $Q,\ P\in\wD$, $\mathfrak M_{Q,P}(\ez)$ is as 
in \eqref{eq-defopq} with $J:=\omega/\min\{1,p,q\}$.
\end{enumerate}
\end{definition}

On the boundedness of almost diagonal operators on Besov and Triebel--Lizorkin
sequence spaces, we have the following conclusion.

\begin{theorem}\label{thm-ado}
Let $s\in\rr$, $p\in(0,\fz]$ [resp., $p\in(0,\fz)$], $q\in(0,\fz]$, and $A:=\{A_{Q,P}\}_{Q,\ P\in\wD}$ be an almost
diagonal operator on $\dot b^s_{p,q}$ (resp., $\dot f^s_{p,q}$). Then $A$ is bounded on $\dot b^s_{p,q}$
(resp., $\dot f^s_{p,q}$). Moreover, there exists a positive constant $C$,
independent of $A$,  such that, for any $\lz\in\dot b^s_{p,q}$
(resp., $\lz\in\dot f^s_{p,q}$), $\|A\lz\|_{\dot b^s_{p,q}}\le CK\|\lz\|_{\dot b^s_{p,q}}$
(resp., $\|A\lz\|_{\dot f^s_{p,q}}\le CK\|\lz\|_{\dot f^s_{p,q}}$).
\end{theorem}

\begin{proof}
Let $A$ be the same as in this theorem, and $\ez$, $\omega$, and $J$ as in Definition \ref{def-ado},
We separate $A$ into the following two parts: For any sequence
$\lz:=\{\lz_P\}_{P\in\wD}$ and any $Q\in\wD$, let
$$
(A_0\lz)_Q:=\sum_{\{P\in\wD:\ \ell(P)\ge\ell(Q)\}} A_{Q,P}\lz_P \quad\text{and}
\quad (A_1\lz)_Q:=\sum_{\{P\in\wD:\ \ell(P)<\ell(Q)\}} A_{Q,P}\lz_P.
$$
To prove this theorem, it suffices to show that $A_0$ and $A_1$ are both
bounded on $\dot b^s_{p,q}$ and $\dot f^s_{p,q}$, respectively, with $s$, $p$,
and $q$ as in this theorem.

We first establish the boundedness of $A_0$ on $\dot b^s_{p,q}$
by considering the following two cases on $p$.

{\it Case 1.1) $p\in(1,\fz]$}. In this case, $J=\omega$.
Let $\lz:=\{\lz_Q\}_{Q\in\wD}\in\dot b^s_{p,q}$. We have, for any $k_0\in\zz$ and $\az_0\in\CG_{k_0}$,
\begin{align*}
\lf|(A_0\lz)_{Q_{\az_0}^{k_0+1}}\r|&\le\sum_{\{P\in\wD:\ \ell(P)\ge\dz^{k_0+1}\}}
\lf|A_{Q_{\az_0}^{k_0+1},P}\r|
|\lz_P|\ls\sum_{k=-\fz}^{k_0}\sum_{\az\in\CG_k}\lf|\mathfrak M_{Q_{\az_0}^{k_0+1},Q_{\az}^{k+1}}(\ez)\r|
\lf|\lz_{Q_{\az}^{k+1}}\r|\\
&\ls\sum_{k=-\fz}^{k_0}\sum_{\az\in\CG_k}\dz^{(k_0-k)s}\dz^{(k_0-k)\ez/2}\lf[\mu\lf(Q_{\az_0}^{k_0+1}\r)
\mu\lf(Q_{\az}^{k+1}\r)\r]^{1/2}P_\ez\lf(y_\az^k,y_{\az_0}^{k_0};\dz^k\r)\lf|\lz_{Q_{\az}^{k+1}}\r|,
\end{align*}
which further implies that, for any $k_0\in\zz$ and $\az_0\in\CG_{k_0}$,
\begin{align}\label{eq-1.1}
&\lf[\mu\lf(Q_{\az_0}^{k_0+1}\r)\r]^{-1/2}\lf|(A_0\lz)_{Q_{\az_0}^{k_0+1}}\r|\noz\\
&\quad\ls\sum_{k=-\fz}^{k_0}\sum_{\az\in\CG_k}\dz^{(k_0-k)s}\dz^{(k_0-k)\ez/2}
\mu\lf(Q_{\az}^{k+1}\r)P_\ez\lf(y_\az^k,y_{\az_0}^{k_0};\dz^k\r)
\lf[\mu\lf(Q_\az^{k+1}\r)\r]^{-1/2}\lf|\lz_{Q_{\az}^{k+1}}\r|.
\end{align}
By this and the H\"{o}lder inequality, we conclude that, for any $k_0\in\zz$,
\begin{align*}
&\sum_{\az_0\in\CG_{k_0}}\lf[\mu\lf(Q_{\az_0}^{k_0+1}\r)\r]^{1-p/2}\lf|(A_0\lz)_{Q_{\az_0}^{k_0+1}}\r|^p\\
&\quad\ls\sum_{\az_0\in\CG_{k_0}}\mu\lf(Q_{\az_0}^{k_0+1}\r)\sum_{k=-\fz}^{k_0}\sum_{\az\in\CG_k}
\dz^{(k_0-k)sp}\dz^{(k_0-k)\ez/2}P_\ez\lf(y_\az^k,y_{\az_0}^{k_0};\dz^k\r)\lf[\mu\lf(Q_\az^{k+1}\r)\r]^{1-p/2}
\lf|\lz_{Q_{\az}^{k+1}}\r|^p\\
&\quad\ls\sum_{k=-\fz}^{k_0}\dz^{(k_0-k)sp}\dz^{(k_0-k)\ez/2}\sum_{\az\in\CG_k}
\lf[\mu\lf(Q_\az^{k+1}\r)\r]^{1-p/2}\lf|\lz_{Q_{\az}^{k+1}}\r|^p.
\end{align*}
Using this and the H\"{o}lder inequality when $q/p\in(1,\fz]$, or Lemma \ref{lem-pin} when
$q/p\in(0,1]$, we obtain
\begin{align*}
&\sum_{k_0=-\fz}^\fz\dz^{-k_0sq}\lf\{\sum_{\az_0\in\CG_{k_0}}\lf[\mu\lf(Q_{\az_0}^{k_0+1}\r)\r]^{1-p/2}
\lf|(A_0\lz)_{Q_{\az_0}^{k_0+1}}\r|^p\r\}^{q/p}\\
&\quad\ls\sum_{k_0=-\fz}^\fz\sum_{k=-\fz}^{k_0}\dz^{-ksq}\dz^{(k_0-k)\ez\min\{q/p,1\}/2}
\lf\{\sum_{\az\in\CG_k}\lf[\mu\lf(Q_\az^{k+1}\r)\r]^{1-p/2}\lf|\lz_{Q_{\az}^{k+1}}\r|^p\r\}^{q/p}\\
&\quad\ls\sum_{k=-\fz}^{\fz}\dz^{-ksq}\lf\{\sum_{\az\in\CG_k}\lf[\mu\lf(Q_\az^{k+1}\r)\r]^{1-p/2}
\lf|\lz_{Q_{\az}^{k+1}}\r|^p\r\}^{q/p}.
\end{align*}
Thus, we conclude that $\|A_0\lz\|_{\dot b^s_{p,q}}\ls\|\lz\|_{\dot b^s_{p,q}}$, which is the desired
conclusion in this case.

{\it Case 1.2) $p\in(0,1]$}. In this case, $J=\omega/p$. Let $\lz:=\{\lz_Q\}_{Q\in\wD}\in\dot b^s_{p,q}$.
We conclude that, for any $k_0\in\zz$ and $\az_0\in\CG_{k_0}$,
\begin{align*}
\lf|(A_0\lz)_{Q_{\az_0}^{k_0+1}}\r|
&\ls\sum_{k=-\fz}^{k_0}\sum_{\az\in\CG_k}\dz^{(k_0-k)s}\dz^{(k_0-k)\ez/2}
\lf[\mu\lf(Q_{\az_0}^{k_0+1}\r)\mu\lf(Q_{\az}^{k+1}\r)\r]^{1/2}\\
&\quad\times P_{\ez+\omega(1/p-1)}
\lf(y_\az^k,y_{\az_0}^{k_0};\dz^k\r)\lf|\lz_{Q_{\az}^{k+1}}\r|.
\end{align*}
Thus, we find that, for any $k_0\in\zz$ and $\az_0\in\CG_{k_0}$,
\begin{align}\label{eq-1.2}
&\lf[\mu\lf(Q_{\az_0}^{k_0+1}\r)\r]^{-1/2}\lf|(A_0\lz)_{Q_{\az_0}^{k_0+1}}\r|\noz\\
&\quad\ls\sum_{k=-\fz}^{k_0}\sum_{\az\in\CG_k}\dz^{(k_0-k)s}\dz^{(k_0-k)\ez/2}
\mu\lf(Q_{\az}^{k+1}\r)P_{\ez+\omega(1/p-1)}\lf(y_\az^k,y_{\az_0}^{k_0};\dz^{k}\r)
\lf[\mu\lf(Q_\az^{k+1}\r)\r]^{-1/2}\lf|\lz_{Q_{\az}^{k+1}}\r|.
\end{align}
From this, Lemmas \ref{lem-pin} and \ref{lem-bbes}, we further deduce that, for any $k_0\in\zz$,
\begin{align*}
&\sum_{\az_0\in\CG_{k_0}}\lf[\mu\lf(Q_{\az_0}^{k_0+1}\r)\r]^{1-p/2}\lf|(A_0\lz)_{Q_{\az_0}^{k_0+1}}\r|^p\\
&\quad\ls\sum_{k=-\fz}^{k_0}\sum_{\az\in\CG_k}\dz^{(k_0-k)sp}
\dz^{(k_0-k)\ez p/2}\sum_{\az_0\in\CG_{k_0}}\mu\lf(Q_{\az_0}^{k_0+1}\r)
\lf[P_{\ez+\omega(1/p-1)}\lf(y_\az^k,y_{\az_0}^{k_0};\dz^{k}\r)\r]^p\\
&\quad\qquad\times\lf[\mu\lf(Q_\az^{k+1}\r)\r]^{p-p/2}\lf|\lz_{Q_{\az}^{k+1}}\r|^p\\
&\quad\ls\sum_{k=-\fz}^{k_0}\dz^{(k_0-k)sp}\dz^{(k_0-k)\ez p/2}\sum_{\az\in\CG_k}
\lf[\mu\lf(Q_\az^{k+1}\r)\r]^{1-p/2}\lf|\lz_{Q_{\az}^{k+1}}\r|^p,
\end{align*}
where we used the fact that $\omega/[\omega+\omega(1/p-1)+\ez]<p$. Using an argument similar to that used in the
estimation of Case 1.1), we find that
\begin{align*}
&\sum_{k_0=-\fz}^\fz\dz^{-k_0sq}\lf\{\sum_{\az_0\in\CG_{k_0}}\lf[\mu\lf(Q_{\az_0}^{k_0+1}\r)\r]^{1-p/2}
\lf|(A_0\lz)_{Q_{\az_0}^{k_0+1}}\r|^p\r\}^{q/p}\\
&\quad\ls\sum_{k_0=-\fz}^\fz\sum_{k=-\fz}^{k_0}\dz^{-ksq}\dz^{(k_0-k)\ez\min\{q,p\}/2}
\lf\{\sum_{\az\in\CG_k}\lf[\mu\lf(Q_\az^{k+1}\r)\r]^{1-p/2}\lf|\lz_{Q_{\az}^{k+1}}\r|^p\r\}^{q/p}\\
&\quad\ls\sum_{k=-\fz}^{\fz}\dz^{-ksq}\lf\{\sum_{\az\in\CG_k}\lf[\mu\lf(Q_\az^{k+1}\r)\r]^{1-p/2}
\lf|\lz_{Q_{\az}^{k+1}}\r|^p\r\}^{q/p},
\end{align*}
which further implies that $\|A_0\lz\|_{\dot b^s_{p,q}}\ls\|\lz\|_{\dot b^s_{p,q}}$.
This is also the desired conclusion in this
case, which, combined with the
conclusion in Case 1.1), then completes the proof of the boundedness of $A_0$ on $\dot b^s_{p,q}$.

Now, we establish the boundedness of $A_1$ on $\dot b^s_{p,q}$
also by considering the following two cases on $p$.

{\it Case 2.1) $p\in(1,\fz]$.} In this case, $J=\omega$. Let $\lz:=\{\lz_Q\}_{Q\in\wD}\in\dot b^s_{p,q}$.
We conclude that, for any $k_0\in\zz$ and $\az_0\in\CG_{k_0}$,
\begin{align*}
\lf|(A_1\lz)_{Q_{\az_0}^{k_0+1}}\r|&\le\sum_{\{P\in\wD:\ \ell(P)<\dz^{k_0+1}\}}
\lf|A_{Q_{\az_0}^{k_0+1},P}\r|
|\lz_P|\ls\sum_{k=k_0+1}^{\fz}\sum_{\az\in\CG_k}\lf|\omega_{Q_{\az_0}^{k_0+1},Q_{\az}^{k+1}}(\ez)\r|
\lf|\lz_{Q_{\az}^{k+1}}\r|\\
&\ls\sum_{k=k_0+1}^\fz\sum_{\az\in\CG_k}\dz^{(k_0-k)s}\dz^{(k-k_0)\ez/2}\lf[\mu\lf(Q_{\az_0}^{k_0+1}\r)
\mu\lf(Q_{\az}^{k+1}\r)\r]^{1/2}P_\ez\lf(y_{\az_0}^{k_0},y_\az^k;\dz^{k_0}\r)\lf|\lz_{Q_{\az}^{k+1}}\r|.
\end{align*}
Therefore, we have, for any $k_0\in\zz$ and $\az_0\in\CG_{k_0}$,
\begin{align}\label{eq-2.1}
&\lf[\mu\lf(Q_{\az_0}^{k_0+1}\r)\r]^{-1/2}\lf|(A_1\lz)_{Q_{\az_0}^{k_0+1}}\r|\noz\\
&\quad\ls\sum_{k=k_0+1}^{\fz}\sum_{\az\in\CG_k}\dz^{(k_0-k)s}\dz^{(k-k_0)\ez/2}
\mu\lf(Q_{\az}^{k+1}\r) P_\ez\lf(y_{\az_0}^{k_0},y_\az^k;\dz^{k_0}\r)
\lf[\mu\lf(Q_\az^{k+1}\r)\r]^{-1/2}\lf|\lz_{Q_{\az}^{k+1}}\r|.
\end{align}
By this and the H\"{o}lder inequality, we conclude that, for any $k_0\in\zz$,
\begin{align*}
&\sum_{\az_0\in\CG_{k_0}}\lf[\mu\lf(Q_{\az_0}^{k_0+1}\r)\r]^{1-p/2}\lf|(A_0\lz)_{Q_{\az_0}^{k_0+1}}\r|^p\\
&\quad\ls\sum_{\az_0\in\CG_{k_0}}\mu\lf(Q_{\az_0}^{k_0+1}\r)\sum_{k=k_0+1}^{\fz}\sum_{\az\in\CG_k}
\dz^{(k_0-k)sp}\dz^{(k-k_0)\ez/2}P_\ez\lf(y_{\az_0}^{k_0},y_\az^k;\dz^{k_0}\r)
\lf[\mu\lf(Q_\az^{k+1}\r)\r]^{1-p/2}\lf|\lz_{Q_{\az}^{k+1}}\r|^p\\
&\quad\ls\sum_{k=k_0+1}^{\fz}\dz^{(k_0-k)sp}\dz^{(k-k_0)\ez/2}\sum_{\az\in\CG_k}
\lf[\mu\lf(Q_\az^{k+1}\r)\r]^{1-p/2}\lf|\lz_{Q_{\az}^{k+1}}\r|^p.
\end{align*}
This, together with the H\"{o}lder inequality when $q/p\in(1,\fz]$, or Lemma \ref{lem-pin} when $q/p\in(0,1]$,
further implies that
\begin{align*}
&\sum_{k_0=-\fz}^\fz\dz^{-k_0sq}\lf\{\sum_{\az_0\in\CG_{k_0}}\lf[\mu\lf(Q_{\az_0}^{k_0+1}\r)\r]^{1-p/2}
\lf|(A_1\lz)_{Q_{\az_0}^{k_0+1}}\r|^p\r\}^{q/p}\\
&\quad\ls\sum_{k_0=-\fz}^\fz\sum_{k=k_0+1}^{\fz}\dz^{-ksq}\dz^{(k-k_0)\ez\min\{q/p,1\}/2}
\lf\{\sum_{\az\in\CG_k}\lf[\mu\lf(Q_\az^{k+1}\r)\r]^{1-p/2}\lf|\lz_{Q_{\az}^{k+1}}\r|^p\r\}^{q/p}\\
&\quad\ls\sum_{k=-\fz}^{\fz}\dz^{-ksq}\lf\{\sum_{\az\in\CG_k}\lf[\mu\lf(Q_\az^{k+1}\r)\r]^{1-p/2}
\lf|\lz_{Q_{\az}^{k+1}}\r|^p\r\}^{q/p}.
\end{align*}
Thus, we conclude that $\|A_1\lz\|_{\dot b^s_{p,q}}\ls\|\lz\|_{\dot b^s_{p,q}}$, which is the desired
conclusion in this case.

{\it Case 2.2) $p\in(0,1]$}. In this case, $J=\omega/p$. Let $\lz:=\{\lz_Q\}_{Q\in\wD}\in\dot b^s_{p,q}$.
We have, for any $k_0\in\zz$ and $\az_0\in\CG_{k_0}$,
\begin{align*}
\lf|(A_1\lz)_{Q_{\az_0}^{k_0+1}}\r|
&\ls\sum_{k=k_0+1}^{\fz}\sum_{\az\in\CG_k}\lf|\omega_{Q_{\az_0}^{k_0+1},Q_{\az}^{k+1}}(\ez)\r|
\lf|\lz_{Q_{\az}^{k+1}}\r|\\
&\ls\sum_{k=k_0+1}^{\fz}\sum_{\az\in\CG_k}\dz^{(k_0-k)s}\dz^{(k-k_0)[\ez/2+\omega(1/p-1)]}
\lf[\mu\lf(Q_{\az_0}^{k_0+1}\r)\mu\lf(Q_{\az}^{k+1}\r)\r]^{1/2}\\
&\qquad\times P_{\ez+\omega(1/p-1)}\lf(y_{\az_0}^{k_0},y_\az^k;\dz^{k_0}\r)\lf|\lz_{Q_{\az}^{k+1}}\r|.
\end{align*}
Then we conclude that, for any $k_0\in\zz$ and $\az_0\in\CG_{k_0}$,
\begin{align}\label{eq-2.2}
\lf[\mu\lf(Q_{\az_0}^{k_0+1}\r)\r]^{-1/2}\lf|(A_1\lz)_{Q_{\az_0}^{k_0+1}}\r|
&\ls\sum_{k=k_0+1}^{\fz}\sum_{\az\in\CG_k}\dz^{(k_0-k)s}\dz^{(k-k_0)\ez/2}
\mu\lf(Q_{\az}^{k+1}\r)\noz\\
&\qquad\times P_{\ez+\omega(1/p-1)}\lf(y_{\az_0}^{k_0},y_\az^k;\dz^{k_0}\r)\lf[\mu\lf(Q_\az^{k+1}\r)\r]^{-1/2}\lf|\lz_{Q_{\az}^{k+1}}\r|.
\end{align}
This, combined with Lemmas \ref{lem-pin} and \ref{lem-bbes}, further implies that
\begin{align*}
&\sum_{\az_0\in\CG_{k_0}}\lf[\mu\lf(Q_{\az_0}^{k_0+1}\r)\r]^{1-p/2}\lf|(A_1\lz)_{Q_{\az_0}^{k_0+1}}\r|^p\\
&\quad\ls\sum_{k=k_0+1}^{\fz}\sum_{\az\in\CG_k}\dz^{(k_0-k)sp}
\dz^{(k-k_0)\ez p/2}\sum_{\az_0\in\CG_{k_0}}
\mu\lf(Q_{\az_0}^{k_0+1}\r)\\
&\quad\qquad\times \lf[P_{\ez+\omega(1/p-1)}\lf(y_{\az_0}^{k_0},y_\az^k;\dz^{k_0}\r)\r]^p
\lf[\mu\lf(Q_\az^{k+1}\r)\r]^{p-p/2}\lf|\lz_{Q_{\az}^{k+1}}\r|^p\\
&\quad\ls\sum_{k=k_0+1}^{\fz}\dz^{(k_0-k)sp}\dz^{(k-k_0)\ez p/2}\sum_{\az\in\CG_k}
\lf[\mu\lf(Q_\az^{k+1}\r)\r]^{1-p/2}\lf|\lz_{Q_{\az}^{k+1}}\r|^p,
\end{align*}
where we used the fact $\omega/[\omega+\omega(1/p-1)+\ez]<p$. Using an argument similar to that used in the
estimation of Case 1.2), we find that
\begin{align*}
&\sum_{k_0=-\fz}^\fz\dz^{-k_0sq}\lf\{\sum_{\az_0\in\CG_{k_0}}\lf[\mu\lf(Q_{\az_0}^{k_0+1}\r)\r]^{1-p/2}
\lf|(A_0\lz)_{Q_{\az_0}^{k_0+1}}\r|^p\r\}^{q/p}\\
&\quad\ls\sum_{k_0=-\fz}^\fz\sum_{k=k_0+1}^{\fz}\dz^{-ksq}\dz^{(k_0-k)\ez\min\{q,p\}/2}
\lf\{\sum_{\az\in\CG_k}\lf[\mu\lf(Q_\az^{k+1}\r)\r]^{1-p/2}\lf|\lz_{Q_{\az}^{k+1}}\r|^p\r\}^{q/p}\\
&\quad\ls\sum_{k=-\fz}^{\fz}\dz^{-ksq}\lf\{\sum_{\az\in\CG_k}\lf[\mu\lf(Q_\az^{k+1}\r)\r]^{1-p/2}
\lf|\lz_{Q_{\az}^{k+1}}\r|^p\r\}^{q/p},
\end{align*}
which further implies that
$\|A_1\lz\|_{\dot b^s_{p,q}}\ls\|\lz\|_{\dot b^s_{p,q}}$. This is also the desired conclusion in this case,
which, together with the conclusion in Case 2.1), then completes the proof of the boundedness of $A_1$ on
$\dot b^s_{p,q}$.

Combining Cases 1.1), 1.2), 2.1), and 2.2), we conclude that,
for any given $s$, $p$, and $q$ as in this lemma,
and any $\lz\in \dot b^s_{p,q}$, $\|A\lz\|_{\dot b^s_{p,q}}\ls\|\lz\|_{\dot b^s_{p,q}}$,
namely, $A$ is bounded on $\dot b^s_{p,q}$.

We now establish the boundedness of $A$ on $\dot f^s_{p,q}$
by considering following two cases on
$\min\{p,q\}$.

{\it Case 3.1) $\min\{p,q\}>1$.} In this case, $J=\omega$. Let $\lz:=\{\lz_Q\}_{Q\in\wD}\in\dot f^s_{p,q}$.
By \eqref{eq-1.1} and \cite[Proposition 2.2(ii)]{hlyy19}, we find that, for any $k_0\in\zz$
and $x\in X$,
\begin{align*}
&\sum_{\az_0\in\CG_{k_0}}\lf|(A_0\lz)_{Q_{\az_0}^{k_0+1}}\r|\wz{\mathbf 1}_{Q_{\az_0}^{k_0+1}}(x)\\
&\quad\ls\sum_{k=-\fz}^{k_0}\sum_{\az\in\CG_k}\dz^{(k_0-k)s}\dz^{(k_0-k)\ez/2}
\mu\lf(Q_{\az}^{k+1}\r)\sum_{\az_0\in\CG_{k_0}}P_\ez\lf(y_{\az_0}^{k_0},y_\az^k;\dz^k\r)
\lf[\mu\lf(Q_\az^{k+1}\r)\r]^{-1/2}\lf|\lz_{Q_{\az}^{k+1}}\r|\mathbf 1_{Q_{\az_0}^{k_0+1}}(x)\\
&\quad\ls\sum_{k=-\fz}^{k_0}\dz^{(k-k_0)s}\dz^{(k_0-k)\ez/2}
\int_X P_\ez(x,y;\dz^k)\sum_{\az\in\CG_k}\lf|\lz_{Q_{\az}^{k+1}}\r|
\wz{\mathbf 1}_{Q_\az^{k+1}}(y)\,d\mu(y)\\
&\quad\ls\sum_{k=-\fz}^{k_0}\dz^{(k-k_0)s}\dz^{(k_0-k)\ez/2}\CM\lf(\sum_{\az\in\CG_k}
\lf|\lz_{Q_{\az}^{k+1}}\r|\wz{\mathbf 1}_{Q_\az^{k+1}}\r)(x).
\end{align*}
Thus, from the H\"{o}lder inequality, we deduce that
\begin{align*}
\mathrm{L}:=&\,\lf\{\sum_{k_0=-\fz}^\fz\dz^{-k_0sq}\lf[\sum_{\az_0\in\CG_{k_0}}\lf|(A_0\lz)_{Q_{\az_0}^{k_0+1}}\r|
\wz{\mathbf 1}_{Q_{\az_0}^{k_0+1}}\r]^q\r\}^{1/q}\\
\ls&\,\lf\{\sum_{k_0=-\fz}^\fz\dz^{-k_0sq}\lf[\sum_{k=-\fz}^{k_0}\dz^{(k-k_0)s}\dz^{(k_0-k)\ez/2}
\CM\lf(\sum_{\az\in\CG_k}\lf|\lz_{Q_{\az}^{k+1}}\r|\wz{\mathbf 1}_{Q_\az^{k+1}}\r)\r]^q\r\}^{1/q}\\
\ls&\,\lf\{\sum_{k=-\fz}^\fz\dz^{-ksq}
\lf[\CM\lf(\sum_{\az\in\CG_k}\lf|\lz_{Q_{\az}^{k+1}}\r|\wz{\mathbf 1}_{Q_\az^{k+1}}\r)\r]^q\r\}^{1/q}.
\end{align*}
This, together with Lemma \ref{lem-fs}, further implies that
\begin{align*}
\|A_0\lz\|_{\dot f^s_{p,q}}
&=\lf\|\mathrm{L}\r\|_{L^p(X)}
\ls\lf\|\lf\{\sum_{k=-\fz}^\fz\dz^{-ksq}\lf[\CM\lf(\sum_{\az\in\CG_k}\lf|\lz_{Q_{\az}^{k+1}}\r|
\wz{\mathbf 1}_{Q_\az^{k+1}}\r)\r]^q\r\}^{1/q}\r\|_{L^p(X)}\\
&\ls\lf\|\lf[\sum_{k=-\fz}^\fz\dz^{-ksq}\lf(\sum_{\az\in\CG_k}\lf|\lz_{Q_{\az}^{k+1}}\r|
\wz{\mathbf 1}_{Q_\az^{k+1}}\r)^q\r]^{1/q}\r\|_{L^p(X)}\sim\|\lz\|_{\dot f^s_{p,q}}.
\end{align*}
This finishes the proof of Case 3.1).

{\it Case 3.2) $\min\{p,q\}\le 1$.} In this case, $J=\omega/\min\{p,q\}$. Let $\lz:=\{\lz_Q\}_{Q\in\wD}\in\dot f^s_{p,q}$.
By \eqref{eq-1.2} with $p$ therein replaced by $\min\{p,q\}$, we conclude that,
for any fixed $r\in(0,1]$, and any $k\in\zz$,
\begin{align*}
&\sum_{\az_0\in\CG_{k_0}}\lf|(A_0\lz)_{Q_{\az_0}^{k_0+1}}\r|\wz{\mathbf 1}_{Q_{\az_0}^{k_0+1}}\\
&\quad\ls\sum_{k=-\fz}^{k_0}\dz^{(k_0-k)s}\dz^{(k_0-k)\ez/2}\sum_{\az\in\CG_k}
\mu\lf(Q_{\az}^{k+1}\r)P_{\ez+\omega(1/\min\{p,q\}-1)}\lf(y_{\az_0}^{k_0},y_\az^k;\dz^k\r)
\lf[\mu\lf(Q_\az^{k+1}\r)\r]^{-1/2}\lf|\lz_{Q_{\az}^{k+1}}\r|.
\end{align*}
From this and Lemma \ref{lem-btl}, it then follows that, for any $k_0\in\zz$,
\begin{equation*}
\sum_{\az_0\in\CG_{k_0}}\lf|(A_0\lz)_{Q_{\az_0}^{k_0+1}}\r|\wz{\mathbf 1}_{Q_{\az_0}^{k_0+1}}
\ls\sum_{k=-\fz}^{k_0}\dz^{(k_0-k)s}\dz^{(k_0-k)\ez/2}
\lf[\CM\lf(\sum_{\az\in\CG_k}\lf|\lz_{Q_{\az}^{k+1}}\wz{\mathbf 1}_{Q_\az^{k+1}}\r|^r\r)\r]^{1/r},
\end{equation*}
where $r\in(r_{p,q,\ez},1]$ with $r_{p,q,\ez}:={\omega}/[{\omega+\omega(1/\min\{p,q\}-1)+\ez}]$.
Using this and the H\"{o}lder inequality when $q\in(1,\fz]$,
or Lemma \ref{lem-pin} when $q\in(0,1]$, we find that
\begin{align}\label{eq-3.1}
&\lf\{\sum_{k_0=-\fz}^\fz\dz^{-k_0sq}\lf[\sum_{\az_0\in\CG_{k_0}}\lf|(A_0\lz)_{Q_{\az_0}^{k_0+1}}\r|
\wz{\mathbf 1}_{Q_{\az_0}^{k_0+1}}\r]^q\r\}^{1/q}\noz\\
&\quad\lesssim\lf\{\sum_{k=-\fz}^\fz\dz^{-ksq}\lf[\CM\lf(\sum_{\az\in\CG_k}\lf|\lz_{Q_{\az}^{k+1}}
\wz{\mathbf 1}_{Q_\az^{k+1}}\r|^r\r)\r]^{q/r}\r\}^{1/q}.
\end{align}
Due to $r_{p,q,\ez}<\min\{p,q\},$
we may choose $r$ such that $r\in(r_{p,q,\ez},\min\{p,q\})$
and, by this, \eqref{eq-3.1}, and Lemma \ref{lem-fs},
we conclude that
\begin{align*}
\|A_0\lz\|_{\dot f^s_{p,q}}
&\ls\lf\|\lf\{\sum_{k=-\fz}^\fz\dz^{-ksq}\lf[\CM\lf(\sum_{\az\in\CG_k}\lf|\lz_{Q_{\az}^{k+1}}
\wz{\mathbf 1}_{Q_\az^{k+1}}\r|^r\r)\r]^{q/r}\r\}^{1/q}\r\|_{L^p(X)}\\
&\ls\lf\|\lf[\sum_{k=-\fz}^\fz\dz^{-ksq}\lf(\sum_{\az\in\CG_k}\lf|\lz_{Q_{\az}^{k+1}}
\wz{\mathbf 1}_{Q_\az^{k+1}}\r|^r\r)^{q/r}\r]^{r/q}\r\|_{L^{p/r}(X)}^{1/r}\sim\|\lz\|_{\dot f^s_{p,q}}.
\end{align*}
This is the desired conclusion in this case, which, together with the conclusion in Case 3.1),
then completes the proof of the boundedness of $A_0$ on $\dot f^s_{p,q}$.

Finally, we establish the boundedness of $A_1$ on $\dot f^s_{p,q}$ also by
considering the following two cases on
$\min\{p,q\}$.

{\it Case 4.1) $\min\{p,q\}>1$}. In this case, $J=\omega$. Let $\lz:=\{\lz_Q\}_{Q\in\wD}\in\dot f^s_{p,q}$.
By \eqref{eq-2.1} and Lemma \ref{lem-btl}, we find that, for any $k_0\in\zz$,
\begin{align*}
\sum_{\az_0\in\CG_{k_0}}\lf|(A_1\lz)_{Q_{\az_0}^{k_0+1}}\r|\wz{\mathbf 1}_{Q_{\az_0}^{k_0+1}}
&\ls\sum_{k=k_0+1}^{\fz}\sum_{\az\in\CG_k}\dz^{(k_0-k)s}\dz^{(k-k_0)\ez/2}
\mu\lf(Q_{\az}^{k+1}\r)\sum_{\az_0\in\CG_{k_0}}P_\ez\lf(y_{\az_0}^{k_0},y_\az^k;\dz^{k_0}\r)\\
&\qquad\times\lf[\mu\lf(Q_\az^{k+1}\r)\r]^{-1/2}\lf|\lz_{Q_{\az}^{k+1}}\r|\mathbf 1_{Q_{\az_0}^{k_0+1}}\\
&\ls\sum_{k=k_0+1}^\fz\dz^{(k_0-k)s}\dz^{(k-k_0)\ez/2}
\CM\lf(\sum_{\az\in\CG_k}\lf|\lz_{Q_{\az}^{k+1}}\r|\wz{\mathbf 1}_{Q_\az^{k+1}}\r).
\end{align*}
Thus, from the H\"{o}lder inequality, it follows that
\begin{equation*}
\lf\{\sum_{k_0\in\zz}\dz^{k_0sq}\lf[\sum_{\az_0\in\CG_{k_0}}\lf|(A_1\lz)_{Q_{\az_0}^{k_0+1}}\r|
\wz{\mathbf 1}_{Q_{\az_0}^{k_0+1}}\r]^q\r\}^{1/q}
\ls\lf\{\sum_{k=-\fz}^\fz\dz^{-ksq}\lf[\CM\lf(\sum_{\az\in\CG_k}\lf|\lz_{Q_{\az}^{k+1}}
\wz{\mathbf 1}_{Q_\az^{k+1}}\r|\r)\r]^q\r\}^{1/q}.
\end{equation*}
By this and Lemma \ref{lem-fs}, we further conclude that
\begin{align*}
\|A_1\lz\|_{\dot f^s_{p,q}}
&\ls\lf\|\lf\{\sum_{k=-\fz}^\fz\dz^{-ksq}\lf[\CM\lf(\sum_{\az\in\CG_k}\lf|\lz_{Q_{\az}^{k+1}}
\wz{\mathbf 1}_{Q_\az^{k+1}}\r|\r)\r]^q\r\}^{1/q}\r\|_{L^p(X)}\\
&\ls\lf\|\lf[\sum_{k=-\fz}^\fz\dz^{-ksq}\lf(\sum_{\az\in\CG_k}\lf|\lz_{Q_{\az}^{k+1}}
\wz{\mathbf 1}_{Q_\az^{k+1}}\r|\r)^q\r]^{1/q}\r\|_{L^p(X)}\sim\|\lz\|_{\dot f^s_{p,q}}.
\end{align*}
This shows the desired conclusion in this case.

{\it Case 4.2) $\min\{p,q\}\le 1$}. In this case, $J=\omega/\min\{p,q\}$. Let $\lz:=\{\lz_Q\}_{Q\in\wD}\in\dot f^s_{p,q}$.
By \eqref{eq-2.2} with $p$ therein replaced by $\min\{p,q\}$, we find that, for any $k_0\in\zz$,
\begin{align}\label{eq-add}
&\sum_{\az_0\in\CG_{k_0}}\lf|(A_1\lz)_{Q_{\az_0}^{k_0+1}}\r|\wz{\mathbf 1}_{Q_{\az_0}^{k_0+1}}\noz\\
&\quad\ls\sum_{k=k_0+1}^{\fz}\sum_{\az\in\CG_k}\dz^{(k_0-k)s}\dz^{(k-k_0)[\frac {\ez}2+\omega(\frac 1{\min\{p,q\}}-1)]}
\mu\lf(Q_{\az}^{k+1}\r)\sum_{\az_0\in\CG_{k_0}}P_{\ez+
\omega(\frac 1{\min\{p,q\}}-1)}\lf(y_{\az_0}^{k_0},y_\az^k;\dz^{k_0}\r)\noz\\
&\quad\qquad\times\lf[\mu\lf(Q_\az^{k+1}\r)\r]^{-1/2}\lf|\lz_{Q_{\az}^{k+1}}\r|\mathbf 1_{Q_{\az_0}^{k_0+1}}\noz\\
&\quad\ls\sum_{k=k_0+1}^{\fz}\sum_{\az\in\CG_k}\dz^{(k_0-k)s}\dz^{(k-k_0)\ez/2}
\mu\lf(Q_{\az}^{k+1}\r)P_{\ez+\omega(1/\min\{p,q\}-1)}\lf(y_{\az_0}^{k_0},y_\az^k;\dz^{k_0}\r)\noz\\
&\quad\qquad\times\lf[\mu\lf(Q_\az^{k+1}\r)\r]^{-1/2}\lf|\lz_{Q_{\az}^{k+1}}\r|\mathbf 1_{Q_{\az_0}^{k_0+1}}\noz\\
&\quad\ls\sum_{k=k_0+1}^{\fz}\sum_{\az\in\CG_k}\dz^{(k_0-k)s}\dz^{(k-k_0)\ez/2}
\lf[\CM\lf(\sum_{\az\in\CG_k}\lf|\lz_{Q_{\az}^{k+1}}\wz{\mathbf 1}_{Q_\az^{k+1}}\r|^r\r)\r]^{1/r}.
\end{align}
Here, since
$r_{p,q,\ez}<\min\{p,q\},$ we may choose $r\in(0,\min\{p,q\})$ in the above inequality such that
$r>r_{p,q,\ez},$
and then use Lemma \ref{lem-btl}.
Thus, by \eqref{eq-add} and the H\"{o}lder inequality when $q\in(1,\fz]$, or Lemma \ref{lem-pin} when $q\in(0,1]$, we
conclude that
\begin{align*}
&\lf\{\sum_{k_0=-\fz}^\fz\dz^{-k_0sq}\lf[\sum_{\az_0\in\CG_{k_0}}\lf|(A_1\lz)_{Q_{\az_0}^{k_0+1}}\r|
\wz{\mathbf 1}_{Q_{\az_0}^{k_0+1}}\r]^q\r\}^{1/q}\\
&\quad\ls\lf\{\sum_{k=-\fz}^{\fz}\sum_{\az\in\CG_k}\dz^{-ksq}
\lf[\CM\lf(\sum_{\az\in\CG_k}\lf|\lz_{Q_{\az}^{k+1}}
\wz{\mathbf 1}_{Q_\az^{k+1}}\r|^r\r)\r]^{q/r}\r\}^{1/q}.
\end{align*}
Finally, we use Lemma \ref{lem-fs} to obtain
\begin{align*}
\|A_1\lz\|_{\dot f^s_{p,q}}
&\ls\lf\|\lf\{\sum_{k=-\fz}^{\fz}\dz^{-ksq}\lf[\CM\lf(\sum_{\az\in\CG_k}\lf|\lz_{Q_{\az}^{k+1}}
\wz{\mathbf 1}_{Q_\az^{k+1}}\r|^r\r)\r]^{q/r}\r\}^{r/q}\r\|_{L^{p/r}(X)}^{1/r}\\
&\ls\lf\|\lf[\sum_{k=-\fz}^{\fz}\dz^{-ksq}\lf(\sum_{\az\in\CG_k}\lf|\lz_{Q_{\az}^{k+1}}
\wz{\mathbf 1}_{Q_\az^{k+1}}\r|^q\r)\r]^{r/q}\r\|_{L^{p/r}(X)}^{1/r}\sim\|\lz\|_{\dot f^s_{p,q}}.
\end{align*}
which is the desired estimate in this case.

To summarize Cases 4.1) and 4.2), we obtain the boundedness of $A_1$ on $\dot f^s_{p,q}$. Combining all the conclusions
of Cases 1.1) through 4.2), we then complete the proof of Theorem \ref{thm-ado}.
\end{proof}

\section[Molecular characterization of Besov and Triebel--Lizorkin spaces via wavelets]
{Molecular characterization of Besov and Triebel--Lizorkin \\ spaces via wavelets}
\label{s-mol}

In this section, we consider the molecular characterization of Besov and Triebel--Lizorkin spaces.

First, we introduce the following notion of molecules. In what
follows, for any given $Q\in\wD$, as before, we use $\ell(Q)$ to denote its ``\emph{side-length}'' and
$x_Q$ its ``\emph{center}''.

\begin{definition}\label{def-mol}
Let $Q\in\wz{\mathcal D}$ with $\widetilde{D}$ as in \eqref{4.1x},
and $(\bz,\Gamma)\in(0,\fz)^2$. A function $b_Q$ on $X$
is called a \emph{molecule of type $(\bz,\Gamma)$
centered at $Q$} [for short, \emph{$(\bz,\Gamma)$-molecule}] if $b_Q$ satisfies the following conditions:
\begin{enumerate}
\item (the \emph{size condition}) for any $x\in X$, $|b_Q(x)|\le[\mu(Q)]^{1/2}P_\Gamma(x_Q,x;\ell(Q))$;
\item (the \emph{H\"older regularity condition}) for any $x,\ x'\in X$ with
$d(x,x')\le(2A_0)^{-1}[\ell(Q)+d(y_Q,x)]$,
$$
|b_Q(x)-b_Q(x')|\le[\mu(Q)]^{1/2}\lf[\frac{d(x,x')}{\ell(Q)+d(x_Q,x)}\r]^\bz P_\Gamma(x_Q,x;\ell(Q));
$$
\item (the \emph{cancellation  condition}) $\int_X b_Q(x)\,d\mu(x)=0$.
\end{enumerate}
\end{definition}

Observe that any molecule in Definition \ref{def-mol} centers at a subtly selected cube $Q\in\wD$ with $\wD$ as in
\eqref{4.1x}. It is obvious that $\wD$ may not contain all dyadic cubes of $X$ in Lemma \ref{lem-cube} and,
indeed, $\wD$ is the set of all ``supports'' of wavelet functions $\{\psi_\az^k\}_{k\in\zz,\ \az\in\CG_k}$
constructed in \cite{ah13} (see also Lemma \ref{lem-wave}).

We first prove the following proposition. In what follows, for any $Q:=Q_\az^{k+1}\in\wD$
with some $k\in\zz$ and $\az\in\CG_k$, let $\psi_Q:=\psi_\az^k$.

\begin{proposition}\label{prop-bB}
Let $s$, $p$, $q$, $\bz$, and $\gz$ be the same as in Definition \ref{def-btl}(i) and $\eta$ the same
as in Definition \ref{def-eti}.
Suppose that $\lz:=\{\lz_Q\}_{Q\in\wD}\in \dot b^s_{p,q}$ and that $\{b_Q\}_{Q\in\wz{\mathcal D}}$
are $(\bz,\gz)$-molecules centered, respectively, at $\{Q\}_{Q\in\wz\CD}$. Then there exists an $f\in (\GOO{\bz,\gz})'$
such that $f=\sum_{Q\in\wz\CD}\lz_Q b_Q$ in $(\GOO{\bz,\gz})'$, and $f\in\dot B^s_{p,q}(X)$.
Moreover, there exists a positive constant $C$,
independent of $\{\lz_Q\}_{Q\in\wD}$ and $\{b_Q\}_{Q\in\wD}$, such that
$\|f\|_{\dot B^s_{p,q}(X)}\le C\|\lz\|_{\dot b^s_{p,q}}$.
\end{proposition}

\begin{proof}
Let all the notation be the same as in this proposition. We now prove the first
conclusion of this proposition. To this end,
we claim that, for any $\vz\in\GOO{\bz,\gz}$
with $\|\vz\|_{\GOO{\bz,\gz}}\le 1$,
\begin{equation}\label{eq-dis}
\sum_{Q\in\wD}|\lz_Q||\langle b_Q,\vz\rangle|\ls \|\lz\|_{\dot b^s_{p,q}}.
\end{equation}
Indeed, if \eqref{eq-dis} holds true, then, by the completeness of $(\GOO{\bz,\gz})'$, we find that there
exists an $f\in (\GOO{\bz,\gz})'$ such that $f=\sum_{Q\in\wD}\lz_Q b_Q$ in $(\GOO{\bz,\gz})'$, which is the
desired conclusion.

Now, we show \eqref{eq-dis}. Without loss of generality, we may assume that $x_0\in Q_{\az_0}^0$ for some
$\az\in\CG_{-1}\subset\CA_{0}$. Observe that, if we define $|\lz|:=\{|\lz|_Q\}_{Q\in\wz\CD}$ by setting, for any $Q\in\wz\CD$,
$|\lz|_Q:=|\lz_Q|$, we find that $\|\,|\lz|\,\|_{\dot{b}^s_{p,q}}=\|\lz\|_{\dot{b}^s_{p,q}}$. For any
$Q,\ P\in\wz\CD$, we let
$$
A_{Q,P}^{(1)}:=\begin{cases}
|\langle b_P,\vz\rangle| & \textup{if }Q=Q_0, \\
0 & \textup{if } Q\neq Q_0,
\end{cases}
$$
and $A^{(1)}:=\{A^{(1)}_{Q,P}\}_{Q,\ P\in\wz\CD}$. Then
$$
\sum_{P\in\wD}|\lz_P||\langle b_P,\vz\rangle|=\lf(A^{(1)}|\lz|\r)_{Q_0}
\sim\lf\|A^{(1)}|\lz|\r\|_{\dot b^s_{p,q}}.
$$
If we can show that $A^{(1)}$ is an almost diagonal operator on $\dot b^s_{p,q}$, then,
from Theorem \ref{thm-ado}, it follows that
$$
\sum_{P\in\wD}|\lz_P|\langle b_P,\vz\rangle|\sim\lf\|A^{(1)}|\lz|\r\|_{\dot b^s_{p,q}}\ls\|\,|\lz|\,\|_{\dot b^s_{p,q}}
\sim\|\lz\|_{\dot b^s_{p,q}},
$$
which is just \eqref{eq-dis}.

Next, we show that $A^{(1)}$ is an almost diagonal operator on $\dot b^s_{p,q}$. To this end, we suppose
$P:=Q_\az^{k+1} $ with some $k\in\zz$ and $\az\in\CG_k\subset\CA_{k+1}$. When $Q\neq Q_{\az_0}^0$, we
have $A^{(1)}_{Q,P}=0$. Otherwise, similarly to the proof of \cite[Lemma 3.9]{whhy20}, we have
\begin{equation*}
\lf|A^{(1)}_{Q,P}\r|=|\langle b_Q,\vz\rangle|\ls\dz^{|k|\bz'}
\lf[\mu\lf(Q_\az^{k+1}\r)\r]^{1/2} P_\gz\lf(x_0,y_\az^k,\dz^{k\wedge 0}\r),
\end{equation*}
where $\bz'\in(0,\bz\wedge\gz)$.
Thus, we find that
$A^{(1)}$ is an almost diagonal operator on $\dot b^s_{p,q}$, which implies  that
$f:=\sum_{P\in\wD}\lz_P b_P$ converges in $(\GOO{\bz,\gz})'$, and hence completes the proof
of the first conclusion of this proposition.

Now, we prove that $\|f\|_{\dot B^s_{p,q}(X)}\ls\|\lz\|_{\dot b^s_{p,q}}$. Indeed,
since $f=\sum_{P\in\wD}\lz_Pb_P$ in $(\GOO{\bz,\gz})'$, it then follows that, for any $Q\in\wD$,
$$
\lf<f,\psi_Q\r>=\sum_{P\in\wD}\lz_P\lf\langle b_P,\psi_Q\r\rangle.
$$
If we let $A_{Q,P}:=\langle b_P,\psi_Q\rangle$, we then
find that, for any $Q,\ P\in\wD$ and $A:=\{A_{Q,P}\}_{Q,\ P\in\wD}$,
$$
\lf\langle f,\psi_Q\r\rangle=\sum_{P\in\wD}A_{Q,P}\lz_P=(A\lz)_Q,
$$
which, together with Theorem \ref{thm-wavebtl}(i), further implies that
\begin{equation}\label{eq-5}
\|f\|_{\dot B^s_{p,q}(X)}\sim\lf\|\lf\{\lf< f,\psi_Q\r>\r\}_{Q\in\wD}\r\|_{\dot b^s_{p,q}}\sim
\|A\lz\|_{\dot b^s_{p,q}}.
\end{equation}
We next claim that $A$ is an almost diagonal operator on $\dot b^s_{p,q}$.
Assuming this for the moment, by \eqref{eq-5} and Theorem \ref{thm-ado}, we find that
\begin{equation*}
\|f\|_{\dot B^s_{p,q}(X)}\sim\|A\lz\|_{\dot b^s_{p,q}}\ls\|\lz\|_{\dot b^s_{p,q}},
\end{equation*}
which is the desired conclusion.

It remains to show the above claim. To this end, letting $Q,\ P\in\wD$,
we then consider the following two cases on
$\ell(Q)$ and $\ell(P)$.

{\it Case 1) $\ell(Q)\ge\ell(P)$.} In this case, by the cancellation of $b_P$, we have
\begin{align}\label{eq-bp1}
\lf|\lf<b_P,\psi_Q\r>\r|&\le\int_{W_1}\lf|b_P(y)\r|\lf|\psi_Q(y)-\psi_Q(x_P)\r|\,d\mu(y)+\int_{W_2}
\lf|b_P(y)\psi_Q(y)\r|\,d\mu(y)\noz\\
&\qquad+|\psi_Q(x_P)|\int_{W_2}|b_P(y)|\,d\mu(y),
\end{align}
where $W_1:=\{y\in X:\ d(y,x_P)\le(2A_0)^{-1}[\ell(Q)+d(x_Q,x_P)]\}$ and $W_2:=W_1^\complement$.

We first deal with the  integral on $W_1$ of \eqref{eq-bp1}. By the size condition of $b_P$,
the regularity of $\psi_Q$, and $\gz<\eta$, we find that
\begin{align*}
&\int_{W_1}\lf|b_P(y)\r|\lf|\psi_Q(y)-\psi_Q(x_P)\r|\,d\mu(y)\\
&\quad\ls[\mu(P)]^{1/2}\int_{W_1}\lf[\frac{d(y,x_P)}{\ell(Q)}\r]^\eta\frac 1{\sqrt{V_{\ell(Q)}(x_P)}}
\exp\lf\{-\nu\lf[\frac{d(x_Q,x_P)}{\ell(Q)}\r]^a\r\}P_\gz(x_P,y;\ell(P))\,d\mu(y)\\
&\quad\ls\lf[\frac{\ell(P)}{\ell(Q)}\r]^\gz[\mu(Q)\mu(P)]^{1/2}\frac{1}{V_{\ell(Q)}(x_Q)+V(x_Q,x_P)}
\exp\lf\{-\nu'\lf[\frac{d(x_Q,x_P)}{\ell(Q)}\r]^a\r\}.
\end{align*}
Now, we consider the second integral of  \eqref{eq-bp1}. By the size conditions of $\psi_Q$ and $b_P$, we have
\begin{equation*}
\int_{W_2}\lf|b_P(y)\psi_Q(y)\r|\,d\mu(y)\ls\lf[\frac{\ell(P)}{\ell(Q)}\r]^\gz[\mu(Q)\mu(P)]^{-1/2}P_\gz(x_P,y;\ell(Q)).
\end{equation*}
Finally, for the last integral of \eqref{eq-bp1}, from the size conditions of $\psi_Q$ and $b_P$, we deduce that
\begin{align*}
&|\psi_Q(x_P)|\int_{W_2}|b_P(y)|\,d\mu(y)\\
&\quad\ls\lf[\frac{\ell(P)}{\ell(Q)}\r]^\gz[\mu(Q)\mu(P)]^{-1/2}\frac 1{V_{\ell(Q)}(x_Q)+V(x_Q,x_P)}
\exp\lf\{-\nu'\lf[\frac{d(y,x_Q)}{\ell(Q)}\r]^a\r\}.
\end{align*}

Combining the above three estimates, we find that
\begin{align}\label{eq-lQ>lP}
\lf|\lf<b_P,\psi_Q\r>\r|
&\ls\lf[\frac{\ell(P)}{\ell(Q)}\r]^\gz[\mu(Q)\mu(P)]^{-1/2}P_\gz(x_Q,x_P;\ell(Q))\noz\\
&\sim\lf[\frac{\ell(Q)}{\ell(P)}\r]^s[\mu(Q)\mu(P)]^{-1/2}P_\gz(x_Q,x_P;\ell(Q))\lf[\frac{\ell(P)}{\ell(Q)}\r]^{s+\gz}.
\end{align}
This is the desired estimate in this case.

{\it Case 2) $\ell(Q)<\ell(P)$.} In this case, we let
$W_3:=\{y\in X:\ d(y,x_Q)\le(2A_0)^{-1}[\ell(P)+d(x_Q,x_P)]\}$ and $W_4:=W_3^\complement$. Then, by the
cancellation of $\psi_Q$, we have
\begin{align*}
\lf|\lf<\psi_Q,b_P\r>\r|&\le\int_{W_3}\lf|\psi_Q(y)\r|\lf|b_P(y)-b_P(x_Q)\r|\,d\mu(y)
+\int_{W_4}\lf|\psi_Q(y)b_P(y)\r|\,d\mu(y)\\
&\qquad+\lf|b_P(x_Q)\r|\int_{W_4}\lf|\psi_Q(y)\r|\,d\mu(y)\\
&=:\RJ_1+\RJ_2+\RJ_3.
\end{align*}

For $\RJ_1$, by the regularity of $b_P$ and the size condition of $\psi_Q$, we obtain
\begin{align*}
\RJ_1&\ls[\mu(Q)\mu(P)]^{1/2}P_\gz(x_Q,x_P;\ell(Q))\\
&\qquad\times\int_{W_3}\lf[\frac{d(y,x_Q)}{\ell(P)+d(x_Q,x_P)}\r]^\bz\frac 1{V_{\ell(Q)}(x_Q)}
\exp\lf\{-\nu\lf[\frac{d(y,x_Q)}{\ell(Q)}\r]^a\r\}\,d\mu(y).
\end{align*}
To estimate the above integral on $W_3$, we let $W_{3,1}:=\{y\in W_3:\ d(y,x_Q)\le\ell(Q)\}$ and
$W_{3,2}:=W_3\setminus W_{3,1}$, and conclude that
\begin{align*}
&\int_{W_{3,1}}\lf[\frac{d(y,x_Q)}{\ell(P)+d(x_Q,x_P)}\r]^\bz\frac 1{V_{\ell(Q)}(x_Q)}
\exp\lf\{-\nu\lf[\frac{d(y,x_Q)}{\ell(Q)}\r]^a\r\}\,d\mu(y)
\ls\lf[\frac{\ell(Q)}{\ell(P)}\r]^\bz.
\end{align*}
Moreover, for the integral on $W_{3,2}$, we have
\begin{align*}
&\int_{W_{3,2}}\lf[\frac{d(y,x_Q)}{\ell(P)+d(x_Q,x_P)}\r]^\bz\frac 1{V_{\ell(Q)}(x_Q)}
\exp\lf\{-\nu\lf[\frac{d(y,x_Q)}{\ell(Q)}\r]^a\r\}\,d\mu(y)\\
&\quad\ls\lf[\frac{\ell(Q)}{\ell(P)}\r]^\bz\int_{d(y,x_Q)>\ell(Q)}\frac{1}{V(x_Q,y)}
\lf[\frac{\ell(Q)}{d(y,x_Q)}\r]^{L-\bz}\,d\mu(y)\ls\lf[\frac{\ell(Q)}{\ell(P)}\r]^\bz,
\end{align*}
where we chose $L\in(\bz,\fz)$.
By the above inequalities, we find that
\begin{equation*}
\RJ_1\ls\lf[\frac{\ell(Q)}{\ell(P)}\r]^\bz [\mu(Q)\mu(P)]^{1/2}P_{\gz}(x_Q,x_P;\ell(P)),
\end{equation*}
which is the desired estimate on $\RJ_1$.

For $\RJ_2$, we use the size conditions of $\psi_Q$ and $b_P$ to conclude that
\begin{align*}
\RJ_2&\ls\lf[\frac{\ell(Q)}{\ell(P)}\r]^\Gamma[\mu(Q)\mu(P)]^{1/2}P_{\Gamma}(x_Q,x_P;\ell(P)).
\end{align*}
Here and thereafter, $\Gamma\in(\max\{\bz,\gz\},\fz)$ is a fixed large positive number. This is the desired
estimate on $\RJ_2$.

For $\RJ_3$, again, by the size conditions of $\psi_Q$ and $b_P$, we obtain
\begin{equation*}
\RJ_3\ls\lf[\frac{\ell(Q)}{\ell(P)}\r]^\Gamma[\mu(Q)\mu(P)]^{1/2}P_{\gz}(x_Q,x_P;\ell(P)).
\end{equation*}

Combining the above three estimates, we find that
\begin{equation}\label{eq-lQ<lP}
\lf|\langle\psi_Q, b_P\rangle\r|\ls\lf[\frac{\ell(Q)}{\ell(P)}\r]^s[\mu(Q)\mu(P)]^{1/2}P_{\gz}(x_Q,x_P;\ell(P))
\lf[\frac{\ell(Q)}{\ell(P)}\r]^{\bz-s}.
\end{equation}
This is also the desired estimate in this case.

Finally, by \eqref{eq-lQ>lP}, \eqref{eq-lQ<lP}, $\bz\wedge\gz>|s|$, and $p>p(s,\bz\wedge\gz)$, we find that $A$ is an
almost diagonal operator on $\dot b^s_{p,q}$ with any given
$$
\ez\in\lf(0,\min\lf\{\gz-\omega\lf(\frac 1p-1\r)_+,2\lf[s+\gz-\omega\lf(\frac 1p-1\r)_+\r],2(\bz-s)\r\}\r),
$$
where $\omega\in[\omega_0,\fz)$ satisfies \eqref{eq-doub} and all the assumption of this proposition with $\omega_0$
replaced by $\omega$. This finishes the proof of Proposition \ref{prop-bB}.
\end{proof}

Similarly to the proof of Proposition \ref{prop-bB}, we can obtain the
following conclusion on Triebel--Lizorkin spaces, and we omit the details here.

\begin{proposition}\label{prop-fF}
Let $s$, $p$, $q$, $\bz$, and $\gz$ be the same as in Definition \ref{def-btl}(ii), and $\eta$ the same
as in Definition \ref{def-eti}.
Suppose that $\lz:=\{\lz_Q\}_{Q\in\wD}\in \dot f^s_{p,q}$ and that $\{b_Q\}_{Q\in\wz{\mathcal D}}$ are
$(\bz,\gz)$-molecules centered, respectively, at $\{Q\}_{Q\in\wz\CD}$. Then there exists an $f\in (\GOO{\bz,\gz})'$ such
that $f=\sum_{Q\in\wz\CD}\lz_Q b_Q$ in $(\GOO{\bz,\gz})'$, and $f\in\dot F^s_{p,q}(X)$.
Moreover, there exists a positive constant $C$, independent of
$\{\lz_Q\}_{Q\in\wD}$ and $\{b_Q\}_{Q\in\wD}$, such that
$\|f\|_{\dot F^s_{p,q}(X)}\le C\|\lz\|_{\dot f^s_{p,q}}$.
\end{proposition}

These propositions, together with the wavelet characterization of Besov and Triebel--Lizorkin spaces
in Theorem \ref{thm-wavebtl}, further imply the following molecular characterization of Besov
and Triebel--Lizorkin spaces; we omit the details.
\begin{theorem}\label{thm-mol}
Let $s$, $p$, $q$, $\bz$, and $\gz$ be as in Definition \ref{def-btl}(i) [resp., Definition \ref{def-btl}(ii)],
and $\eta$ the same as in Definition \ref{def-eti}.
Then $f\in\dot B^s_{p,q}(X)$ [resp., $f\in\dot F^s_{p,q}(X)$] if and only if
there exist a $\lz:=\{\lz_Q\}_{Q\in\wD}\in\dot b^s_{p,q}$ [resp., $\lz:=\{\lz_Q\}_{Q\in\wD}\in\dot f^s_{p,q}$]
and $(\bz,\gz)$-molecules $\{b_Q\}_{Q\in\wD}$ centered, respectively, at $\{Q\}_{Q\in\wD}$
such that $f=\sum_{Q\in\wz\CD}\lz_Q b_Q$ in $(\GOO{\bz,\gz})'$.
Moreover, there exists a constant $C\in[1,\fz)$, independent of $f$, $\lz$, and $\{b_Q\}_{Q\in\wD}$, such that
$C^{-1}\|\lz\|_{\dot b^s_{p,q}}\le \|f\|_{\dot B^s_{p,q}(X)}\le C\|\lz\|_{\dot b^s_{p,q}}$
[resp., $C^{-1}\|\lz\|_{\dot f^s_{p,q}}\le \|f\|_{\dot F^s_{p,q}(X)}\le C\|\lz\|_{\dot f^s_{p,q}}$].
\end{theorem}

\begin{remark}\label{rem-mol}
\begin{enumerate}
\item In \cite{gkkp19}, Georgiadis et al.  introduced two kinds
of molecules,  respectively, called synthesis and analysis
molecules, which have different properties
and were motivated by Frazier and Jawerth
\cite{fj90}. Using these molecules,
Georgiadis et al. \cite{gkkp19} characterized Triebel--Lizorkin spaces,
associated with operators, on spaces of homogeneous type
(see \cite[Theorems 7.4 and 7.5]{gkkp19}). Differently from those molecules in
\cite{gkkp19} (see also \cite{fj90}), we introduce a unified
kind of molecules (see Definition \ref{def-mol})
to characterize Besov and Triebel--Lizorkin spaces here.

\item Notice that, by Definition \ref{def-mol}, it is easy to see that
a molecule has the form of test functions
on $X$, which only has the polynomial decay, and hence has no
exponential decay. To get rid of dependence of the
reverse doubling property of $\mu$, differently from those
molecules in \cite[Definition (6.21)]{hs94} which
center at all cubes, the molecules in Definition \ref{def-mol} are centered at
some subtly chosen cubes, more precisely, at the ``supports''
of wavelet functions introduced in \cite{ah13}.
As a result, we can use the wavelet characterization of
Besov and Triebel--Lizorkin spaces (see Theorem \ref{thm-wavebtl}),
and therefore the geometrical properties of $X$ to get rid of the dependence of the
reverse doubling property of $\mu$.
\end{enumerate}
\end{remark}

\section{Littlewood--Paley characterizations of Triebel--Lizorkin spaces}\label{s-lp}

In this section, we establish the Littlewood--Paley characterizations of homogeneous Triebel--Lizorkin spaces by first
recalling the notions of Littlewood--Paley functions.
Let $\eta$ be the same as in Definition \ref{def-eti}, $s\in\rr$, and $q\in(0,\fz]$. Recall that, for any $f\in(\GOO{\bz,\gz})'$
with $\bz,\ \gz\in(0,\eta)$, the \emph{Littlewood--Paley $g$-function $\dot g^s_q(f)$ of $f$} is defined by setting,
for any $x\in X$,
$$
\dot g^s_q(f)(x):=\lf[\sum_{k\in\zz}\dz^{-ksq}|Q_kf(x)|^q\r]^{1/q},
$$
the \emph{Lusin area function $\dot\CS^s_q(f)$ of $f$} by setting, for any $x\in X$,
$$
\dot\CS^s_q(f)(x):=\lf[\sum_{k\in\zz}\dz^{-ksq}\int_{B(x,\dz^k)}|Q_kf(y)|^q
\,\frac{d\mu(y)}{V_{\dz^k}(x)}\r]^{1/q},
$$
and, for any given $\lz\in(0,\fz)$, the \emph{Littlewood--Paley $g_\lz^*$-function $(\dot g_\lz^*)^s_q(f)$ of
$f$} by setting, for any $x\in X$,
\begin{equation}\label{eq-glstar}
\lf(\dot g_\lz^*\r)^s_q(f)(x):=\lf\{\sum_{k\in\zz}\dz^{-ksq}\int_X |Q_kf(y)|^q
\lf[\frac{\dz^k}{\dz^k+d(x,y)}\r]^\lz\,\frac{d\mu(y)}{V_{\dz^k}(x)+V_{\dz^k}(y)}\r\}^{1/q}.
\end{equation}
Obviously, by the definition of Triebel--Lizorkin spaces, we find that, for any
$f\in\dot F^s_{p,q}(X)$,
\begin{equation}\label{eq-gfun}
\lf\|\dot g^s_q(f)\r\|_{L^p(X)}=\|f\|_{\dot F^s_{p,q}(X)}.
\end{equation}

We have the following Lusin area function characterization of Triebel--Lizorkin spaces.

\begin{theorem}\label{thm-g=s}
Let $s$, $p$, $q$, $\bz$, and $\gz$ be the same as in Definition \ref{def-btl}(ii). Then $f\in \dot F^s_{p,q}(X)$ if
and only if $f\in(\GOO{\bz,\gz})'$ and $\dot\CS^s_q(f)\in L^p(X)$. Moreover, there exists a constant
$C\in[1,\fz)$, independent of $f$, such that
$C^{-1}\|\dot\CS^s_q(f)\|_{L^p(X)}\le\|f\|_{\dot F^s_{p,q}(X)}\le C\|\dot\CS^s_q(f)\|_{L^p(X)}$.
\end{theorem}

To prove this theorem, we need the following Plancherel--P\^olya inequality from \cite{whhy20}.

\begin{lemma}[{\cite[(3.8)]{whhy20}}]\label{lem-ppi}
Let $\{Q_k\}_{k=-\infty}^\infty$ and $\{P_k\}_{k=-\infty}^\infty$ be two {\rm $\exp$-ATIs}, and
$\beta,\ \gamma \in (0, \eta)$ with $\eta$ as in Definition \ref{def-eti}. Then, when
$s\in(-(\beta\wedge\gamma),\beta\wedge\gamma)$, $p\in (p(s,\beta\wedge\gamma),\infty)$, and
$q \in (p(s,\beta\wedge\gamma),\infty]$, there exists a positive constant $C$
such that, for any $f\in(\GOO{\bz,\gz})'$,
\begin{align*}
&\left\|\left[\sum_{k=-\infty}^\infty\sum_{\alpha \in \CA_k}\sum_{m=1}^{N(k,\alpha)}\delta^{-ksq}
\sup_{z\in Q_\az^{k,m}}|P_kf(z)|^q\mathbf 1_{Q_\az^{k,m}}\right]^{1/q}\right\|_{L^p(X)}\\
&\quad\le C\left\|\left[\sum_{k=-\infty}^\infty\sum_{\alpha \in \CA_k}\sum_{m=1}^{N(k,\alpha)}\delta^{-ksq}
\inf_{z\in Q_\az^{k,m}}|Q_kf(z)|^q\mathbf 1_{Q_\az^{k,m}}\right]^{1/q}\right\|_{L^p(X)}
\end{align*}
with the usual modification made when $q=\infty$.
\end{lemma}

\begin{proof}[Proof of Theorem \ref{thm-g=s}]
We first show the sufficiency of this theorem. To this end, suppose that $f\in(\GOO{\bz,\gz})'$ and
$\dot\CS^s_q(f)\in L^p(X)$ with $\bz$, $\gz$, $s$, $p$, and $q$ as in this theorem. On one hand, we have
\begin{align*}
\lf[\dot g^s_q(f)\r]^q&=\sum_{k=-\fz}^\fz\dz^{-ksq}|Q_kf|^q
=\sum_{k=-\fz}^\fz\dz^{-ksq}\sum_{\az\in\CA_k}\sum_{m=1}^{N(k,\az)}|Q_kf|^q\mathbf 1_{Q_\az^{k,m}}\\
&\le\sum_{k=-\fz}^\fz\dz^{-ksq}\sum_{\az\in\CA_k}\sum_{m=1}^{N(k,\az)}\sup_{z\in Q_\az^{k,m}}
|Q_kf(z)|^q\mathbf 1_{Q_\az^{k,m}}.
\end{align*}
On the other hand, for any $x\in Q_\az^{k,m}$ for some $k\in\zz$, $\az\in\CA_k$, and
$m\in\{1,\ldots,N(k,\az)\}$, we conclude that $\mu(Q_\az^{k,m})\sim \mu(Q_\az^k)\sim V_{\dz^k}(x)$. Moreover, since
$\dz$ is small, it then follows that $Q_\az^{k,m}\subset B(x,\dz^k)$.
Thus, we obtain, for any $x\in X$,
\begin{align*}
&\sum_{k=-\fz}^\fz\dz^{-ksq}\sum_{\az\in\CA_k}\sum_{m=1}^{N(k,\az)}\inf_{z\in Q_\az^{k,m}}
|Q_kf(z)|^q\mathbf 1_{Q_\az^{k,m}}(x)\\
&\quad\le\sum_{k=-\fz}^\fz\dz^{-ksq}\sum_{\az\in\CA_k}\sum_{m=1}^{N(k,\az)}\frac{1}{\mu(Q_\az^{k,m})}
\int_{Q_\az^{k,m}}|Q_kf(z)|^q\,d\mu(z)\mathbf 1_{Q_\az^{k,m}}(x)\\
&\quad\ls\sum_{k=-\fz}^\fz\dz^{-ksq}\frac 1{V_{\dz^k}(x)}\int_{B(x,\dz^k)} |Q_kf(z)|^q\,d\mu(z)
\sim \lf[\dot\CS^s_q(f)(x)\r]^q.
\end{align*}
Therefore, by the above two inequalities and Lemma \ref{lem-ppi}, we further conclude that
\begin{align*}
\|f\|_{\dot F^s_{p,q}(X)}&=\lf\|\dot g^s_q(f)\r\|_{L^p(X)}
\ls\lf\|\lf[\sum_{k=-\fz}^\fz\dz^{-ksq}\sum_{\az\in\CA_k}\sum_{m=1}^{N(k,\az)}\sup_{z\in Q_\az^{k,m}}
|Q_kf(z)|^q\mathbf 1_{Q_\az^{k,m}}\r]^{1/q}\r\|_{L^p(X)}\\
&\ls\lf\|\lf[\sum_{k=-\fz}^\fz\dz^{-ksq}\sum_{\az\in\CA_k}\sum_{m=1}^{N(k,\az)}\inf_{z\in Q_\az^{k,m}}
|Q_kf(z)|^q\mathbf 1_{Q_\az^{k,m}}\r]^{1/q}\r\|_{L^p(X)}\sim\lf\|\dot\CS^s_q(f)\r\|_{L^p(X)}.
\end{align*}
This finishes the proof of the sufficiency of this theorem.

Now, we consider the necessity of this theorem. To show this, suppose $f\in\dot F^s_{p,q}(X)$
with $s$, $p$, and $q$ as in this theorem. By Theorem \ref{thm-mol}, we find that there exist
$(\bz,\gz)$-molecules $\{b_Q\}_{Q\in\wD}$ centered, respectively, at $\{Q\}_{Q\in\wD}$ with $\bz$ and $\gz$ as in
this theorem, and $\lz:=\{\lz_Q\}_{Q\in\wD}\in\dot f^s_{p,q}$ such that
$$
f=\sum_{j\in\zz}\sum_{\az\in\CG_j}\lz_{Q_\az^{j+1}}b_{Q_\az^{j+1}} \quad
\textup{in }\lf(\GOO{\bz,\gz}\r)',
$$
and
\begin{equation}\label{eq-6.xx}
\|\lz\|_{\dot f^s_{p,q}}\ls\|f\|_{\dot F^s_{p,q}(X)}.
\end{equation}
Therefore, we have, for any $k\in\zz$, $x\in X$, and $y\in B(x,\dz^k)$,
$$
Q_kf(y)=\sum_{j\in\zz}\sum_{\az\in\CG_j}\lz_{Q_{\az}^{j+1}}Q_k\lf(b_{Q_\az^{j+1}}\r)(y).
$$
Similarly to the proof of \cite[Lemma 3.9]{whhy20}, we obtain, for any fixed $\bz'\in(0,\bz\wedge\gz)$, and any $j,\ k\in\zz$,
$\az\in\CG_j$, $x\in X$, and $y\in B(x,\dz^k)$,
\begin{equation*}
\lf|Q_k\lf(b_{Q_\az^{j+1}}\r)(y)\r|\ls\dz^{|k-j|\bz'}\lf[\mu(Q_\az^{j+1})\r]^{-1/2}P_\gz\lf(x,y_\az^j;\dz^{j\wedge k}\r),
\end{equation*}
which further implies that, for any $k\in\zz$ and $x\in X$,
\begin{equation}\label{eq-add2}
\frac 1{V_{\dz^k}(x)}\int_{B(x,\dz^k)}|Q_kf(y)|^q\,d\mu(y)
\ls\lf\{\sum_{j\in\zz}\dz^{|k-j|\bz'}\sum_{\az\in\CG_j}
\lf[\mu(Q_\az^{j+1})\r]^{-1/2}\lf|\lz_{Q_{\az}^{j+1}}\r|P_\gz\lf(x,y_\az^j;\dz^{j\wedge k}\r)\r\}^q.
\end{equation}

To proceed, we consider the following two cases on $\min\{p,q\}$.
When $\min\{p,q\}>1$, then, by \eqref{eq-add2}, the H\"older inequality, and \eqref{eq-btl}, we find that,
for any $k\in\zz$ and $x\in X$,
\begin{align}\label{eq-molq>1}
\frac 1{V_{\dz^k}(x)}\int_{B(x,\dz^k)}|Q_kf(y)|^q\,d\mu(y)&\ls\lf\{\sum_{j\in\zz}\dz^{|k-j|\bz'}\sum_{\az\in\CG_j}
\lf[\mu(Q_\az^{j+1})\r]^{-1/2}\lf|\lz_{Q_{\az}^{j+1}}\r|P_\gz\lf(x,y_\az^j;\dz^{j\wedge k}\r)\r\}^q\noz\\
&\ls\sum_{j\in\zz}\dz^{|k-j|(\bz'-\ez)q}\CM\lf(\sum_{\az\in\CG_j}\lf|\lz_{Q_{\az}^{j+1}}
\wz{\mathbf 1}_{Q_\az^{j+1}}\r|^q\r)(x),
\end{align}
where $\ez\in(0,\fz)$ is a fixed positive number. Since $|s|<\min\{\bz,\gz\}$, we can choose $\bz'\in(0,\bz\wedge\gz)$
and $\ez\in(0,\bz')$ such that $\bz'-\ez>|s|$. Using this and \eqref{eq-molq>1}, we further conclude that
\begin{equation*}
\lf[\dot\CS^s_q(f)\r]^q\ls\sum_{j\in\zz}\dz^{-jsq}\CM\lf(\sum_{\az\in\CG_j}\lf|\lz_{Q_{\az}^{j+1}}
\wz{\mathbf 1}_{Q_\az^{j+1}}\r|^q\r).
\end{equation*}
Thus, from this, Lemma \ref{lem-fs}, and \eqref{eq-6.xx}, we deduce that
\begin{align*}
\lf\|\dot\CS^s_q(f)\r\|_{L^p(X)}&\ls\lf\|\lf[\sum_{j\in\zz}\dz^{-jsq}
\CM\lf(\sum_{\az\in\CG_j}\lf|\lz_{Q_{\az}^{j+1}}\wz{\mathbf 1}_{Q_\az^{j+1}}\r|^q\r)\r]^{1/q}\r\|_{L^p(X)}\\
&\ls\lf\|\lf(\sum_{j\in\zz}\dz^{-jsq}\sum_{\az\in\CG_j}\lf|\lz_{Q_{\az}^{j+1}}
\wz{\mathbf 1}_{Q_\az^{j+1}}\r|^q\r)^{1/q}\r\|_{L^p(X)}\sim\|\lz\|_{\dot f^s_{p,q}}
\ls\|f\|_{\dot F^s_{p,q}(X)}.
\end{align*}
This finishes the proof of the necessity of this theorem when $\min\{p,q\}>1$.

When $\min\{p,q\}\le 1$, let $\omega\in[\omega_0,\fz)$ satisfy \eqref{eq-doub} and all the assumptions of this theorem
with $\omega_0$ replaced by $\omega$. Using this, \eqref{eq-add2}, and \eqref{eq-btl}, and choosing
$r\in(\omega/[\omega+\gz],\min\{p,q\})$, we obtain, for any $k\in\zz$ and $x\in X$,
\begin{align}\label{eq-lusin3}
&\frac 1{V_{\dz^k}(x)}\int_{B(x,\dz^k)}|Q_kf(y)|^q\,d\mu(y)\noz\\
&\quad\ls\lf\{\sum_{j\in\zz}\dz^{|k-j|\bz'}\dz^{[j-(j\wedge k)]\omega(1-1/r)}
\lf[\CM\lf(\sum_{\az\in\CG_j}\lf|\lz_{Q_{\az}^{j+1}}\wz{\mathbf 1}_{Q_\az^{j+1}}\r|^r\r)(x)\r]^{1/r}\r\}^q.
\end{align}
Since $p(s,\bz\wedge\gz)<\min\{p,q\}$, we may choose $\bz'\in(0,\bz\wedge\gz)$ and $r\in(\omega/[\omega+\eta],\min\{p,q\})$
such that $\bz'>s$ and $\bz'+s>\omega(1/r-1)$. Thus, when $k\ge j$,
$$
(j-k)s+|k-j|\bz'+[j-(j\wedge k)]\omega(1-1/r)=(k-j)(\bz'-s)>0
$$
and, when $k<j$,
$$
(j-k)s+|k-j|\bz'+[j-(j\wedge k)]\omega(1-1/r)=(j-k)[s+\bz'-\omega(1/r-1)]>0.
$$
On one hand, for any $k\in\zz$,
$$
\sum_{j\in\zz}\dz^{(j-k)s}\dz^{|k-j|\bz'}\dz^{[j-(j\wedge k)]\omega(1-1/r)}\ls 1
$$
with the implicit positive constant independent of $k$; on the other hand, for any $j\in\zz$,
$$
\sum_{k\in\zz}\lf\{\dz^{(j-k)s}\dz^{|k-j|\bz'}\dz^{[j-(j\wedge k)]\omega(1-1/r)}\r\}^{q\wedge 1}\ls 1
$$
with the implicit positive constant independent of $j$. By the above two
inequalities, \eqref{eq-lusin3}, and the H\"older
inequality when $q\in(1,\fz]$, or Lemma \ref{lem-pin} when $q\in(p(s,\bz\wedge\gz),1]$,
we conclude that
\begin{align}\label{eq-sum}
\lf[\dot\CS^s_q(f)\r]^q
&\ls\sum_{k\in\zz}\lf\{\sum_{j\in\zz}\dz^{-js}\dz^{(j-k)s}\dz^{|k-j|\bz'}\dz^{[j-(j\wedge k)]\omega(1-1/r)}
\lf[\CM\lf(\sum_{\az\in\CG_j}\lf|\lz_{Q_{\az}^{j+1}}\wz{\mathbf 1}_{Q_\az^{j+1}}\r|^r\r)\r]^{1/r}\r\}^q\noz\\
&\ls\sum_{j\in\zz}\dz^{-jsq}
\lf[\CM\lf(\sum_{\az\in\CG_j}\lf|\lz_{Q_{\az}^{j+1}}\wz{\mathbf 1}_{Q_\az^{j+1}}\r|^r\r)\r]^{q/r}.
\end{align}
Therefore, from this and Lemma \ref{lem-fs}, we further deduce that
\begin{align*}
\lf\|\dot\CS^s_q(f)\r\|_{L^p(X)}&\ls\lf\|\lf\{\sum_{j\in\zz}\dz^{-jsq}
\lf[\CM\lf(\sum_{\az\in\CG_j}\lf|\lz_{Q_{\az}^{j+1}}\wz{\mathbf 1}_{Q_\az^{j+1}}\r|^r\r)\r]^{q/r}\r\}^{1/q}
\r\|_{L^p(X)}\\
&\ls\lf\|\lf[\sum_{j\in\zz}\dz^{-jsq}
\lf(\sum_{\az\in\CG_j}\lf|\lz_{Q_{\az}^{j+1}}\wz{\mathbf 1}_{Q_\az^{j+1}}\r|^r\r)^{q/r}\r]^{r/q}
\r\|_{L^{p/r}(X)}^{1/r}\sim\|\lz\|_{\dot f^s_{p,q}}\ls\|f\|_{\dot F^s_{p,q}(X)}.
\end{align*}
This finishes the proof of the necessity of this theorem when $\min\{p,q\}\le 1$,
and hence of Theorem \ref{thm-g=s}.
\end{proof}

Finally, we establish the Littlewood--Paley $g_\lz^*$-function characterization of Triebel--Lizorkin spaces.
We first consider the case $q\in(0,p]$ and, in this case, we have the following conclusion.

\begin{proposition}\label{prop-glp>q}
Let $\eta$ be as in Definition \ref{def-eti},
$s$, $p$, $q$, $\bz$, and $\gz$ as in Definition \ref{def-btl} and $q\in(0,p]$,
and $\lz\in(\omega_0,\fz)$ with $\omega_0$ as in \eqref{eq-updim}. Then $f\in\dot F^s_{p,q}(X)$
if and only if $f\in(\GOO{\bz,\gz})'$ with $\bz$ and $\gz$ satisfying \eqref{eq-bzgz}, and
$(\dot g_\lz^*)^s_q\in L^p(X)$. Moreover, there exists a constant $C\in[1,\fz)$, independent of $f$, such that
$C^{-1}\|f\|_{\dot F^s_{p,q}(X)}\le\|(\dot g_\lz^*)^s_q(f)\|_{L^p(X)}\le C\|f\|_{\dot F^s_{p,q}(X)}$.
\end{proposition}

\begin{proof}
Let all the notation be the same as in this proposition. The sufficiency
of this proposition holds true by using Theorem \ref{thm-g=s} and the fact
$\dot\CS^\az_q(f)\ls(\dot g_\lz^*)^\az_q(f)$ for any $f\in(\GOO{\bz,\gz})'$;
we omit the details.

Next, we prove the necessity of this proposition.
Let $\omega\in[\omega_0,\fz)$
satisfy \eqref{eq-doub} and all the assumptions of this proposition with $\omega_0$
replaced by $\omega$.
We first consider the case $p=q$. Indeed, let $f\in\dot F^s_{p,p}(X)$. By the Tonelli theorem, $\lz>\omega$, and
\eqref{eq-gfun}, we find that
\begin{align*}
&\lf\|\lf(\dot g_\lz^*\r)^s_p(f)\r\|_{L^p(X)}^p\\
&\quad\ls\int_X\sum_{k=-\fz}^\fz\dz^{-ksp}\int_X|Q_kf(y)|^q
\lf[\frac{\dz^k}{\dz^k+d(x,y)}\r]^{\lz-\omega}\frac{1}{V_{\dz^k}(x)+V(x,y)}\,d\mu(y)\,d\mu(x)\\
&\quad\sim\int_X\sum_{k=-\fz}^\fz\dz^{-ksp}|Q_kf(y)|^q\,d\mu(y)\sim\lf\|\dot g^s_q(f)\r\|_{L^p(X)}^p
\sim\|f\|_{\dot F^s_{p,p}(X)}.
\end{align*}
This finishes the proof of the necessity of this proposition in the case $p=q$.

Now, we consider the case $p>q$. Since $p<\fz$, it then follows that $(p/q)'\in(1,\fz)$ and
$(L^{p/q}(X))'=L^{(p/q)'}(X)$. Moreover, for any $\lz\in(\omega,\fz)$ and $f\in(\GOO{\bz,\gz})'$, we
always have $\|(\dot g_\lz^*)^s_q(f)\|_{L^p(X)}=\|[(\dot g_\lz^*)^s_q(f)]^q\|_{L^{p/q}(X)}^{1/q}$. By this,
the Tonelli theorem, $\lz>\omega$, \cite[Proposition 2.2(ii)]{hlyy19}, and the boundedness of $\CM$ on
$L^{(p/q)'}(X)$ (see, for instance \cite[(3.6)]{cw77}), we conclude that, for any $f\in\dot F^s_{p,q}(X)$
and any non-negative function $\vz\in L^{(p/q)'}(X)$,
\begin{align*}
\lf<\lf[\lf(\dot g_\lz^*\r)^s_q(f)\r]^q,\vz\r>&\ls\int_X\sum_{k=-\fz}^\fz\dz^{-ksq}|Q_kf(y)|^q\CM(\vz)(y)\,d\mu(y)\\
&\ls\lf\|\sum_{k=-\fz}^\fz\dz^{-ksq}|Q_kf|^q\r\|_{L^{p/q}(X)}\|\CM(\vz)\|_{L^{(p/q)'}(X)}
\ls\lf\|\dot g^s_q(f)\r\|_{L^p(X)}^q\|\vz\|_{L^{(p/q)'}(X)}.
\end{align*}
Taking the supremum over all such $\vz$ with $\|\vz\|_{L^{(p/q)'}(X)}\le 1$,
and using \eqref{eq-gfun}, we further obtain
$$
\lf\|\lf(\dot g_\lz^*\r)^s_q(f)\r\|_{L^p(X)}
=\lf\|\lf[\lf(\dot g_\lz^*\r)^s_q(f)\r]^q\r\|_{L^{p/q}(X)}^{1/q}
\ls\lf\|\dot g^s_q(f)\r\|_{L^p(X)}\sim\|f\|_{\dot F^s_{p,q}(X)}.
$$
This finishes the proof of the necessity of this proposition in the case $p>q$,
and hence of Proposition \ref{prop-glp>q}.
\end{proof}

Next, we consider the case $p<q$. To this end, we first introduce the following Littlewood--Paley auxiliary
function. Let $f\in(\GOO{\bz,\gz})'$ with $\bz,\ \gz\in(0,\eta)$, and $\thz\in[1,\fz)$,
where $\eta$ is the same
as in Definition \ref{def-eti}. The
\emph{Littlewood--Paley auxiliary function $\dot\CS^{s,(1)}_{q,\thz}(f)$ of $f$ with aperture
$\thz$} is defined by setting, for any $x\in X$,
$$
\dot\CS^{s,(1)}_{q,\thz}(f)(x):=\lf[\sum_{k\in\zz}\dz^{-ksq}\int_{B(x,\thz\dz^k)}|Q_kf(y)|^q
\,\frac{d\mu(y)}{V_{\dz^k}(y)}\r]^{1/q}.
$$
It is obvious that there exists a constant $C\in[1,\fz)$ such that, for any $f\in(\GOO{\bz,\gz})'$ with
$\bz,\ \gz\in(0,\eta)$, and $\eta$ as in Definition \ref{def-eti},
\begin{equation}\label{eq-ss}
C^{-1}\dot\CS^s_q(f)\le\dot\CS^{s,(1)}_{q,1}(f)\le C\dot\CS^s_q(f).
\end{equation}
We have the following change-of-angle formula of $\dot\CS^{s,(1)}_{q,\thz}$ on $\thz\in(1,\fz)$.

\begin{proposition}\label{prop-angle}
Let $p\in(0,\fz)$, $q\in(0,p)$, $\omega$ be as in \eqref{eq-doub}, and $\eta$ as in Definition \ref{def-eti}. Then there exists a positive
constant $C$ such that, for any $\thz\in(1,\fz)$ and $f\in(\GOO{\bz,\gz})'$ with $\bz,\ \gz\in(0,\eta)$,
$\|\dot\CS^{s,(1)}_{q,\thz}(f)\|_{L^p(X)}\le C\thz^{\omega/p}\|\dot\CS^{s,(1)}_{q,1}(f)\|_{L^p(X)}$.
\end{proposition}

\begin{proof}
Fix $\thz\in(1,\fz)$ and let $s$, $p$, and $q$ be the same as in this proposition.
For any non-negative function $g$ and any $x\in X$, let
$$
\wz\CM(g)(x):=\sup_{k\in\zz}\sup_{d(x,y)<\thz\dz^k}\frac{1}{V_{\dz^k}(y)}\int_{B(y,\dz^k)}g(z)
\,d\mu(z).
$$
It is easy to see that, for any $k\in\zz$, $y\in B(x,\thz\dz^k)$, and $z\in B(y,\dz^k)$,
$d(x,z)\le A_0[d(x,y)+d(y,z)]<2A_0\thz\dz^k$,
which further implies that $B(y,\dz^k)\subset B(x,2A_0\thz\dz^k)$. By this, we conclude that, for any
$k\in\zz$, $x\in X$, and $y\in B(x,\thz\dz^k)$,
\begin{align*}
\frac{1}{V_{\dz^k}(y)}\int_{B(y,\dz^k)}g(z)\,d\mu(z)&\le\frac{V_{2A_0\thz\dz^k}(x)}{V_{\dz^k}(y)}
\frac{1}{V_{2A_0\thz\dz^k}(x)}\int_{B(x,2A_0\thz\dz^k)}g(z)\,d\mu(z)\\
&\ls\frac{V_{2A_0\thz\dz^k}(y)}{V_{\dz^k}(y)}\CM(g)(x)\ls\thz^{\omega}\CM(g)(x),
\end{align*}
which, together with the boundedness of $\CM$ from $L^1(X)$ to $L^{1,\fz}(X)$ (see, for instance
\cite[pp.\ 71--72, Theorem 2.1]{cw71}), further implies that, for any
$r\in(0,\fz)$,
\begin{equation}\label{eq-w11}
\mu\lf(\lf\{x\in X:\ \wz\CM(g)(x)>r\r\}\r)\ls\frac{\thz^\omega}{r}\|g\|_{L^1(X)}.
\end{equation}

Next, for any $t\in(0,\fz)$ and $f\in(\GOO{\bz,\gz})'$ with $\bz$ and $\gz$ as in this proposition, let
$$
E_t:=\lf\{x\in X:\ \dot\CS^{s,(1)}_{q,1}(f)(x)>t\r\}
$$
and $\wz{E}_t:=\{x\in X:\ \wz\CM({\mathbf 1}_{E_t})(x)>1/2\}$. We claim that, for any $t\in(0,\fz)$,
\begin{equation}\label{eq-cl}
\int_{\wz E_t^\complement}\lf[\dot\CS^{s,(1)}_{q,\thz}(f)(x)\r]^q\,d\mu(x)
\ls\thz^\omega\int_{E_t^\complement}\lf[\dot\CS^{s,(1)}_{q,1}(f)(x)\r]^q\,d\mu(x).
\end{equation}
Assuming this for the moment, we use the Chebyshev inequality, \eqref{eq-w11} with $g:=\mathbf 1_{E_t}$ therein,
and \eqref{eq-cl} to conclude that, for any $t\in(0,\fz)$,
\begin{align*}
\mu\lf(\lf\{x\in X:\ \dot\CS^{s,(1)}_{q,\thz}(f)(x)>t\r\}\r)&\le\mu\lf(\wz{E}_t\r)
+\mu\lf(\lf\{x\in\wz{E}_t^\complement:\ \dot\CS^{s,(1)}_{q,\thz}(f)(x)>t\r\}\r)\\
&\ls\thz^\omega\mu\lf(E_t\r)+t^{-q}\thz^\omega\int_{E_t^\complement}\lf[\dot\CS^{s,(1)}_{q,1}(f)(x)\r]^q
\,d\mu(x)\\
&\ls\thz^\omega\lf[\mu\lf(E_t\r)+t^{-q}\int_0^t r^{q-1}\mu(E_r)\,dr\r].
\end{align*}
Thus, by this, the Tonelli theorem, and $p\in(0,q)$, we further find that
\begin{align*}
\lf\|\dot\CS^{s,(1)}_{q,\thz}(f)\r\|_{L^p(X)}^p
&\ls\thz^\omega\lf[\int_0^\fz t^{p-1}\mu(E_t)\,dt+\int_0^\fz t^{p-q-1}\int_0^t r^{q-1}\mu(E_r)\,dr\,dt\r]\\
&\sim\thz^\omega\lf\|\dot\CS^{s,(1)}_{q,1}(f)\r\|_{L^p(X)}^p+\thz^\omega\int_0^\fz r^{p-1}\mu(E_r)\,dr
\sim\thz^\omega\lf\|\dot\CS^{s,(1)}_{q,1}(f)\r\|_{L^p(X)}^p.
\end{align*}
This finishes the proof of Proposition \ref{prop-angle} under the assumption \eqref{eq-cl}.

It remains to show \eqref{eq-cl}. Fix $t\in(0,\fz)$ and, for any $y\in X$, let
$\rho(y):=\inf_{x\in\wz E_t^\complement}d(x,y)$. It is then easy to see that, for any $k\in\zz$ and $x,\ y\in X$,
$x\in\wz E_t^\complement\cap B(y,\thz\dz^k)$ implies that $\rho(y)<\thz\dz^k$. By this and the
Tonelli theorem, we obtain
\begin{align}\label{eq-1}
\int_{\wz E_t^\complement}\lf[\dot\CS^{s,(1)}_{q,\thz}(f)(x)\r]^q\,d\mu(x)
&=\sum_{k=-\fz}^\fz\dz^{-ksq}\int_{\rho(y)<\thz\dz^k}|Q_kf(y)|^q
\mu\lf(\wz E_t^\complement\cap B(y,\thz\dz^k)\r)\,\frac{d\mu(y)}{V_{\dz^k}(y)}\noz\\
&\ls\thz^\omega\sum_{k=-\fz}^\fz\dz^{-ksq}\int_{\rho(y)<\thz\dz^k}|Q_kf(y)|^q\mu(B(y,\dz^k))
\,\frac{d\mu(y)}{V_{\dz^k}(y)}.
\end{align}
Notice that, if $\rho(y)<\thz\dz^k$, then $\wz E_t^\complement\cap B(y,\thz\dz^k)\neq\emptyset$. We
can then choose a $y_0\in\wz E_t^\complement\cap B(y,\thz\dz^k)$ to conclude that
$$
\mu(E_t\cap B(y,\dz^k))=\int_{B(y,\dz^k)}{\mathbf 1}_{E_t}(z)\,d\mu(z)
\le\mu(B(y,\dz^k))\wz\CM\lf({\mathbf 1}_{E_t}\r)(y_0)\le\frac 12\mu(B(y,\dz^k)).
$$
Thus, $\mu(E_t^\complement\cap B(y,\dz^k))\ge \frac 12\mu(B(y,\dz^k))$. This, combined with \eqref{eq-1} and
the Tonelli theorem, further implies that
\begin{align*}
\int_{\wz E_t^\complement}\lf[\dot\CS^{s,(1)}_{q,\thz}(f)(x)\r]^q\,d\mu(x)
&\ls\thz^\omega\sum_{k=-\fz}^\fz\dz^{-ksq}\int_{X}|Q_kf(y)|^q\mu\lf(E_t^\complement\cap B(y,\dz^k)\r)
\,\frac{d\mu(y)}{V_{\dz^k}(y)}\\
&\sim\thz^\omega\sum_{k=-\fz}^\fz\dz^{-ksq}\int_{E_t^\complement}\int_{B(x,\dz^k)}|Q_kf(y)|^q
\,\frac{d\mu(y)}{V_{\dz^k}(y)}\,d\mu(x)\\
&\sim\thz^\omega\int_{E_t^\complement}\lf[\dot\CS^{\az,(1)}_{q,1}(f)(x)\r]^q\,d\mu(x).
\end{align*}
This finishes the proof of \eqref{eq-cl} and hence of Proposition \ref{prop-angle}.
\end{proof}

With the help of Proposition \ref{prop-angle}, we have the following Littlewood--Paley $g_\lz^*$-function
characterization of Triebel--Lizorkin spaces in the case $q\in(p,\fz)$.

\begin{proposition}\label{prop-glp<q}
Let $s$, $p$, $q$, $\bz$, and $\gz$ be as in Definition \ref{def-btl}(ii), $q\in(p,\fz)$,
and $\lz\in(q\omega_0/p,\fz)$ with $\omega_0$ as in \eqref{eq-updim}. Then $f\in \dot F^s_{p,q}(X)$ if and only if
$f\in(\GOO{\bz,\gz})'$ and $(\dot g_\lz^*)^s_q\in L^p(X)$. Moreover, there exists a constant $C\in[1,\fz)$,
independent of $f$, such that
$C^{-1}\|f\|_{\dot F^s_{p,q}(X)}\le\|(\dot g_\lz^*)^s_q(f)\|_{L^p(X)}\le
C\|f\|_{\dot F^s_{p,q}(X)}$.
\end{proposition}

\begin{proof}
Let all the notation be as in this proposition. The sufficiency of this proposition
holds true directly due to Theorem \ref{thm-g=s} and the fact $(\dot g_\lz^*)^s_q(f)\ge\dot\CS^s_q(f)$;
we omit the details. We only prove the necessity of this proposition. Let $s$, $p$, $q$, and $\lz$
be as in this proposition. To this end, we have, for any $f\in\dot F^s_{p,q}(X)$ and $x\in X$,
\begin{align*}
\lf[\lf(\dot g_\lz^*\r)^s_q(f)(x)\r]^q&\ls\lf[\dot\CS^{\az,(1)}_{q,1}(f)(x)\r]^q+\sum_{j=1}^\fz 2^{-j\lz}
\sum_{k=-\fz}^\fz\dz^{-ksq}\int_{2^{j-1}\dz^k\le d(x,y)<2^j\dz^k}|Q_kf(y)|^q\,\frac{d\mu(y)}{V_{\dz^k}(y)}\\
&\ls\sum_{j=0}^\fz 2^{-j\lz}\lf[\dot\CS^{s,(1)}_{q,2^j}(f)(x)\r]^q.
\end{align*}
Choose $\omega\in[\omega_0,\fz)$ satisfying \eqref{eq-doub} and all the assumptions of this proposition.
By this, $p<q$, Proposition \ref{prop-angle}, \eqref{eq-ss}, and Theorem \ref{thm-g=s}, we conclude that,
for any $f\in\dot F^s_{p,q}(X)$,
\begin{align*}
\lf\|\lf(\dot g_\lz^*\r)^s_q(f)\r\|_{L^p(X)}^p&\ls\sum_{j=0}^\fz 2^{-j\lz p/q}
\int_X\lf[\dot\CS^{s,(1)}_{q,2^j}(f)(x)\r]^p\,d\mu(x)\ls\sum_{j=0}^\fz
2^{-j\lz p/q}2^{j\omega}\lf\|\dot\CS^{s,(1)}_{q,1}(f)\r\|_{L^p(X)}^p\\
&\sim\lf\|\dot\CS^{s,(1)}_{q,1}(f)\r\|_{L^p(X)}^p\sim \lf\|\dot\CS^s_q(f)\r\|_{L^p(X)}^p
\sim\|f\|_{\dot F^s_{p,q}(X)}^p
\end{align*}
due to $\lz>q\omega/p$. This finishes the proof of the necessity
of this proposition, and hence of
Proposition \ref{prop-glp<q}.
\end{proof}

Combining Propositions \ref{prop-glp<q} and \ref{prop-glp>q}, we
directly obtain the following  Littlewood--Paley $g_\lz^*$-function
characterization of $\dot F^s_{p,q}(X)$; we omit
the details here.

\begin{theorem}\label{thm-glstar}
Let $s$, $p$, $q$, $\bz$, and $\gz$ be as in Definition \ref{def-btl}(ii), $q\in(p(s,\bz\wedge\gz),\fz)$, and
$\lz\in(\max\{\omega_0,q\omega_0/p\},\fz)$ with $\omega_0$ as in \eqref{eq-updim}. Then $f\in\dot F^s_{p,q}(X)$ if and only if
$f\in(\GOO{\bz,\gz})'$ and $(\dot g_\lz^*)^s_q(f)\in L^p(X)$. Moreover, there exists a constant $C\in[1,\fz)$,
independent of $f$, such that
$$
C^{-1}\lf\|\lf(\dot g_\lz^*\r)^s_q(f)\r\|_{L^p(X)}\le \|f\|_{\dot F^s_{p,q}(X)}
\le C\lf\|\lf(\dot g_\lz^*\r)^s_q(f)\r\|_{L^p(X)}.
$$
\end{theorem}

\begin{remark}\label{rem-glz}
P\"{a}iv\"{a}rinta \cite{pai82} showed that, when $\lz\in(nq/\min\{p,q\}+n,\fz)$, the Littlewood--Paley
$g_\lz^*$-function characterization of the Triebel--Lizorkin space $F^s_{p,q}(\rn)$ holds true by observing that
$\lz$ in \eqref{eq-glstar} equals to $q\lz$ in \cite{pai82}. Then it is easy to see that Theorem
\ref{thm-glstar} is better than \cite[Remark 2.6]{pai82} (see also \cite[pp.\ 182--183, Theorem 2.12.1]{tri83}). In
2013, Hu \cite{hu13} proved that, if $X$ is a stratified Lie group, then the Littlewood--Paley
$g_\lz^*$-function characterization of Triebel--Lizorkin spaces holds true when $\lz\in(\max\{\omega_0 q/p,\omega_0\},\fz)$
with $\omega_0$ as in \eqref{eq-updim},
by also observing that $\lz$ in \eqref{eq-glstar} equals to $q\lz$ in \cite{hu13} (see \cite[Proposition 14]{hu13}).
Therefore, when $X$ is a stratified Lie group which is a space of homogeneous type,
Theorem \ref{thm-glstar} coincides with \cite[Proposition 14]{hu13}.
Moreover, since $\dot F^0_{p,2}(X)=H^p(X)$ (the Hardy space)
when $p\in(\omega_0/(\omega_0+\eta),1]$ with $\eta$ as in Definition \ref{def-eti},
we know that, in this special case,
Theorem \ref{thm-glstar} coincides with \cite[Theorem 5.12]{hhllyy19}. Thus, the range of $\lz$ in Theorem \ref{thm-glstar}
is the known best possible.
\end{remark}

\section{Inhomogeneous counterparts}\label{s-in}

In this section, we consider inhomogeneous counterparts of results in Sections \ref{s-wave} through
\ref{s-lp}, which are listed  in three subsections. In the first subsection, we use the inhomogeneous wavelet
system to construct an exp-IATI. In the second subsection, we establish the wavelet
characterization of inhomogeneous Besov and Triebel--Lizorkin spaces. Finally, in the third subsection,
we extend the results in Sections \ref{s-ado}, \ref{s-mol}, and \ref{s-lp} to the inhomogeneous case.
From now on, we do \emph{not} need to assume $\mu(X)=\fz$, namely, $\mu(X)$ can be finite or
infinite.

\subsection{Existence of exp-IATIs}\label{ss-ex}

In this subsection, we use the inhomogeneous wavelet system on a given space $X$ of
homogeneous type to construct an exp-IATI.

For any $k\in\zz$, let $\{\phi_\az^k\}_{\az\in\CA_k}$ be as in \cite[Theorem 6.1]{ah13}.
Then we have the following conclusion.

\begin{proposition}\label{prop-ieti1}
Let $\eta$ be the same as in Definition \ref{def-eti} and, for any $k\in\zz_+$ and $x,\ y\in X$,
$P_k(x,y):=\sum_{\az\in\CA_k}\phi_\az^k(x)\phi_\az^k(y)$. Then
there exist constants $C,\ \nu\in(0,\fz)$ and $a\in(0,1]$ such that, for any $k\in\zz_+$, $P_k$ has the following
properties:
\begin{enumerate}
\item (the \emph{symmetry}) for any $x,\ y\in X$, $P_k(x,y)=P_k(y,x)$;
\item (the \emph{size condition}) for any $x,\ y\in X$,
$$
|P_k(x,y)|\le C\frac{1}{\sqrt{V_{\dz^k}(x)V_{\dz^k}(y)}}\exp\lf\{-\nu\lf[\frac{d(x,y)}{\dz^k}\r]^a\r\}=:C\wz E_k(x,y);
$$
\item (the \emph{H\"{o}lder regularity condition}) for any $x,\ x',\ y\in X$ with $d(x,x')\le\dz^k$,
$$
|P_k(x,y)-P_k(x',y)|\le C\lf[\frac{d(x,x')}{\dz^k}\r]^\eta\wz E_k(x,y);
$$
\item (the \emph{second difference regularity condition}) for any $x,\ x',\ y,\ y'\in X$ with
$d(x,x')\le\dz^k$ and $d(y,y')\le\dz^k$,
\begin{align*}
|[P_k(x,y)-P(x',y)]-[P_k(x,y')-P_k(x,y')]|\le C\lf[\frac{d(x,x')}{\dz^k}\r]^\eta\lf[\frac{d(y,y')}{\dz^k}\r]^\eta
\wz E_k(x,y);
\end{align*}
\item (the \emph{conservation law}) for any $x\in X$,
$$
\int_X P_k(x,y)\,d\mu(y)=1.
$$
\end{enumerate}
\end{proposition}

\begin{proof}
The proof of (i) is obvious by the definition of $\{P_k\}_{k=0}^\fz$. For (ii), by
\cite[Theorem 6.1 and Lemma 6.4]{ah13}, we find that, for any $k\in\zz_+$ and $x,\ y\in X$,
\begin{align*}
\sum_{\az\in\CA_k}\lf|\phi_\az^k(x)\phi_\az^k(y)\r|&\ls\sum_{\az\in\CA_k}\frac 1{V_{\dz^k}(z_\az^k)}
\exp\lf\{-c\lf[\frac{d(x,z_\az^k)}{\dz^k}\r]^a\r\}\exp\lf\{-c\lf[\frac{d(y,z_\az^k)}{\dz^k}\r]^a\r\}\\
&\ls\frac{1}{\sqrt{V_{\dz^k}(x)V_{\dz^k}(y)}}\exp\lf\{-c'\lf[\frac{d(x,y)}{\dz^k}\r]^a\r\}.
\end{align*}
Here and thereafter, $c\in(0,\fz)$ and $c'\in(0,c)$, which both are independent of $k$, $x$, and
$y$. The proofs of (iii) and (iv) are similar to that of (i); we omit the details here.

Finally, we prove (v). Fix $k\in\zz$. Let $s_\az^k$, for any $\az\in\CA_k$, be the same as in \cite[(3.1)]{ah13},
and $V_k$  the same as in \cite[Theorem 5.1]{ah13}. Then
$s_\az^k\in V_k$. Moreover, since $\{\phi_\bz^k\}_{\bz\in\CA_k}$ is an orthonormal basis of $V_k$
(see \cite[Theorem 6.1]{ah13}), it then follows that, for any $\az\in\CA_k$ and almost every $x\in X$,
\begin{equation}\label{eq-orthdec}
s_\az^k(x)=\sum_{\bz\in\CA_k}\lf<s_\az^k,\phi_\bz^k\r>\phi_\bz^k(x)=
\sum_{\bz\in\CA_k}\int_X s_\az^k(y)\phi_\bz^k(y)\,d\mu(y)\phi_\bz^k(x).
\end{equation}
By \cite[Theorems 3.1 and 6.1, and Lemma 6.4]{ah13}, we find that, for any $x\in X$,
\begin{align*}
&\sum_{\az\in\CA_k}\sum_{\bz\in\CA_k}\int_X\lf|s_\az^k(y)\phi_\bz^k(y)\r|\,d\mu(y)\lf|\phi_\bz^k(x)\r|\\
&\quad\ls\sum_{\bz\in\CA_k}\frac 1{\sqrt{V_{\dz^k}(x_\bz^k)}}
\exp\lf\{-c\lf[\frac{d(x,x_\bz^k)}{\dz^k}\r]^a\r\}\int_{X}
\frac 1{\sqrt{V_{\dz^k}(x_\bz^k)}}\exp\lf\{-c\lf[\frac{d(y,x_\bz^k)}{\dz^k}\r]^a\r\}\,d\mu(y)\\
&\quad\ls\sum_{\bz\in\CA_k}\frac{1}{V_{\dz^k}(x_\bz^k)}
\exp\lf\{-c'\lf[\frac{d(x,x_\bz^k)}{\dz^k}\r]^a\r\}\ls 1.
\end{align*}
From this, the Fubini theorem, \eqref{eq-orthdec}, and \cite[Theorem 3.1]{ah13}, we deduce that,
for almost every $x\in X$,
\begin{align*}
1&=\sum_{\az\in\CA_k}s_\az^k(x)=\sum_{\az\in\CA_k}\sum_{\bz\in\CA_k}\int_X s_\az^k(y)\phi_\bz^k(y)\,d\mu(y)\phi_\bz^k(x)\\
&=\int_X\sum_{\bz\in\CA_k}\phi_\bz^k(x)\phi_\bz^k(y)\,d\mu(y)=\int_X P_k(x,y)\,d\mu(y).
\end{align*}
Noting that the function $F(\cdot):=\int_X P_k(\cdot,y)\,d\mu(y)$ is continuous on $X$, we then conclude that
(v) holds true for any $x\in X$. This finishes the proof of Proposition \ref{prop-ieti1}.
\end{proof}

Using this, we construct an exp-IATI on $X$, which is stated in the following theorem.

\begin{theorem}\label{thm-eieti}
For any $k\in\zz_+$ and $x,\ y\in X$, let
$$
Q_k(x,y):=\begin{cases}
\displaystyle \sum_{\az\in\CA_0}\phi_\az^0(x)\phi_\az^0(y) & \textit{if } k=0,\\
\displaystyle \sum_{\az\in\CG_{k-1}}\psi_\az^{k-1}(x)\psi_\az^{k-1}(y) & \textit{if } k\in\nn.\\
\end{cases}
$$
Then $\{Q_k\}_{k=0}^\fz$ is an {\rm exp-IATI}. Moreover, for any $f\in(\go{\bz,\gz})'$ with
$\bz,\ \gz\in(0,\eta)$,
\begin{equation}\label{eq-iwave}
f=\sum_{\az\in\CA_0}\lf<f,\phi_\az^0\r>\phi_\az^0+\sum_{k\in\nn}\sum_{\az\in\CG_{k-1}}
\lf<f,\psi_\az^{k-1}\r>\psi_\az^{k-1}
\end{equation}
in $(\go{\bz,\gz})'$, where $\eta$ is the same as in Definition \ref{def-eti}.
\end{theorem}

\begin{proof}
For any $k\in\zz_+$, let $V_k$ be the same as in \cite[Theorem 5.1]{ah13}, $W_{k-1}$ the complement of $V_{k-1}$
in $V_{k}$, and $P_k$ the same as in Proposition \ref{prop-ieti1}. Fix
$k\in\zz_+$. Note that, by \cite[Theorem 6.1]{ah13}, $\{\phi_\az^k\}_{\az\in\CA_k}$ is an orthogonal basis of $V_k$,
which further implies that $P_k$ is an orthonormal projection onto $V_k$. Moreover, since
$\{\psi_\az^{k-1}\}_{\az\in\CG_{k-1}}$ is an orthogonal basis of $W_{k-1}$, it then follows that $Q_{k-1}$ is an
orthonormal projection onto $W_{k-1}$. Therefore, by the proof of \cite[Theorem 10.2]{ah13}, we know that,
for any $k\in\nn$, $P_k=P_{k-1}+Q_{k-1}$. For any $f\in L^2(X)$, since $\lim_{k\to\fz}P_kf=f$ in $L^2(X)$
(see \cite[Proposition 2.7]{hmy08}), it then follows that
$$
\sum_{k=0}^\fz Q_kf=P_0f+\sum_{k=1}^\fz(P_k-P_{k-1})f=f\quad\textup{in}\quad L^2(X).
$$
By this, Proposition \ref{prop-ieti1}, and \cite[Lemma 3.6]{hlw18}, we find that $\{Q_k\}_{k=0}^\fz$ is an
exp-IATI. Moreover, using an argument similar to that used in the proof of \cite[Theorem 3.6]{hlw18}, we obtain
\eqref{eq-iwave}. This finishes the proof of Theorem \ref{thm-eieti}.
\end{proof}

\subsection[Wavelet characterization of inhomogeneous Besov and Triebel--Lizorkin spaces]
{Wavelet characterization of inhomogeneous Besov and \\ Triebel--Lizorkin spaces}
\label{ss-iwave}

In this subsection, we establish the wavelet characterization of inhomogeneous Besov and
Triebel--Lizorkin spaces. First, we recall the notion of these inhomogeneous spaces.
To this end, for any dyadic cube $Q$ and any non-negative measurable function $f$ on $X$, let
\begin{equation}\label{eq-defmean}
m_Q(f):=\frac 1{\mu(Q)}\int_Q f(y)\,d\mu(y).
\end{equation}

\begin{definition}\label{def-ibtl}
Let $\{Q_k\}_{k=0}^\fz$ be an exp-IATI, and $s\in(-\eta,\eta)$ with $\eta$ as in Definition \ref{def-eti}.
Suppose that $p\in(p(s,\eta),\fz]$, $q\in(0,\fz]$, and $\bz$ and $\gz$ satisfy
\begin{equation*}
\bz\in\lf(\max\left\{0,-s+\omega_0\lf(\frac{1}{p}-1\r)_+\right\},\eta\r)
\quad\hbox{and}\quad
\gz\in\lf(\omega_0\lf(\frac{1}{p}-1\r)_+,\eta\r)
\end{equation*}
with $\omega_0$ as in \eqref{eq-updim}, and $N$ is the same as in Lemma \ref{lem-idrf}.
\begin{enumerate}
\item If $s\in(-(\bz\wedge\gz),\bz\wedge\gz)$, $p\in(p(s,\bz\wedge\gz),\fz]$, and $q\in(0,\fz]$, then
the \emph{inhomogeneous Besov space $B^s_{p,q}(X)$} is defined to be set of all $f\in(\go{\bz,\gz})'$ such that
\begin{align*}
\|f\|_{B^{s}_{p,q}(X)}:=&\,\lf\{\sum_{k=0}^N\sum_{\az\in\CA_k}\sum_{m=1}^{N(k,\az)}\mu\lf(Q_\az^{k,m}\r)
\lf[m_{Q_\az^{k,m}}(|Q_kf|)\r]^p\r\}^{1/p}+\lf[\sum_{k=N+1}^\fz\dz^{-ksq}\|Q_kf\|_{L^p(X)}^q\r]^{1/q}\\
<&\,\fz
\end{align*}
with the usual modifications made when $p=\fz$ or $q=\fz$.
\item If $s\in(-(\bz\wedge\gz),\bz\wedge\gz)$, $p\in(p(s,\bz\wedge\gz),\fz)$, and
$q\in(p(s,\bz\wedge\gz),\fz]$, then the \emph{inhomogeneous Triebel--Lizorkin space $F^s_{p,q}(X)$} is defined
to be the set of all $f\in(\go{\bz,\gz})'$ such that
\begin{align*}
\|f\|_{F^{s}_{p,q}(X)}&:=\lf\{\sum_{k=0}^N\sum_{\az\in\CA_k}\sum_{m=1}^{N(k,\az)}\mu\lf(Q_\az^{k,m}\r)
\lf[m_{Q_\az^{k,m}}(|Q_kf|)\r]^p\r\}^{1/p}+\lf\|\lf(\sum_{k=N+1}^\fz\dz^{-ksq}|Q_kf|^q\r)^{1/q}\r\|_{L^p(X)}\\
&<\fz
\end{align*}
with the usual modification made when $q=\fz$.
\end{enumerate}
\end{definition}

On $B^s_{p,q}(X)$ and $F^s_{p,q}(X)$, we have the following wavelet characterization.

\begin{theorem}\label{thm-w=q}
Let $\eta$ be the same as in Definition \ref{def-eti}.
\begin{enumerate}
\item If $s$, $p$, $q$, $\bz$, and $\gz$ are as in Definition \ref{def-ibtl}(i), then $f\in B^s_{p,q}(X)$
if and only if $f\in(\go{\bz,\gz})'$ and
\begin{align}\label{eq-defib}
\|f\|_{B^s_{p,q}(\mathrm w,X)}:=&\,\lf\{\sum_{\az\in\CA_0}\lf[\mu\lf(Q_\az^0\r)\r]^{1-p/2}\lf|\lf<f,\phi_\az^0\r>
\r|^p\r\}^{1/p}\noz\\
&\qquad+\lf\{\sum_{k\in\nn}\dz^{-ksq}\lf[\sum_{\az\in\CG_{k-1}}\lf[\mu\lf(Q_\az^{k}\r)\r]^{1-p/2}
\lf|\lf<f,\psi_\az^{k-1}\r>\r|^p\r]^{q/p}\r\}^{1/q}\noz\\
<&\,\fz
\end{align}
with the usual modifications made when $p=\fz$ or $q=\fz$. Moreover, there exists a constant $C\in[1,\fz)$,
independent of $f$, such that
$C^{-1}\|f\|_{B^s_{p,q}(\mathrm w,X)}\le\|f\|_{B^s_{p,q}(X)}\le C\|f\|_{B^s_{p,q}(\mathrm w,X)}$.

\item If $s$, $p$, $q$, $\bz$, and $\gz$ are as in Definition \ref{def-ibtl}(ii), then $f\in F^s_{p,q}(X)$ if
and only if $f\in(\go{\bz,\gz})'$ and
\begin{align}\label{eq-defitl}
\|f\|_{F^s_{p,q}(\mathrm w,X)}:=&\,\lf\{\sum_{\az\in\CA_0}\lf[\mu\lf(Q_\az^0\r)\r]^{1-p/2}\lf|\lf<f,\phi_\az^0\r>
\r|^p\r\}^{1/p}\noz\\
&\qquad+\lf\|\lf(\sum_{k\in\nn}\dz^{-ksq}\lf[\sum_{\az\in\CG_{k-1}}\lf|\lf<f,\psi_\az^{k-1}\r>
\widetilde{\mathbf 1}_{Q_\az^{k}}\r|\r]^q\r)^{1/q}\r\|_{L^p(X)}\noz\\
<&\,\fz
\end{align}
with the usual modification made when $q=\fz$. Moreover, there exists a constant $C\in[1,\fz)$, independent
of $f$, such that
$C^{-1}\|f\|_{F^s_{p,q}(\mathrm w,X)}\le\|f\|_{F^s_{p,q}(X)}\le C\|f\|_{F^s_{p,q}(\mathrm w,X)}$.
\end{enumerate}
\end{theorem}

\begin{proof}
Let $s$, $p$, and $q$ be as in this theorem. By \cite[Proposition 4.3]{whhy20}, we find that $B^s_{p,q}(X)$
and $F^s_{p,q}(X)$ are independent of the choice of exp-IATIs. Therefore, we may assume that
$\{Q_k\}_{k=0}^\fz$ in both \eqref{eq-defib} and \eqref{eq-defitl} are the same as in Theorem \ref{thm-eieti}.

We only prove (i). The proof of (ii) is similar to that of (i) with some slight modifications. We omit the
details here. We first suppose $f\in B^s_{p,q}(X)$. Then, by Definition \ref{def-ibtl}, we know that $f\in(\go{\bz,\gz})'$
with $\bz$ and $\gz$ as in
this theorem. Therefore, from Lemma \ref{lem-idrf} (with the same notion as therein),
it follows that there exist $\{\wz Q_k\}_{k=0}^\fz$ such that
\begin{align*}
f(\cdot)&=\sum_{\az\in\CA_0}\sum_{m=1}^{N(0,\az)}\int_{Q_\az^{0,m}}\wz Q_0(\cdot,y)\,d\mu(y)m_{Q_\az^{0,m}}
(Q_0f)\\
&\qquad+\sum_{k=1}^N\sum_{\az\in\CA_k}\sum_{m=1}^{N(k,\az)}\mu\lf(Q_\az^{k,m}\r)\wz Q_k\lf(\cdot,y_\az^{k,m}\r)
m_{Q_\az^{k,m}}(Q_kf)\\
&\qquad+\sum_{k=N+1}^\fz\sum_{\az\in\CA_k}\sum_{m=1}^{N(k,\az)}\mu\lf(Q_\az^{k,m}\r)
\wz Q_k\lf(\cdot,y_\az^{k,m}\r)Q_kf\lf(y_\az^{k,m}\r)
\end{align*}
in $(\go{\bz,\gz})'$. Now, we consider two cases on $k'\in\zz_+$.

{\it Case 1) $k'=0$.} In this case, we conclude that, for any $\az'\in\CA_0$,
\begin{align*}
\lf<f,\frac{\phi_{\az'}^0}{\sqrt{\mu(Q_{\az'}^0)}}\r>
&=\sum_{\az\in\CA_0}\sum_{m=1}^{N(0,\az)}\int_{Q_\az^{0,m}}
\wz Q_0^*\lf(\frac{\phi_{\az'}^0}{\sqrt{\mu(Q_{\az'}^0)}}\r)(y)\,d\mu(y)m_{Q_\az^{0,m}}(Q_0f)\\
&\qquad+\sum_{k=1}^N\sum_{\az\in\CA_k}\sum_{m=1}^{N(k,\az)}\mu\lf(Q_\az^{k,m}\r)
\wz Q_k^*\lf(\frac{\phi_{\az'}^0}{\sqrt{\mu(Q_{\az'}^0)}}\r)\lf(y_\az^{k,m}\r)m_{Q_\az^{k,m}}(Q_kf)\\
&\qquad+\sum_{k=N+1}^\fz\sum_{\az\in\CA_k}\sum_{m=1}^{N(k,\az)}\mu\lf(Q_\az^{k,m}\r)
\wz Q_k^*\lf(\frac{\phi_{\az'}^0}{\sqrt{\mu(Q_{\az'}^0)}}\r)\lf(y_\az^{k,m}\r)Q_kf\lf(y_\az^{k,m}\r)\\
&=:\RI_1+\mathrm{II}_1+\mathrm{III}_1,
\end{align*}
where, for any $k\in\nn$, $\az\in\CA_k$, and $m\in\{1,\ldots,N(k,\az)\}$, $y_\az^{k,m}$
is an arbitrary point of $Q_\az^{k,m}$.

For $\RI_1$, from \cite[Lemma 3.6]{hyy19} and \cite[Theorem 6.1]{ah13}, we deduce that, for any $\az\in\CA_0$,
$m\in\{1,\ldots,N(0,\az)\}$, and $y\in Q_\az^{0,m}$,
$$
\lf|\wz Q_0^*\lf(\frac{\phi_{\az'}^0}{\sqrt{\mu(Q_{\az'}^0)}}\r)(y)\r|
\ls P_\gz\lf(y_\az^{m,0},x_{\az'}^0;1\r).
$$
Thus, we conclude that
$$
|\RI_1|\ls\sum_{\az\in\CA_0}\sum_{m=1}^{N(0,\az)}\mu\lf(Q_{\az}^0\r)P_\gz\lf(y_\az^{0,m},x_{\az'}^0;1\r)
m_{Q_\az^{0,m}}(|Q_kf|).
$$

For $\mathrm{II}_1$, since $N$ is a fixed integer, from an argument similar to that used in the estimation of
$\RI_1$, it then follows that
$$
|\mathrm{II}_1|\ls\sum_{k=1}^N\sum_{\az\in\CA_k}\sum_{m=1}^{N(k,\az)}\mu\lf(Q_{\az}^{k,m}\r)
P_\gz\lf(y_\az^{k,m},x_{\az'}^0;1\r)m_{Q_\az^{k,m}}(|Q_kf|).
$$

For $\mathrm{III}_1$, by \cite[Lemma 3.6]{hyy19} and \cite[Theorem 6.1]{ah13},
we know that, for any fixed $\eta'\in(0,\bz\wedge\gz)$, and any $y_\az^{k,m}\in Q_\az^{k,m}$,
$$
\lf|\wz Q_k^*\lf(\frac{\phi_{\az'}^0}{\sqrt{\mu(Q_{\az'}^0)}}\r)\lf(y_\az^{k,m}\r)\r|
\ls\dz^{k\eta'}P_\gz\lf(y_\az^{0,m},x_{\az'}^0;1\r).
$$
By this, we conclude that
$$
|\mathrm{III}_1|\ls\sum_{k=N+1}^\fz\dz^{k\eta'}\sum_{\az\in\CA_k}\sum_{m=1}^{N(k,\az)}\mu\lf(Q_{\az}^{k,m}\r)
P_\gz\lf(y_\az^{k,m},x_{\az'}^0;1\r)\lf|Q_kf\lf(y_\az^{k,m}\r)\r|.
$$

From the estimates of $\RI_1$, $\mathrm{II}_1$, and $\mathrm{III}_1$, and Lemma \ref{lem-pin} when $p\le 1$, or
the H\"older inequality when $p>1$, we deduce that
\begin{align*}
&\sum_{\az'\in\CA_0}\mu\lf(Q_{\az'}^0\r)\lf|\lf<f,\frac{\phi_{\az'}^0}{\sqrt{\mu(Q_{\az'}^0)}}\r>\r|^p\\
&\quad\ls\sum_{\az'\in\CA_0}\mu\lf(Q_{\az'}^0\r)\sum_{k=0}^N\sum_{\az\in\CA_k}\sum_{m=1}^{N(k,\az)}
\lf[\mu\lf(Q_{\az}^k\r)\r]^{p}\lf[P\lf(y_\az^{k,m},x_{\az'}^0;1\r)\r]^p\lf[m_{Q_\az^{0,m}}(|Q_kf|)\r]^p\\
&\qquad\quad+\sum_{\az'\in\CA_0}\mu\lf(Q_{\az'}^0\r)\sum_{k=N+1}^\fz\dz^{kp\eta'}\sum_{\az\in\CA_k}
\sum_{m=1}^{N(k,\az)}\lf[\mu\lf(Q_{\az}^{k,m}\r)\r]^p\lf[P_\gz\lf(y_\az^{k,m},x_{\az'}^0;1\r)\r]^p
\lf|Q_kf\lf(y_\az^{k,m}\r)\r|^p\\
&\quad=:\RJ_1+\RJ_2,
\end{align*}

We next estimate $\RJ_1$. When $p\le 1$, due to
$$
p>p(s,\bz\wedge\gz)>\frac{\omega_0}{\omega_0+(\bz\wedge\gz)}>\frac{\omega_0}{\omega_0+\gz},
$$
by Lemma \ref{lem-bbes}, we conclude that
\begin{equation*}
\RJ_1\ls\sum_{k=0}^N\sum_{m=1}^{N(k,\az)}\lf[\mu\lf(Q_\az^{k,m}\r)\r]^p\lf[m_{Q_\az^{k,m}}(|Q_kf|)\r]^p
\lf[V_{1}\lf(y_\az^{k,m}\r)\r]^{1-p}\ls
\sum_{k=0}^N\sum_{m=1}^{N(k,\az)}\mu\lf(Q_\az^{k,m}\r)\lf[m_{Q_\az^{k,m}}(|Q_kf|)\r]^p.
\end{equation*}
When $p\in(1,\fz]$, the above estimate also holds true directly by \cite[Lemma 2.4(ii)]{hlyy19}; we omit the
details here. This is the desired estimate.

Now, we estimate $\RJ_2$. Using an argument similar to that used in the estimation of $\RJ_1$,
Lemma \ref{lem-bbes} when $q\le p$, or the H\"{o}lder inequality when $q>p$, and $\eta'>-s$, we find that
\begin{align*}
\RJ_2&\ls\sum_{k=N+1}^\fz\dz^{k\eta'p}\sum_{\az\in\CA_k}\sum_{m=1}^{N(k,\az)}\mu\lf(Q_\az^{k,m}\r)
\lf|Q_kf\lf(y_\az^{k,m}\r)\r|^p\\
&\ls\lf\{\sum_{k=N+1}^\fz\dz^{-ksq}\lf[\sum_{\az\in\CA_k}\sum_{m=1}^{N(k,\az)}\mu\lf(Q_\az^{k,m}\r)
\lf|Q_kf\lf(y_\az^{k,m}\r)\r|^p\r]^{q/p}\r\}^{p/q}.
\end{align*}
This is the desired estimate.

Combining the estimates of $\RJ_1$ and $\RJ_2$ with the arbitrariness of $y_\az^{k,m}$, we obtain
\begin{align}\label{eq-k'=0}
&\lf\{\sum_{\az'\in\CA_0}\lf[\mu\lf(Q_{\az'}^0\r)\r]^{1-p/2}\lf|\lf<f,\phi_{\az'}^0\r>\r|^p\r\}^{1/p}\noz\\
&\quad\ls\lf\{\sum_{k=0}^N\sum_{m=1}^{N(k,\az)}\mu\lf(Q_\az^{k,m}\r)
\lf[m_{Q_\az^{k,m}}(|Q_kf|)\r]^p\r\}^{1/p}\noz\\
&\quad\qquad+\lf\{\sum_{k=N+1}^\fz\dz^{-ksq}\lf[\sum_{\az\in\CA_k}\sum_{m=1}^{N(k,\az)}\mu\lf(Q_\az^{k,m}\r)
\inf_{z\in Q_\az^{k,m}}\lf|Q_kf(z)\r|^p\r]^{q/p}\r\}^{1/q}\noz\\
&\quad\sim\|f\|_{B^s_{p,q}(X)},
\end{align}
which is the desired estimate in this case.

{\it Case 2) $k'\in\nn$.} In this case, by Lemma \ref{lem-idrf}, we conclude that,
for any $\az'\in\CG_{k'-1}$,
\begin{align*}
\lf<f,\frac{\psi_{\az'}^{k'-1}}{\sqrt{\mu(Q_{\az'}^{k'})}}\r>
&=\sum_{\az\in\CA_0}\sum_{m=1}^{N(0,\az)}\int_{Q_\az^{0,m}}
\wz Q_0^*\lf(\frac{\psi_{\az'}^{k'-1}}{\sqrt{\mu(Q_{\az'}^0)}}\r)(y)\,d\mu(y)m_{Q_\az^{0,m}}(Q_0f)\\
&\qquad+\sum_{k=1}^N\sum_{\az\in\CA_k}\sum_{m=1}^{N(k,\az)}\mu\lf(Q_\az^{k,m}\r)
\wz Q_k^*\lf(\frac{\psi_{\az'}^{k'-1}}{\sqrt{\mu(Q_{\az'}^{k'})}}\r)\lf(y_\az^{k,m}\r)m_{Q_\az^{k,m}}(Q_kf)\\
&\qquad+\sum_{k=N+1}^\fz\sum_{\az\in\CA_k}\sum_{m=1}^{N(k,\az)}\mu\lf(Q_\az^{k,m}\r)
\wz Q_k^*\lf(\frac{\psi_{\az'}^{k'-1}}{\sqrt{\mu(Q_{\az'}^{k'})}}\r)\lf(y_\az^{k,m}\r)Q_kf\lf(y_\az^{k,m}\r)\\
&=:\RI_2+\mathrm{II}_2+\mathrm{III}_2,
\end{align*}
where, for any $k\in\nn$, $\az\in\CA_k$, and $m\in\{1,\ldots,N(k,\az)\}$, $y_\az^{k,m}$ is
an arbitrary point of $Q_\az^{k,m}$.

For $\RI_2$ and $\mathrm{II}_2$, by \cite[Lemma 3.6]{hyy19} and Lemma \ref{lem-wave}, we
find that, for any $k\in\{0,\ldots,N\}$, $\az\in\CA_k$, $m\in\{1,\ldots,N(k,\az)\}$, and $y\in Q_\az^{k,m}$,
\begin{equation*}
\lf|\wz Q_k^*\lf(\frac{\psi_{\az'}^{k'-1}}{\sqrt{\mu(Q_{\az'}^{k'})}}\r)(y)\r|
\ls\dz^{|k'-k|\eta'}P_\gz\lf(y,y_{\az'}^{k'-1};\dz^k\r)\ls\dz^{k'\eta'}P_\gz\lf(y_\az^{k,m},y_{\az'}^{k'-1};\dz^k\r),
\end{equation*}
where $\eta'\in(0,\bz\wedge\gz)$ is a fixed number. Therefore, we conclude that
\begin{equation}\label{eq-i+ii}
|\RI_2+\mathrm{II}_2|\ls\dz^{k'\eta'}\sum_{k=0}^N\sum_{\az\in\CA_k}\sum_{m=1}^{N(k,\az)}\mu\lf(Q_{\az}^{k,m}\r)
P_\gz\lf(y_\az^{k,m},y_{\az'}^{k'-1};\dz^k\r)m_{Q_\az^{k,m}}(|Q_kf|).
\end{equation}
This is the desired estimate.

To estimate $\mathrm{III}_2$, by \cite[Lemma 3.6]{hyy19}, we have, for any $k\in\{N+1,N+2,\ldots\}$, $\alpha\in\mathcal{A}_k$, and
$y_\az^{k,m}\in Q_\az^{k,m}$,
\begin{equation*}
\lf|\wz Q_k^*\lf(\frac{\psi_{\az'}^{k'-1}}{\sqrt{\mu(Q_{\az'}^{k'})}}\r)\lf(y_\az^{k,m}\r)\r|
\ls\dz^{|k-k'|\eta'}P_\gz\lf(y_\az^{k,m},y_{\az'}^{k'-1};\dz^{k\wedge k'}\r).
\end{equation*}
Then we conclude that
\begin{equation}\label{eq-iii2}
|\mathrm{III}_2|\ls\sum_{k=N+1}^\fz\dz^{|k-k'|\eta'}\sum_{\az\in\CA_k}\sum_{m=1}^{N(k,\az)}
\mu\lf(Q_{\az}^{k,m}\r)P_\gz\lf(y_\az^{k,m},y_{\az'}^{k'-1};\dz^{k\wedge k'}\r)\lf|Q_kf\lf(y_\az^{k,m}\r)\r|,
\end{equation}
which is also the desired estimate.

To show $\|f\|_{B^s_{p,q}(\mathrm w,X)}\ls\|f\|_{B^s_{p,q}(X)}$,
we first consider the case $p\in(p(s,\bz\wedge\gz),1]$.
Indeed, in this case, by \eqref{eq-i+ii}, \eqref{eq-iii2},
and Lemma \ref{lem-pin},
we find that, for any $k'\in\nn$,
\begin{align*}
&\sum_{\az'\in\CG_{k'-1}}\mu\lf(Q_{\az'}^{k'}\r)\lf|\lf<f,\frac{\psi_{\az'}^{k'-1}}
{\sqrt{\mu(Q_{\az'}^{k'})}}\r>\r|^p\\
&\quad\ls\sum_{\az'\in\CG_{k'-1}}\dz^{k'p\eta'}\mu\lf(Q_{\az'}^{k'}\r)\sum_{k=0}^N\sum_{\az\in\CA_k}
\sum_{m=1}^{N(k,\az)}\lf[\mu\lf(Q_{\az}^{k,m}\r)\r]^p\lf[P_\gz\lf(y_\az^{k,m},y_{\az'}^{k'-1};\dz^k\r)\r]^p
\lf[m_{Q_\az^{k,m}}(|Q_kf|)\r]^p\\
&\qquad+\sum_{\az'\in\CG_{k'-1}}\mu\lf(Q_{\az'}^{k'}\r)
\sum_{k=N+1}^\fz\dz^{|k-k'|p\eta'}\sum_{\az\in\CA_k}\sum_{m=1}^{N(k,\az)}
\lf[\mu\lf(Q_{\az}^{k,m}\r)\r]^p\\
&\qquad\times\lf[P_\gz\lf(y_\az^{k,m},y_{\az'}^{k'-1};\dz^{k\wedge k'}\r)\r]^p
\lf|Q_kf\lf(y_\az^{k,m}\r)\r|^p\\
&\quad=:\RJ_3+\RJ_4.
\end{align*}

We next estimate $\RJ_3$. Notice that, for any $k'\in\nn$ and $k\in\{0,\ldots,N\}$, we have
$\dz^k\sim\dz^{k\wedge k'}$. By this, Lemma \ref{lem-bbes}, and \eqref{eq-doub}, we conclude that, for any $k'\in\nn$,
\begin{equation*}
\RJ_3\ls\dz^{k'p\eta'}\sum_{k=0}^N\sum_{\az\in\CA_k}\sum_{m=1}^{N(k,\az)}
\mu\lf(Q_\az^{k,m}\r)\lf[m_{Q_\az^{k,m}}(|Q_kf|)\r]^p.
\end{equation*}

Now, for $\RJ_4$, by Lemma \ref{lem-bbes} again, we find that
\begin{equation*}
\RJ_4\ls\sum_{k=N+1}^\fz \dz^{|k-k'|\{p\eta'-\omega(1-p)[(k-k\wedge k')]\}}\sum_{\az\in\CA_k}
\sum_{m=1}^{N(k,\az)}\mu\lf(Q_{\az}^{k,m}\r)\lf|Q_kf\lf(y_\az^{k,m}\r)\r|^p,
\end{equation*}
where $\omega\in[\omega_0,\fz)$ satisfy \eqref{eq-doub} and all the assumptions of this theorem with
$\omega_0$ replaced by $\omega$.

Therefore, using the estimates of $\RJ_3$ and $\RJ_4$, and choosing $\eta'>-s$, we obtain
\begin{align}\label{eq-k'>0a}
&\sum_{k'\in\nn}\dz^{-k'sq}\lf\{\sum_{\az'\in\CG_{k'-1}}\lf[\mu\lf(Q_{\az'}^{k'}\r)\r]^{1-p/2}
\lf|\lf<f,\psi_{\az'}^{k'-1}\r>\r|^p\r\}^{q/p}\noz\\
&\quad\ls\lf\{\sum_{k=0}^N\sum_{\az\in\CA_k}\sum_{m=1}^{N(k,\az)}
\mu\lf(Q_\az^{k,m}\r)\lf[m_{Q_\az^{k,m}}(|Q_kf|)\r]^p\r\}^{q/p}\noz\\
&\quad\qquad+\sum_{k'\in\nn}\lf[\sum_{k=N+1}^\fz\dz^{-ksp}\dz^{(k-k')sp}
\dz^{|k-k'|\{p\eta'-\omega(1-p)[k-(k\wedge k')]\}}\r.\noz\\
&\quad\qquad\lf.\times\sum_{\az\in\CA_k}\sum_{m=1}^{N(k,\az)}\mu\lf(Q_{\az}^{k,m}\r)
\lf|Q_kf\lf(y_\az^{k,m}\r)\r|^p\r]^{q/p}\noz\\
&\quad=:\lf\{\sum_{k=0}^N\sum_{\az\in\CA_k}\sum_{m=1}^{N(k,\az)}
\mu\lf(Q_\az^{k,m}\r)\lf[m_{Q_\az^{k,m}}(|Q_kf|)\r]^p\r\}^{q/p}+\mathrm R.
\end{align}

Next, we estimate $\mathrm R$. Indeed, when $k'\ge k$, we conclude that
\begin{align*}
(k-k')sp+|k-k'|\{p\eta'-\omega(1-p)[k-(k\wedge k')]\}=(k'-k)p(\eta'-s),
\end{align*}
while, when $k'<k$, we obtain
\begin{align*}
(k-k')sp+|k-k'|\{p\eta'-\omega(1-p)[k-(k\wedge k')]\}=(k'-k)p\lf[\eta'+s-\frac{\omega(1-p)}p\r].
\end{align*}
Due to $p>p(s,\bz\wedge\gz)$, choosing $\eta'\in(0,\bz\wedge\gz)$ such that $\eta'>|s|$ and $p>\omega/(\omega+\eta'+s)$,
and using an argument similar to that used in the estimation of \eqref{eq-sum}, we find that
$$
\mathrm R\ls\sum_{k=N+1}^\fz\dz^{-ksq}\lf\{\sum_{\az\in\CA_k}\sum_{m=1}^{N(k,\az)}
\mu\lf(Q_\az^{k,m}\r)\lf|Q_kf\lf(y_\az^{k,m}\r)\r|^p\r\}^q.
$$
By this, \eqref{eq-k'>0a}, and \eqref{eq-k'=0}, we obtain
\begin{align*}
&\lf\{\sum_{\az'\in\CA_0}\lf[\mu\lf(Q_{\az'}^0\r)\r]^{1-p/2}\lf|\lf<f,\phi_{\az'}^0\r>\r|^p\r\}^{1/p}
+\lf[\sum_{k'\in\nn}\dz^{-k'sq}\lf\{\sum_{\az'\in\CG_{k'-1}}\lf[\mu\lf(Q_{\az'}^{k'}\r)\r]^{1-p/2}
\lf|\lf<f,\psi_{\az'}^{k'-1}\r>\r|^p\r\}^{q/p}\r]^{1/q}\\
&\quad\ls\|f\|_{B^s_{p,q}(X)},
\end{align*}
which further implies that $\|f\|_{B^s_{p,q}(\mathrm w,X)}\ls\|f\|_{B^s_{p,q}(X)}$.
This finishes the proof of the case $p\in(p(s,\bz\wedge\gz),1]$.

Now, we consider the case $p\in(1,\fz]$.
In this case, by \eqref{eq-i+ii} and \eqref{eq-iii2}, we observe that
\begin{align*}
&\sum_{\az'\in\CG_{k'-1}}\mu\lf(Q_{\az'}^{k'}\r)\lf|\lf<f,\frac{\psi_{\az'}^{k'-1}}
{\sqrt{\mu(Q_{\az'}^{k'})}}\r>\r|^p\\
&\quad\ls\sum_{\az'\in\CG_{k'-1}}\dz^{k'p\eta'}\mu\lf(Q_{\az'}^{k'}\r)\sum_{k=0}^N\lf\{\sum_{\az\in\CA_k}
\sum_{m=1}^{N(k,\az)}\mu\lf(Q_{\az}^{k,m}\r)P_\gz\lf(y_\az^{k,m},y_{\az'}^{k'-1};\dz^k\r)m_{Q_\az^{k,m}}(|Q_kf|)\r\}^p\\
&\quad\qquad+\sum_{\az'\in\CG_{k'-1}}\mu\lf(Q_{\az'}^{k'}\r)
\lf\{\sum_{k=N+1}^\fz\dz^{|k-k'|\eta'}\sum_{\az\in\CA_k}\sum_{m=1}^{N(k,\az)}\mu\lf(Q_{\az}^{k,m}\r)
P_\gz\lf(y_\az^{k,m},y_{\az'}^{k'-1};\dz^{k\wedge k'}\r)\lf|Q_kf\lf(y_\az^{k,m}\r)\r|\r\}^p\\
&\quad=:\RJ_5+\RJ_6.
\end{align*}

We first estimate $\RJ_5$. From the H\"older inequality and the Tonelli theorem, we  deduce that
\begin{equation*}
\RJ_5\ls\dz^{k'p\eta'}\sum_{k=0}^N\sum_{\az\in\CA_k}\sum_{m=1}^{N(k,\az)}\mu\lf(Q_{\az}^{k,m}\r)
\lf[m_{Q_\az^{k,m}}(|Q_kf|)\r]^p,
\end{equation*}
which is the desired estimate.

We then estimate $\RJ_6$. By the H\"older inequality and the Tonelli theorem, we find that
\begin{align*}
\RJ_6\ls\sum_{k=N+1}^\fz\dz^{|k-k'|p\eta'}\sum_{\az\in\CA_k}\sum_{m=1}^{N(k,\az)}\mu\lf(Q_{\az}^{k,m}\r)
\lf|Q_kf\lf(y_\az^{k,m}\r)\r|^p.
\end{align*}

By the estimates of $\RJ_5$ and $\RJ_6$, using Lemma \ref{lem-pin} when $q\le p$, or the H\"{o}lder inequality
when $q>p$, and choosing $\eta'\in(0,\bz\wedge\gz)$ such that $\eta'>|s|$, we conclude that
\begin{align*}
&\sum_{k'\in\nn}\dz^{-k'sq}\lf\{\sum_{\az'\in\CG_{k'-1}}\lf[\mu\lf(Q_{\az'}^{k'}\r)\r]^{1-p/2}
\lf|\lf<f,\psi_{\az'}^{k'-1}\r>\r|^p\r\}^{q/p}\\
&\quad\ls\lf\{\sum_{k=0}^N\sum_{\az\in\CA_k}\sum_{m=1}^{N(k,\az)}
\mu\lf(Q_\az^{k,m}\r)\lf[m_{Q_\az^{k,m}}(|Q_kf|)\r]^p\r\}^{q/p}\\
&\qquad\quad+\sum_{k=N+1}^\fz\dz^{-ksq}\lf\{\sum_{\az\in\CA_k}\sum_{m=1}^{N(k,\az)}\mu\lf(Q_\az^{k,m}\r)
\lf|Q_kf\lf(y_\az^{k,m}\r)\r|^p\r\}^{q/p}\sum_{k'\in\nn}\dz^{(k-k')sq}\dz^{|k-k'|q\eta'}\\
&\quad \ls\|f\|_{B^s_{p,q}(X)}^q.
\end{align*}
Combining the above arguments, we obtain
$\|f\|_{B^s_{p,q}(\mathrm w,X)}\ls\|f\|_{B^s_{p,q}(X)}$. This finishes the proof of the sufficiency
of (i).

Next, we prove the necessity of (i). To this end, assume $f\in(\go{\bz,\gz})'$
and $\|f\|_{B^s_{p,q}(\mathrm{w},X)}<\fz$.
By Theorem \ref{thm-eieti}, \eqref{eq-iwave}, and the orthogonality of $\{\phi_\az^0\}_{\az\in\CA_0}\cup
\{\psi_\az^k\}_{k\in\zz_+,\ \az\in\CG_k}$ (see, for instance, \cite[(5.15)]{hyy19}), we find that, for any $k\in\zz_+$ and
$z\in X$,
\begin{align*}
Q_kf(z)&=\sum_{\az'\in\CA_0}\lf<f,\phi_{\az'}^0\r>\lf<Q_k(z,\cdot),\phi_{\az'}^0\r>
+\sum_{k'\in\nn}\sum_{\az'\in\CG_{k'-1}}\lf<f,\psi_{\az'}^{k'-1}\r>
\lf<Q_k(x,\cdot),\psi_{\az'}^{k'-1}\r>\\
&=\begin{cases}
\displaystyle \sum_{\az\in\CA_0}\lf<f,\phi_\az^0\r>\phi_\az^0(z) & \textup{if } k=0, \\
\displaystyle \sum_{\az\in\CG_{k-1}}\lf<f,\psi_\az^{k-1}\r>\psi_\az^{k-1}(z) & \textup{if } k\in\nn.
\end{cases}
\end{align*}
Now, we consider the following three cases on $k\in\zz_+$.

{\it Case 1) $k=0$.} In this case, by \cite[Theorem 6.1]{ah13}, we conclude that, for any $\az\in\CA_0$,
$m\in\{1,\ldots,N(0,\az)\}$, and $z\in X$,
\begin{align*}
|Q_0f(z)|&\ls\sum_{\az'\in\CA_0}\mu\lf(Q_{\az'}^0\r)\lf|\lf<f,\frac{\phi_{\az'}^0}{\sqrt{\mu(Q_{\az'}^0)}}\r>\r|
\frac 1{V_1(x_{\az'}^0)+V(x_{\az'}^0,z_\az^{0,m})}\exp\lf\{-\nu'\lf[d\lf(z_\az^{0,m},x_{\az'}^0\r)\r]^a\r\},
\end{align*}
which further implies that
\begin{align*}
\lf|m_{Q_\az^{0,m}}(|Q_kf|)\r|
\ls\sum_{\az'\in\CA_0}\mu\lf(Q_{\az'}^0\r)\lf|\lf<f,\frac{\phi_{\az'}^0}{\sqrt{\mu(Q_{\az}^0)}}\r>\r|
\frac 1{V_1(x_{\az'}^0)+V(x_{\az'}^0,z_\az^{0,m})}\exp\lf\{-\nu'\lf[d\lf(z_\az^{0,m},x_{\az'}^0\r)\r]^a\r\},
\end{align*}
where, for any $\az\in\CA_0$ and $m\in\{1,\ldots,N(k,\az)\}$, $z_\az^{0,m}$ denotes the ``center'' of $Q_\az^{0,m}$,
and $\nu'\in(0,\nu)$ is independent of $\az$, $m$, $k$, and $f$.
By this and Lemma \ref{lem-pin} when $p\in(0,1]$, or the H\"older inequality when $p\in(1,\fz]$, we find that
\begin{equation}\label{eq-bw1}
\sum_{\az\in\CA_0}\sum_{m=1}^{N(0,\az)}\mu\lf(Q_\az^{k,m}\r)\lf[m_{Q_\az^{0,m}}(|Q_kf|)\r]^p
\ls\sum_{\az'\in\CA_0}\lf[\mu\lf(Q_{\az'}^0\r)\r]^{1-p/2}\lf|\lf<f,\phi_{\az'}^0\r>\r|^p,
\end{equation}
which is the desired estimate.

{\it Case 2) $k\in\{1,\ldots,N\}$}. In this case, using an argument similar to that used in Case 1), we obtain
\begin{equation}\label{eq-bw2}
\sum_{k=1}^N\sum_{\az\in\CA_k}\sum_{m=1}^{N(k,\az)}\mu\lf(Q_\az^{k,m}\r)\lf[m_{Q_\az^{k,m}}(|Q_kf|)\r]^p
\ls\sum_{k=1}^N\sum_{\az'\in\CG_{k-1}}\lf[\mu\lf(Q_{\az'}^k\r)\r]^{1-p/2}
\lf|\lf<f,\psi_{\az'}^{k-1}\r>\r|^p,
\end{equation}
which is also the desired estimate.

{\it Case 3) $k\in\{N+1,N+2,\ldots\}$.} In this case, by the size condition of $\psi_{\az'}^k$
(see, for instance, Lemma \ref{lem-idrf}), we find that, for any $\az\in\CA_k$,
$m\in\{1,\ldots,N(k,\az)\}$, and $z\in Q_\az^{k,m}$,
\begin{equation*}
|Q_kf(z)|\ls\sum_{\az'\in\CG_{k-1}}\mu\lf(Q_{\az'}^k\r)
\lf|\lf<f,\frac{\psi_{\az'}^{k-1}}{\sqrt{\mu(Q_{\az'}^k)}}\r>\r|
\frac 1{V_{\dz^k}(y_{\az'}^{k-1})+V(y_{\az'}^{k-1},z_\az^{k,m})}
\exp\lf\{-\nu'\lf[\frac{d(z_\az^{k,m},y_{\az'}^{k-1})}{\dz^k}\r]^a\r\}.
\end{equation*}
By the arbitrariness of $z$, we further conclude that, for any $\az\in\CA_k$ and $m\in\{1,\ldots,N(k,\az)\}$,
\begin{align*}
&\sup_{z\in Q_\az^{k,m}}|Q_kf(z)|\\
&\quad\ls\sum_{\az'\in\CG_{k-1}}\mu\lf(Q_{\az'}^k\r)
\lf|\lf<f,\frac{\psi_{\az'}^{k-1}}{\sqrt{\mu(Q_{\az'}^k)}}\r>\r|
\frac 1{V_{\dz^k}(y_{\az'}^{k-1})+V(y_{\az'}^{k-1},z_\az^{k,m})}
\exp\lf\{-\nu'\lf[\frac{d(z_\az^{k,m},y_{\az'}^{k-1})}{\dz^k}\r]^a\r\},
\end{align*}
where, for any $\az\in\CA_k$ and $m\in\{1,\ldots,N(k,\az)\}$, $z_\az^{k,m}$ denotes the ``center'' of $Q_\az^{k,m}$.
From this and Lemmas \ref{lem-pin} and \ref{lem-bbes} when $p\in(0,1]$, or the H\"older inequality when $p\in(1,\fz]$,
we deduce that
\begin{align}\label{eq-bw3}
\|Q_kf\|_{L^p(X)}^p&\ls\sum_{\az\in\CA_k}\sum_{m=1}^{N(k,\az)}\mu\lf(Q_\az^{k,m}\r)\sum_{\az'\in\CG_{k-1}}
\lf[\mu\lf(Q_{\az'}^0\r)\r]^p\lf|\lf<f,\frac{\psi_{\az'}^{k-1}}{\sqrt{\mu(Q_{\az'}^k)}}\r>\r|^p\noz\\
&\qquad\times\lf[\frac 1{V_{\dz^k}(y_{\az'}^{k-1})+V(y_{\az'}^{k-1},z_\az^{k,m})}\r]^p
\exp\lf\{-p\nu'\lf[\frac{d(z_\az^{k,m},y_{\az'}^{k-1})}{\dz^k}\r]^a\r\}\noz\\
&\ls\sum_{\az'\in\CG_{k-1}}\lf[\mu\lf(Q_{\az'}^k\r)\r]^{1-p/2}\lf|\lf<f,\psi_{\az'}^k\r>\r|^p,
\end{align}
which is the desired estimate.

Combining \eqref{eq-bw1}, \eqref{eq-bw2}, and \eqref{eq-bw3}, we conclude that
\begin{align*}
\|f\|_{B^s_{p,q}(X)}&\ls\lf\{\sum_{\az'\in\CA_0}\lf[\mu\lf(Q_{\az'}^0\r)\r]^{1-p/2}
\lf|\lf<f,\phi_{\az'}^0\r>\r|^p\r\}^{1/p}\\
&\qquad+\lf[\sum_{k=1}^N\dz^{-ksq}\lf\{\sum_{\az'\in\CG_{k-1}}
\lf[\mu\lf(Q_{\az'}^k\r)\r]^{1-p/2}\lf|\lf<f,\psi_{\az'}^{k-1}\r>\r|^p\r\}^{q/p}\r]^{1/q}
+\lf[\sum_{k=N+1}^\fz\cdots\vphantom{\dz^{-ksq}\lf\{\sum_{\az'\in\CG_{k-1}}
\lf[\mu\lf(Q_{\az'}^k\r)\r]^{1-p/2}\lf|\lf<f,\psi_{\az'}^{k-1}\r>\r|^p\r\}^{q/p}}\r]^{1/q}\\
&\sim\|f\|_{B^s_{p,q}(\mathrm w,X)}.
\end{align*}
This finishes the proof of the necessity of (i), and hence of Theorem \ref{thm-w=q}.
\end{proof}

\subsection{Almost diagonal operators, and molecular and Littlewood--Paley characterizations of
inhomogeneous Besov and Triebel--Lizorkin spaces}

In this subsection, we state some results similar to those of homogeneous Besov and Triebel--Lizorkin spaces
in Sections \ref{s-ado}, \ref{s-mol} and \ref{s-lp}. Since their proofs are similar, respectively, to those
of the homogeneous case, we only list the corresponding conclusions here and
omit the details. We first introduce the inhomogeneous Besov and Triebel--Lizorkin sequence spaces. For any
$k\in\zz_+$, let
$$
\CH_k:=\begin{cases}
\CA_0 & \hbox{if $k=0$,} \\
\CG_{k-1}:=\CA_{k}\setminus\CA_{k-1} & \hbox{if $k\in\nn$},
\end{cases}
$$
and $\wz\CD_0:=\{Q_\az^k:\ k\in\zz_+,\ \az\in\CH_k\}$.

\begin{definition}
Let $s\in\rr$ and $q\in(0,\fz]$.
\begin{enumerate}
\item If $p\in(0,\fz]$, then the \emph{inhomogeneous  Besov sequence space $b^s_{p,q}$} is defined to be
set of all $\lz:=\{\lz_Q\}_{Q\in\wz\CD_0}=:\{\lz_\az^k\}_{k\in\zz_+,\ \az\in\CH_k}\subset\cc$ such that
\begin{align*}
\|\lz\|_{b^s_{p,q}}:=&\,\lf\{\sum_{\az\in\CA_0}\lf[\mu\lf(Q_\az^0\r)\r]^{1-p/2}\lf|\lz_\az^0\r|^p\r\}^{1/p}
+\lf[\sum_{k\in\nn}\dz^{-ksq}\lf\{\sum_{\az\in\CG_{k-1}}\lf[\mu\lf(Q_\az^{k}\r)\r]^{1-p/2}
\lf|\lz_\az^k\r|^p\r\}^{q/p}\r]^{1/q}\\
<&\,\fz
\end{align*}
with usual modifications made when $p=\fz$ or $q=\fz$.

\item If $p\in(0,\fz)$, then the \emph{inhomogeneous  Triebel--Lizorkin sequence space
$f^s_{p,q}$} is defined to be the set of all $\lz:=\{\lz_Q\}_{Q\in\wz\CD_0}
=:\{\lz_\az^k\}_{k\in\zz_+,\ \az\in\CH_k}\subset\cc$ such that
$$
\|\lz\|_{f^s_{p,q}}:=\lf\{\sum_{\az\in\CA_0}\lf[\mu\lf(Q_\az^0\r)\r]^{1-p/2}\lf|\lz_\az^0\r|^p\r\}^{1/p}
+\lf\|\lf(\sum_{k\in\nn}\dz^{-ksq}\lf|\lz_\az^k\widetilde{\mathbf 1}_{Q_\az^{k}}\r|^q\r)^{1/q}\r\|_{L^p(X)}
<\fz
$$
with usual modification made when $q=\fz$.
\end{enumerate}
\end{definition}

Now, we introduce the notion of inhomogeneous almost diagonal operators. Similarly, for any dyadic cube $Q$,
we denote by $x_Q$ the ``center'' of $Q$ and by $\ell(Q)$ the ``side-length'' of $Q$.
Let $A:=\{A_{Q,P}\}_{Q,\ P\in\wD_0}\subset\cc$. For any sequence $\lz:=\{\lz_P\}_{P\in\wD_0}$, define
$A\lz:=\{(A\lz)_Q\}_{Q\in\wD_0}$ by setting, for any $Q\in\wD_0$,
$$
(A\lz)_Q:=\sum_{P\in\wD_0} A_{Q,P}\lz_P
$$
if, for any $Q\in\wD_0$, the above summation converges.

\begin{definition}\label{def-iado}
Let $A:=\{A_{Q,P}\}_{Q,\ P\in\wD_0}\subset\cc$ and $\omega_0$ be as in \eqref{eq-updim}.
\begin{enumerate}
\item Let $s\in\rr$ and $p,\ q\in(0,\fz]$.
The operator $A$ is called an \emph{inhomogeneous almost diagonal operator on $b^s_{p,q}$}
if there exist an $\ez\in(0,\fz)$ and an $\omega\in[\omega_0,\fz)$ satisfying \eqref{eq-doub} such that
\begin{equation}\label{eq-defiado}
K:=\sup_{Q,\ P\in\wD_0} \frac{|A_{Q,P}|}{\mathfrak M_{Q,P}(\ez)}<\fz,
\end{equation}
where, for any $Q,\ P\in\wD$, $\mathfrak M_{Q,P}(\ez)$ is defined as in \eqref{eq-defopq} with
$J:=\omega/\min\{1,p\}$.

\item Let $s\in\rr$,
$p\in(0,\fz)$, and $q\in(0,\fz]$.
The operator $A$ is called an \emph{inhomogeneous almost diagonal operator on $f^s_{p,q}$}
if there exist an $\ez\in(0,\fz)$ and an $\omega\in[\omega_0,\fz)$ satisfying \eqref{eq-doub}
such that \eqref{eq-defiado} holds true,
where, for any $Q,\ P\in\wD_0$, $\mathfrak M_{Q,P}(\ez)$ is as in \eqref{eq-defopq} with
$J:=\omega/\min\{1,p,q\}$.
\end{enumerate}
\end{definition}

Using an argument similar to that used in the proof of Theorem \ref{thm-ado}, we have the following boundedness
of inhomogeneous almost diagonal operators on $b^s_{p,q}$ and $f^s_{p,q}$; we omit the details here.
\begin{theorem}\label{thm-iado}
Let $s\in\rr$, $p\in(0,\fz]$ [resp., $p\in(0,\fz)$], $q\in(0,\fz]$, and $A:=\{A_{Q,P}\}_{Q,\ P\in\wD_0}$ be an
inhomogeneous almost diagonal operator on $b^s_{p,q}$ (resp., $f^s_{p,q}$). Then $A$ is bounded on $b^s_{p,q}$
(resp., $f^s_{p,q}$). Moreover, there exists a positive constant $C$,
independent of $A$, such that, for any $\lz\in b^s_{p,q}$
(resp., $\lz\in f^s_{p,q}$), $\|A\lz\|_{b^s_{p,q}}\le CK\|\lz\|_{b^s_{p,q}}$
(resp., $\|A\lz\|_{f^s_{p,q}}\le CK\|\lz\|_{f^s_{p,q}}$).
\end{theorem}

Next, we state the molecular characterization of inhomogeneous Besov and Triebel--Lizorkin spaces. To
distinguish from the notion of molecules in Definition \ref{def-mol}, we first introduce the following notion
of local molecules.

\begin{definition}\label{def-imol}
Let $(\bz,\Gamma)\in(0,\fz)^2$ and $Q\in\wz{\mathcal D}_0$. A function $b_Q$ is called a
\emph{local molecule of type $(\bz,\Gamma)$} [for short, \emph{local $(\bz,\Gamma)$-molecule}]
centered at $Q$ if $b_Q$ satisfies the following conditions:
\begin{enumerate}
\item (the \emph{size condition}) for any $x\in X$, $|b_Q(x)|\le[\mu(Q)]^{1/2} P_{\Gamma}(y_Q,x;\ell(Q))$;
\item (the \emph{H\"older regularity condition}) for any $x,\ x'\in X$ with
$d(x,x')\le(2A_0)^{-1}[\ell(Q)+d(y_Q,x)]$,
$$
|b_Q(x)-b_Q(x')|\le[\mu(Q)]^{1/2}\lf[\frac{d(x,x')}{\ell(Q)+d(y_Q,x)}\r]^\bz
P_\Gamma(y_Q,x;\ell(Q));
$$
\item (the \emph{cancellation  condition}) $\int_X b_Q(x)\,d\mu(x)=0$ if $\ell(Q)\in(0,1)$.
\end{enumerate}
\end{definition}

We have the molecular characterization of inhomogeneous Besov and Triebel--Lizorkin
spaces as follows.

\begin{theorem}\label{thm-imol}
Let $\eta$ be the same as in Definition \ref{def-eti}, and $s$, $p$, $q$, $\bz$, and $\gz$ the same as in
Definition \ref{def-ibtl}(i) [resp., Definition \ref{def-ibtl}(ii)].
\begin{enumerate}
\item If $f\in(\go{\bz,\gz})'$, then there exist local $(\bz,\gz)$-molecules $\{b_Q\}_{Q\in\wD_0}$
 centered, respectively, at $\{Q\}_{Q\in\wD_0}$, and $\lz:=\{\lz_Q\}_{Q\in\wD_0}\in b^s_{p,q}$
(resp., $\lz:=\{\lz_Q\}_{Q\in\wD_0}\in f^s_{p,q}$)
such that $f=\sum_{Q\in\wD_0}\lz_Qb_Q$ in $(\go{\bz,\gz})'$ and $\|\lz\|_{b^s_{p,q}}\le C\|f\|_{B^s_{p,q}(X)}$
[resp., $\|\lz\|_{f^s_{p,q}}\le C\|f\|_{F^s_{p,q}(X)}$], where $C$ is a positive constant
independent of $f$.

\item Conversely, if $\{b_Q\}_{Q\in\wD_0}$ is a sequence of local $(\bz,\gz)$-molecules
centered, respectively, at $\{Q\}_{Q\in\wD_0}$, and $\lz:=\{\lz_Q\}_{Q\in\wD_0}\in b^s_{p,q}$
(resp., $\lz:=\{\lz_Q\}_{Q\in\wD_0}\in f^s_{p,q}$), then there
exists a $g\in(\go{\bz,\gz})'$ such that $g=\lz_Qb_Q$ in $(\go{\bz,\gz})'$,
$g\in B^s_{p,q}(X)$ [resp., $g\in F^s_{p,q}(X)$], and
$\|g\|_{B^s_{p,q}(X)}\le C\|\lz\|_{b^s_{p,q}}$
[resp., $\|g\|_{F^s_{p,q}(X)}\le C\|\lz\|_{f^s_{p,q}}$], where $C$ is a positive constant
independent of $\{b_Q\}_{Q\in\wD_0}$ and $\lz$.
\end{enumerate}
\end{theorem}

\begin{remark}
If we replace $\ell(Q)\in(0,1)$ in Definition \ref{def-imol}(iii) by $\ell(Q)\in(0,R)$ for any fixed
$R\in(0,1]$, then Theorem \ref{thm-imol} also holds true in this case; we omit the details here.
\end{remark}

Finally, we concern about the Lusin area function and the Littlewood--Paley $g_\lz^*$-function characterizations
of inhomogeneous Triebel--Lizorkin spaces. To this end, we first introduce the following inhomogeneous
Littlewood--Paley functions.

\begin{definition}\label{def-inlp}
Let $s\in(-\eta,\eta)$ with $\eta$ as in Definition \ref{def-eti}, $q\in(0,\fz]$,
$\bz,\ \gz\in(0,\eta)$, and $f\in(\go{\bz,\gz})'$. The \emph{inhomogeneous Littlewood--Paley
$g$-function $g^s_q(f)$ of $f$} is defined by
setting, for any $x\in X$,
$$
g^s_q(f)(x):=\lf\{\sum_{k=0}^N\sum_{\az\in\CA_k}\sum_{m=1}^{N(k,\az)}m_{Q_\az^{k,m}}(|Q_kf|^q)
\mathbf 1_{Q_\az^{k,m}}(x)+\sum_{k=N+1}^\fz\dz^{-ksq}|Q_kf(x)|^q\r\}^{1/q},
$$
where $N$ is as in Lemma \ref{lem-idrf} and $m_Q$, for any given dyadic cube $Q$, as in \eqref{eq-defmean},
the \emph{inhomogeneous Lusin area function $\CS^s_q(f)$ of $f$} is defined by setting, for any $x\in X$,
$$
\CS^s_q(f)(x):=\lf[\sum_{k=0}^\fz\dz^{-ksq}\frac 1{V_{\dz^k}(x)}\int_{B(x,\dz^k)}|Q_kf(y)|^q
\,d\mu(y)\r]^{1/q},
$$
and, for any given $\lz\in(0,\fz)$, the \emph{inhomogeneous Littlewood--Paley $g_\lz^*$-function $(g_\lz^*)^s_q(f)$
of $f$} is defined by setting, for any $x\in X$,
$$
\lf(g_\lz^*\r)^s_q(f)(x):=\lf\{\sum_{k=0}^\fz\dz^{-ksq}\int_X |Q_kf(y)|^q\lf[\frac{\dz^k}{\dz^k+d(x,y)}\r]^\lz
\,\frac{d\mu(y)}{V_{\dz^k}(x)+V_{\dz^k}(y)}\r\}^{1/q}.
$$
\end{definition}

Obviously, for any given $s$, $p$, $q$, $\bz$, and $\gz$ as in Definition \ref{def-ibtl}(ii),
and $f\in(\go{\bz,\gz})'$, it holds true that $\|g^s_q(f)\|_{L^p(X)}=\|f\|_{F^s_{p,q}(X)}$. The next two
theorems give the Lusin area function and the Littlewood--Paley $g_\lz^*$-function characterizations of
inhomogeneous Triebel--Lizorkin spaces, respectively. Since their proofs are similar, respectively,
to those of Theorems \ref{thm-g=s} and \ref{thm-glstar}, we omit the details here.

\begin{theorem}\label{thm-ig=s}
Let $s$, $p$, $q$, $\bz$, and $\gz$ be as in Definition \ref{def-ibtl}(ii). Then $f\in F^s_{p,q}(X)$ if and
only if $f\in(\go{\bz,\gz})'$ and $\CS^s_q(f)\in L^p(X)$. Moreover, there exists a constant $C\in[1,\fz)$,
independent of $f$, such that $C^{-1}\|\CS^s_q(f)\|_{L^p(X)}\le\|f\|_{F^s_{p,q}(X)}\le C\|\CS^s_q(f)\|_{L^p(X)}$.
\end{theorem}

\begin{theorem}\label{thm-iglstar}
Let $s$, $p$, $q$, $\bz$, and $\gz$ be as in Definition \ref{def-ibtl}(ii), and $q\in(p(s,\bz\wedge\gz),\fz)$. Suppose
$\lz\in(\max\{\omega_0,q\omega_0/p\},\fz)$ with $\omega_0$ as in \eqref{eq-updim}.
Then $f\in F^s_{p,q}(X)$ if and only if $f\in(\go{\bz,\gz})'$ and
$(g_\lz^*)^s_q(f)\in L^p(X)$. Moreover, there exists a constant $C\in[1,\fz)$,
independent of $f$, such that
$C^{-1}\|(g_\lz^*)^s_q(f)\|_{L^p(X)}\le \|f\|_{F^s_{p,q}(X)}
\le C\|(g_\lz^*)^s_q(f)\|_{L^p(X)}$.
\end{theorem}

\paragraph{Acknowledgments.} The authors would like to thank the referees for their
many valuable suggestions which indeed improve the quality of this article.

\bigskip

\noindent Ziyi He, Fan Wang, Dachun Yang (Corresponding author) and Wen Yuan

\medskip

\noindent Laboratory of Mathematics and Complex Systems (Ministry of Education of China),
School of Mathematical Sciences, Beijing Normal University, Beijing 100875, People's Republic of China

\smallskip

\noindent{\it E-mails:} \texttt{ziyihe@mail.bnu.edu.cn} (Z. He)

\noindent\phantom{{\it E-mails:} }\texttt{fanwang@mail.bnu.edu.cn} (F. Wang)

\noindent\phantom{{\it E-mails:} }\texttt{dcyang@bnu.edu.cn} (D. Yang)

\noindent\phantom{{\it E-mails:} }\texttt{wenyuan@bnu.edu.cn} (W. Yuan)

\end{document}